\documentclass[11pt,reqno]{amsart}

\usepackage{amsmath,amsthm,amssymb,amscd,graphicx}
\usepackage{mathrsfs}

\usepackage{datetime}
\usepackage{hyperref}

\newtheorem{thm}{Theorem}[section]
 \newtheorem{cor}[thm]{Corollary}
 \newtheorem{lem}[thm]{Lemma}
 \newtheorem{claim}[thm]{Claim}
 \newtheorem{prop}[thm]{Proposition}
 
 \theoremstyle{definition}
 
 \theoremstyle{remark}
 \newtheorem{rem}[thm]{Remark}
 \newtheorem{ex}{Example}
 \numberwithin{equation}{section}

\def\be#1 {\begin{equation} \label{#1}}
\newcommand{\ee}{\end{equation}}

\def\Z{{\mathbb Z}}

\def\R{{\mathbb R}}

\def\e{{\varepsilon}}
\def\g{{\gamma}}

\def\b{{\beta}}
\def\d{{\delta}}

\def\sign{ \mbox{sign} }

\def\K{{\mathcal{K}}}

\def\+R{+_{_{ \!\! \R}}}

\def\what{\widehat}
\def\wt{\widetilde}
\def\bar{\overline}
\def\eps{\epsilon}
\def\veps{\varepsilon}
\def\S{Schr\"{o}dinger }

\DeclareMathAlphabet{\mathpzc}{OT1}{pzc}{m}{it}

\numberwithin{equation}{section}

\makeatletter
\newcommand{\pushright}[1]{\ifmeasuring@#1\else\omit\hfill$\displaystyle#1$\fi\ignorespaces}
\newcommand{\pushleft}[1]{\ifmeasuring@#1\else\omit$\displaystyle#1$\hfill\fi\ignorespaces}
\makeatother

\begin{document}

\author{Pierre Germain}
\address{Pierre Germain, Courant Institute of Mathematical Sciences, 251 Mercer Street, New York 10012-1185 NY, USA}
\email{pgermain@cims.nyu.edu}

\author{Fabio Pusateri}
\address{Fabio Pusateri,  Department of Mathematics, Princeton University, Washington Road, Princeton 08540 NJ, USA}
\email{fabiop@math.princeton.edu}

\author{Fr\'ed\'eric Rousset}
\address{Fr\'ed\'eric Rousset, Laboratoire de Math\'ematiques d'Orsay (UMR 8628),
Universit\'e Paris-Sud, 91405 Orsay Cedex France et Institut Universitaire de France}
  \email{frederic.rousset@math.u-psud.fr}

\title{The Nonlinear Schr\"odinger equation with a potential}

\subjclass[2000]{Primary 35Q55 ; 35B34}

\date{\today}

\keywords{Nonlinear Schr\"odinger Equation, Distorted Fourier Transform, Scattering Theory, Modified Scattering}

\begin{abstract}
We consider the cubic nonlinear Schr\"odinger equation with a potential in one space dimension.
Under the assumptions that the potential is generic, sufficiently 
localized, with no bound states,
we obtain the long-time asymptotic behavior of small solutions.
In particular, we prove that, as time goes to infinity, solutions exhibit
nonlinear phase corrections that depend on the scattering matrix associated to the potential.
The proof of our result is based on the use of the distorted Fourier transform - the so-called Weyl-Kodaira-Titchmarsh theory - 
a precise understanding of the ``nonlinear spectral measure'' associated to the equation,
and nonlinear stationary phase arguments and multilinear estimates in this distorted setting.
\end{abstract}

\maketitle

\setcounter{tocdepth}{1}

\begin{quote}
\tableofcontents
\end{quote}

\bigskip
\section{Introduction}

\medskip
\subsection{The equation}
Our aim in this paper is to describe the large time behavior of small solutions of the Cauchy problem for
the one dimensional  cubic nonlinear Schr\"odinger equation with an external potential:
\begin{equation}
\tag{NLS} \label{NLS}
i \partial_t u - \partial_x^2 u + V u = |u|^2 u, 
\end{equation}
where the space and time variables $(t,x) \in \mathbb{R} \times \mathbb{R}$ and $u = u(t,x) \in \mathbb{C}$.
This equation derives formally from the Hamiltonian
\begin{align}
\label{NLSVHam}
H = \frac{1}{2}\int_{\R} {|\partial_x u|}^2 \, dx + \frac{1}{2}\int_{\R} V {|u|}^2 \, dx - \frac{1}{4}\int_{\R} {|u|}^4 \, dx,
\end{align}
and also conserves the total mass
$$
M = \int_{\mathbb{R}} |u|^2\,dx.
$$
We will work under fairly mild assumptions on the potential, namely
\begin{align}
\label{Vass1}
V \in W^{2,1}(\R), \qquad |x|^\gamma V \in L^1(\R), \quad \gamma > 6.
\end{align}
Under this localization assumption, it is well known that the spectrum of $L_{V}=-\partial_{x}^2 + V$ as a self-adjoint operator
on  $L^2(\mathbb{R})$ with domain $H^2(\mathbb{R})$ is made of  $[0+ \infty)$ and a finite number of $L^2$ eigenvalues (bound states).
Moreover, on  $(0,+ \infty)$ the spectrum is purely absolutely continuous (actually it suffices that $V \in L^1$, see for example
 \cite{Reed-Simon} for these classical results).

Our main spectral assumption on $L_{V}$ will be
\begin{align}
\label{Vass2}
L_{V} \quad \mbox{has no bound states}, \qquad V \quad \mbox{is generic}.
\end{align}
The precise  formulation  of the assumption that $V$ is generic is given in  Remark \ref{rem1} after Theorem \ref{maintheo} below;
such assumption can be expressed in terms of properties of the scattering matrix associated to $V$, and is equivalent
to the usual assumption that $0$ is not a resonance.

We are going to consider the Cauchy problem for \eqref{NLS} with initial data $u_0$ small in a suitable weighted Sobolev space,
and study the global properties and asymptotic behavior of solutions.
Since we deal with small solutions, the sign in front of the nonlinearity is irrelevant for our main result to hold.
Our main motivation for studying this problem is the question of {\it asymptotic stability} for special solutions of
nonlinear dispersive and hyperbolic equations, such as solitons, traveling waves, kinks...
Indeed, nonlinear equations with external  potentials arise as the linearization of the full
nonlinear problems around these special solutions, and \eqref{NLS} is a prototypical model for
nonlinear equations under the influence of an external potential.

Our approach will be based on the use of the {\it distorted Fourier transform}
- the so-called Weyl-Kodaira-Titchmarsh theory - which will allow us to extend some
Fourier analytical techniques which have been succesfully employed in recent years
to study small solutions of nonlinear equations without potentials, see for example \cite{KP,GMS2,IoPu2}.
Our hope is that the framework developed in the present article will prove useful
to study open questions concerning the stability of (topological) solitons,
and other special solutions for nonlinear evolution equations.

\medskip
\subsection{Previous results}
Before discussing some recent works on one dimensional problems with potentials 
we briefly consider the one dimensional NLS equation in the case of zero potential
\begin{equation}
\tag{NLS0} \label{NLS0}
i \partial_t u - \partial_x^2 u = |u|^2 u.
\end{equation}
We will call this the {\it flat/unperturbed} NLS in contrast to the {\it distorted/perturbed} equation \eqref{NLS}.

It is well-known that the Cauchy problem for \eqref{NLS0} is globally well-posed in $L^2$. 
Moreover, solutions to the Cauchy problem associated to \eqref{NLS0} with
initial data $u|_{t=0} \in H^1 \cap L^2(x^2 dx)$ (bounded energy and variance) exhibit {\it modified scattering} as time goes to infinity.
More precisely, solutions decay at the same rate as linear solutions but they differ from
linear solutions by a logarithmic phase correction.
Using complete integrability this was proven in the seminal work of Deift and Zhou \cite{DZNLS}.
Without making use of complete integrability (and in the case of similar but non-integrable versions of \eqref{NLS0})
and restricting the analysis to small solutions,
proofs of this fact were given by Hayashi and Naumkin \cite{HN},
Lindblad and Soffer \cite{LinSof}, Kato and Pusateri \cite{KP}, and Ifrim and Tataru \cite{ITNLS}.
Similar results for the nonlinear Klein-Gordon equation have been obtained by Delort \cite{DelortKG1d}, covering also the case
of quasilinear quadratic nonlinearities, and Lindblad and Soffer \cite{LinSofKG}.
A similar asymptotic behavior occurs for solutions of many
other dispersive and hyperbolic equations, such as for example the modified KdV equation \cite{HNKdV,mKdV},
fractional Schr\"odinger equations \cite{IoPu1}, and water waves \cite{IoPu2,ADa,IoPu4,ITg}.

Notice that solutions scatter (without phase correction) if one replaces the cubic nonlinearity in \eqref{NLS0} by a higher power.
In \cite{CGV} Cuccagna, Georgiev and Visciglia considered the subcritical problem with external potential 
$i \partial_t u - \partial_x^2 u + V u = |u|^p u$, with $2 < p < 4$.
and were able to prove linear decay and (regular) scattering in $L^2$ for small initial data with bounded energy and variance.
The key in this work is a commutator estimate involving a distorted version of the vector field $J = x - 2it\partial_x$.
Successful commutation with this distorted vectorfield guarantees the boundedness of its action on solutions, and gives
the decay which is necessary to close the argument.
Recently, Delort published a result \cite{DelortNLSV} for the critical case of \eqref{NLS} in the case of odd solutions and even potentials.
Cuccagna-Georgiev-Visciglia also announced a similar result \cite{CGVann}.
We will comment below on the relevance of considering odd solutions and how this is related to enhanced decay properties,
cancellations and asymptotics.

After completing the present work, we learned of the paper \cite{IPNaumkin}, which proves a result similar to the main theorem below. The very elegant method is an extension to the distorted setting of the factorization method of Hayashi and Naumkin. The conditions on $u_0$ and $V$ are weaker than ours, and probably close to minimal. However, the method which we propose here is very robust and flexible: it would be straightforward to extend it to the cubic nonlinear Klein-Gordon equation; or to consider a nonlinearity of the type $a(x) |u|^2 u$, where $a(x) \to 1$ as $|x| \to \infty$. A more delicate adaptation should allow to treat a quadratic Klein-Gordon equation, by first applying a normal form transform, followed by the analysis performed in this paper.

\medskip
\subsection{Motivation}
As already pointed out, one of our main motivations for studying \eqref{NLS}
is the question of asymptotic stability for special solutions of nonlinear dispersive and hyperbolic equations.
Studies on the existence and stability of solitons, traveling waves,
and other types of special solutions are numerous and span an extensive body of literature.
Given the impossibility of being exhaustive we  refer the reader to the seminal papers by Weinstein \cite{Wein86},
Pego and Weinstein \cite{PW1}, Soffer and Weinstein \cite{SofWeinA, SofWein2},
and the more recent expository articles \cite{Sof06,Tao09} and references therein.

The classical approach to asymptotic stability of, say, solitons, is to split the solution into a modulated soliton,
plus a remainder which is called the radiation.
The modulated soliton lives in a finite dimensional space which mirrors the symmetries of the equation.
As for the radiation, it solves an equation whose linear part is given by an equation involving a potential (related to the soliton).
One then tries to establish dispersive estimates for the linear part - involving the potential -
\cite{GolSch,Sch1,Sof06} and leverage these to control the nonlinear terms, so to obtain decay of the radiation
and therefore asymptotic stability.
This approach is in general easier to implement in higher dimension, due to better decay properties: see for instance
\cite{Pillet-Wayne,Cuccagna}.

When the decay of the radiation is weak, an important difficulty in this program is to understand the coupling between the radiation
 and the modulation parameters. For equations that enjoy a separation property between
  the speeds of linear dispersive waves and solitary waves, such as the Korteweg de Vries equation, this coupling
  is weak and can be handled through monotonicity formulas.
  Asymptotic stability results then follow in the sense that perturbations decay on one side of the  wave
  \cite{PW1,Martel-Merle}, see also \cite{Bethuel} for recent results on solitary waves of the Gross-Pitaevskii equations. Recently, in \cite{mKdV}, we could prove the full  asymptotic stability of solitons
  - that is a description of the asymptotic behavior 
  of the perturbation on the other side of the wave -
  for the mKdV equation, by combining these techniques with the ones used to prove modified scattering for small data.
  For equations like Klein-Gordon or Schr\"odinger, the coupling between the radiation and the modulation
  parameters is stronger and it is usually controlled after  normal form transforms in the system coupling
  the modulation parameters and the radiation via the ``Fermi Golden rule''
  \cite{SofWein2,Cuccagna,Buslaev-Perelman,Bambusi-Cuccagna}.

Note that very interesting virial type arguments have been developed  recently  for the $\phi^4$ model \cite{KowMarMun}.
   Nevertheless in the one-dimensional case, we are not aware of situations where the full asymptotic stability
   of solitons has been shown for a nonlinearity which is critical for the dispersion
    (in the sense that small solutions do not scatter linearly) outside the use of complete integrability,
   see for example \cite{CP} on cubic NLS, or when there is separation between the soliton and the radiation, see our work \cite{mKdV}.

\medskip
\subsection{Main result}

Our main result, stated below, gives, for any initial data in a weighted Sobolev space (in particular for any function in the Schwartz class)
that solutions of the perturbed equation \eqref{NLS}
decay globally-in-time at the same rate as solutions of the linear equation $i\partial_t u - \partial_{xx} u = 0$.
Furthermore, as time approaches infinity, they approach, up to a logarithmic phase correction, solutions of the linear problem.

\begin{thm}\label{maintheo}
Consider the nonlinear Schr\"odinger equation \eqref{NLS} with a  potential $V$ satisfying
\begin{align}
\label{mtV}
V \in W^{2,1},
\qquad V \in L^1_\gamma, \quad \gamma > 6,
\qquad V \quad \mbox{has no bound states},
\end{align}
and $V$ is generic in the sense of Remark \ref{rem1} below.
The following hold true:

\setlength{\leftmargini}{1.5em}
\begin{itemize}

\bigskip
\item ({\it Global existence and decay}).
There exists $\bar{\e}>0$ small enough such that for all $\e_0\leq\bar{\e}$ and $u_0$ satisfying
\begin{align}
\label{mtdata}
{\| u_0 \|}_{H^3} + {\| x u_0 \|}_{L^2} = \e_0,
\end{align}
the equation \eqref{NLS} with initial data $u(t=0) = u_0$ admits a unique global solution satisfying
\begin{align}
\label{mty}
\sup_{t\in \R}{\| u(t)\|}_{L^\infty_x} \lesssim \e_0 (1+|t|)^{-1/2}.
\end{align}

\bigskip
\item ({\it Global bounds}).
Define the profile of $u$ by
\begin{align}
\label{prof0}
f(t,x) := e^{-it (-\partial_{x}^2 + V)} u(t,x), \qquad \wt{f}(t,k) := e^{-itk^2} \wt{u}(t,k),
\end{align}
where, for any $g \in L^2$, $\wt{g} = \wt{\mathcal{F}}g$ denotes the distorted Fourier transform of $g$ (see \eqref{distF}).
Let $p_0=1/100$, $\alpha \in(0,1/4)$, then the global solution of \eqref{NLS} with data \eqref{mtdata} satisfies
\begin{align}
\label{mtbounds}
(1+|t|)^{-p_0}{\big\| (1+|k|)^3 \wt{f}(t) \big\|}_{L^2} + {\big\| \wt{f}(t) \big\|}_{L^\infty}
  + (1+|t|)^{-1/4 + \alpha}{\big\| \partial_k \wt{f}(t) \big\|}_{L^2}
\lesssim \e_0.
\end{align}

\bigskip
\item ({\it Asymptotic behavior as $t \to + \infty$}). There exists $W_{+\infty} \in L^\infty$ such that
\begin{align}
\label{mtas+}
\left| \widetilde{f}(t,k) \exp \left( \frac{i}{2\sqrt{2\pi}} \int_0^t |\widetilde{f}(s,k)|^2 \frac{ds}{s+1} \right)
  - W_{+\infty}(k) \right| \lesssim (1+t)^{-\rho/2}, \qquad \mbox{for} \quad t>0,
\end{align}
for $0<\rho<\alpha/10$.

\bigskip
\item ({\it Asymptotic behavior as $t \to -\infty$}).
Let $S = S(k)$ be the scattering matrix associated to $V$, see \eqref{scatmat}, and let
\begin{align}
\label{mtas1}
\qquad Z(t,k) := \big(\wt{f}(t,k), \wt{f}(t,-k)\big), \qquad k>0.
\end{align}
Define the self-adjoint matrices
\begin{align}
\label{mtas2}
\begin{split}
& \mathcal{S}_0(t,k) := \frac{1}{2\sqrt{2\pi}}\mathrm{diag} \big( |\wt{f}(t,k)|^2, |\wt{f}(t,-k)|^2 \big),
\\
& \mathcal{S}_1(t,k) := \frac{1}{2\sqrt{2\pi}} S^{-1}(k) \mathrm{diag} \big( |(SZ(t,k))_1|^2, |(SZ(t,k))_2|^2 \big) S(k),
\end{split}
\end{align}
and let
\begin{align}
\label{mtas2.5}
\mathcal{S}(t,k) :=
\mathbf{1}(k\leq |t|^{-\rho})\mathcal{S}_0(t,k) +
  \mathbf{1}(k\geq |t|^{-\rho}) \dfrac{1}{2}\Big[\mathcal{S}_0(t,k) + \mathcal{S}_1(t,k)\Big],
\end{align}
for $0<\rho<\alpha/10$.

Then, if we denote
\begin{align}
\label{mtas3}
W (t,k):= \exp\Big(i\int_0^t \mathcal{S}(t,k) \, \frac{ds}{s+1} \Big) Z(t,k), \qquad |W(t,k)| = |Z(t,k)|,
\end{align}
there exists $W_{-\infty} \in L^\infty$ such that
\begin{align}
\label{mtas4}
\big| W(t,k) - W_{-\infty}(k) | \lesssim (1+|t|)^{-\rho/2}, \qquad \mbox{for} \quad t<0.
\end{align}


\end{itemize}

\end{thm}

\bigskip
Before describing in more details some of the main ideas in the proof, let us make some comments:

\medskip
\setlength{\leftmargini}{1.5em}
\begin{enumerate}

\item \label{rem1} {\it Genericity of the potential}.
 We assume that $V$ is generic  in the following sense:
\begin{align}
\label{mtr1}
\int_\R V(x) \, m(x) dx \neq 0
\end{align}
where $m$ is the unique solution of $(-\partial_{x}^2 + V)m = 0$ which is bounded for $x \gg 1$.
In particular one can see that  \eqref{mtr1} is equivalent to the fact that the
 transmission coefficient (see Section \ref{secspth}  below for the definitions of $T$ and $R_\pm$) satisfies  $T(0) = 0$, $T'(0) \neq 0$
(and hence the reflection coefficients $R_\pm(0) = -1$),
see \eqref{TRformula}-\eqref{TRsmallk0}. This is also equivalent to the fact that $0$ is not a resonance.
Indeed the fact that $0$ is not a resonance is usually formulated  in dimension $1$ (see \cite{GolSch} for example)  in terms of $W(0) \neq 0$
where  $W(k)= [f_{+}(k), f_{-}(k)]$ is the Wronskian between the two Jost functions (see section \ref{spth1} for the definition).
Since $W(k)= 2ik/T(k)$ (see \cite{DeiTru} p. 144) and $W$ is continuous, our assumption is equivalent to $W(0) \neq 0$.

Note, see Lemma \ref{Lem0} below, that under this generic assumption, for any $f \in L^1$, one has $\wt{f}(0) = 0$,
where $\wt{f}$ is the distorted Fourier transform of $f$. See again Section \ref{secspth} and the definitions \eqref{psixk}-\eqref{distF}.
Note that if $\widetilde{f}(0) = 0$, according to the asymptotic formulas \eqref{aslin+}-\eqref{aslin-} below,
one would get additional decay in time for $u(x,t)$ when $|x|\ll t$, provided $\widetilde{f}$ is sufficiently smooth.
This type of improved decay has been observed, for example, in \cite{Sch1}.
While we do not directly make use of this additional time decay in physical space,
we do rely on the improved behavior of some of the nonlinear interactions when the input frequencies are small.


\medskip
\item \label{rem2} {\it Assumptions on the data and the special case of odd solutions}.
Notice that we do not put any additional restriction on our initial data besides standard regularity
and spatial decay. In particular we do not require the data to be odd and the potential to be even as in \cite{CGVann,DelortNLSV}.

It is interesting to note that the expressions in \eqref{mtas2}-\eqref{mtas3}
involve explicitly the scattering matrix $S$ associated to the potential $V$, see \eqref{scatmat}.
It turns out that this is not the case if one assumes that $V$ is even and the initial data is odd.
Indeed, under these additional assumptions, $\wt{f}$ is  odd, the reflection coefficients coincide, that is, $R_+ = R_-$,
and the expression in \eqref{mtas2} simplifies
to $\mathcal{S} = \mathcal{S}_0$ for all $t$.

\medskip
\item \label{rem3} {\it About the modified asymptotics: physical space}.
From \eqref{mtas+} and a slight refinement of Proposition \ref{propdisp},
one can also derive a statement about nonlinear asymptotics in physical space.
More precisely one can show that, under our global bounds, see \eqref{bootstrap},
\begin{align}
\label{aslin+}
u(t,x) & = \frac{e^{ix^2/4t}}{\sqrt{-2it}} \, \widetilde{f}\Big(t,-\frac{x}{2t}\Big) + O(|t|^{-1/2+\alpha}), \qquad t \gg 1,
\end{align}
while, for $t \ll -1$, denoting $k_0:= -x/2t$, we have
\begin{align}
\label{aslin-}
\begin{split}
\begin{array}{ll}
u(t,x) = \dfrac{e^{ix^2/4t}}{\sqrt{-2it}} \Big[ T(k_0)\widetilde{f}(k_0)
  + R_+(k_0)\widetilde{f}(-k_0) \Big] + O(|t|^{-1/2+\alpha}),  & \qquad x>0,
\\
\\
u(t,x) = \dfrac{e^{ix^2/4t}}{\sqrt{-2it}} \Big[ T(-k_0)\widetilde{f}(k_0)
  + R_-(-k_0) \widetilde{f} (-k_0) \Big] + O(|t|^{-1/2+\alpha}), & \qquad x<0.
\end{array}
\end{split}
\end{align}
Notice how the scattering matrix \eqref{scatmat} associated to the potential also appears explicitly here.

Combining \eqref{aslin+}-\eqref{aslin-} with \eqref{mtas+} it is then possible to obtain the following asymptotic expression:
\begin{align}
\label{asreal}
u(t,x) & = \frac{e^{ix^2/4t}}{\sqrt{-2it}} \exp\Big(\frac{i}{2\sqrt{2\pi}}
    \Big|W_{+\infty}\Big(-\frac{x}{2t}\Big)\Big|^2 \log t \Big) W_{+\infty}\Big(-\frac{x}{2t}\Big)
    + O(|t|^{-1/2+\alpha/2}), 
\end{align}
for $t\geq 1$.

As $t\rightarrow-\infty$, the expression in the distorted Fourier space is more complicated and involves
the scattering matrix $S$, see \eqref{mtas2.5}-\eqref{mtas4}.
In particular, it is interesting to notice how the expression for the modified profile at frequency $k$
involves both the frequencies $k$ and $-k$.

\medskip
\item {\it Reversing time.} Though \eqref{NLS} is symmetrical by reversing time (and taking the complex conjugate of $u$),
the phase correction for $t \to -\infty$ is much more complicated than it is for $t \to \infty$.
This follows from our choice of the distorted Fourier transform $\widetilde{\mathcal{F}}$
(defined in \eqref{distF}), which is sometimes denoted $\mathcal{F}_+$, and can be defined through the wave operator $W_{+}$ by
$$
W_+ = s - \underset{t\to \infty}{\operatorname{lim}} e^{it(-\partial_x^2 + V)} e^{it \partial_x^2} = \mathcal{F}_+^{-1} \widehat{\mathcal{F}}
$$
(where $\widehat{\mathcal{F}}$ is the flat, classical Fourier transform).
Flipping the $+$ signs in this definition, one obtains another distorted Fourier transform, $\mathcal{F}_-$, defined by
$$
W_- = s - \underset{t\to -\infty}{\operatorname{lim}} e^{it(-\partial_x^2 + V)} e^{it \partial_x^2} = \mathcal{F}_-^{-1} \widehat{\mathcal{F}}.
$$
This second distorted Fourier transform is better adapted to analyzing negative times, and would give simple asymptotics as $t \to - \infty$.

\medskip
\item \label{rem4} {\it The bootstrap space.} The bulk of our analysis is performed in the distorted Fourier space,
and the nonlinear evolution stays small in the space
\begin{align}
\label{mtr4}
(1+|t|)^{-p_0}{\big\| (1+|k|)^3 \wt{f}(t) \big\|}_{L^2} + {\big\| \wt{f}(t) \big\|}_{L^\infty}
  + (1+|t|)^{-1/4 + \alpha}{\big\| \partial_k \wt{f}(t) \big\|}_{L^2},
\end{align}
for some $\alpha \in (0,1/4)$. 
The motivation for choosing the above space is that it guarantees the desired sharp decay of $(1+|t|)^{-1/2}$, see the linear
estimates in Proposition \ref{propdisp}.

\medskip
\item \label{rem5} {\it Vector fields methods.} There is a substantial difference in the way we perform weighted estimates
using the distorted Fourier transform, and alternative approaches based on the vector fields method, 
such as Donninger and Krieger \cite{DoKri} and Cuccagna, Georgiev and Visciglia \cite{CGV}. 

These approaches are based on using $L^2$ norms weighted by vectorfields to deduce decay for a
general function $u$, and then estimating vectorfields of the full nonlinear solution.
In our approach, we look at a true linear solutions of the perturbed equation, establish a decay estimate
- in this case involving $\widetilde{f}$ and $\partial_k \widetilde{f}$ - and then estimate the relevant quantities in the nonlinear problem.


\end{enumerate}

\medskip
\subsection{Ideas of the proof}
Our approach will be based on the use of the {\it distorted Fourier transform}
(the Weyl-Kodaira-Titchmarsh theory), which will allow us to extend many recent successful
Fourier analytical techniques used to study small solutions of nonlinear equations without potentials.
In the setting of the distorted Fourier transform, we then follow the basic idea of the
{\it space-time resonance method} by filtering the solution by the linear group,
and viewing the (nonlinear) Duhamel term as an oscillatory integral: see \cite{G,GMS1,GMS2}
for higher-dimensional instances, and \cite{KP} for \eqref{NLS0}, which provides in many respects a blueprint for the present paper.
A first attempt to extend the space-time resonance method to a perturbed case can be found in \cite{GHW}.

\subsubsection{The equation on the profile in distorted Fourier space}
We refer to section~\ref{secspth} for a more detailed presentation of the distorted Fourier transform,
and admit for the moment the existence of generalized eigenfunctions $\psi(x,k)$ such that
$$
\forall k \in \mathbb{R}, \qquad (- \partial_x^2 + V ) \psi(x,k) = k^2 \psi(x,k),
$$
and that the familiar formulas relating the Fourier transform and its inverse in dimension $d=1$ hold if one replaces
$e^{ikx}$ by $\psi(k,x)$:
$$
\widetilde{f}(k) = \int_\mathbb{R} \overline{\psi(x,k)} f(x)\,dx \qquad \mbox{and} \qquad f(x) = \int_\mathbb{R} {\psi(x,k)} \widetilde{f}(k) \,dk.
$$
Defining then the profile $f$ by
$$
f = e^{-it(-\partial_{xx}+V)} u \quad \mbox{or equivalently} \quad \widetilde{f}(t,k) = e^{-itk^2} \widetilde{u}(t,k),
$$
it is easy to check that it satisfies the equation
\begin{equation*}
\partial_t \widetilde{f}(t,k) = -i\iiint e^{it (-k^2 + \ell^2 - m^2 + n^2)} \widetilde{f}(t,\ell) \overline{\widetilde{f}(t,m)} \widetilde{f}(t,n)
\mu(k,\ell,m,n) \,d\ell \, dm \, dn,
\end{equation*}
hence
\begin{align}
\label{coccinelle}
\partial_t \widetilde{f}(t,k) = \widetilde{u_0}(k)
-i \int_0^t \iiint e^{is (-k^2 + \ell^2 - m^2 + n^2)} \widetilde{f}(s,\ell) \overline{\widetilde{f}(s,m)} \widetilde{f}(s,n)
\mu(k,\ell,m,n) \,d\ell\, dm\, dn\,ds,
\end{align}
where
\begin{align}
\label{coccinelle2}
\mu(k,\ell,m,n) = \int \overline{\psi(x,k) }\psi(x,\ell) \overline{\psi(x,m) }\psi(x,n) \, dx
\end{align}
characterizes the interaction between the generalized eigenfunctions.

At this point, the essential difference with the flat case becomes clear: if $V=0$, $\psi(x,k)$ should be replaced by $e^{ikx}$,
in which case $\mu(k,\ell,m,n) = \delta(k-\ell+m-n)$. But if $V\neq 0$, the structure of $\mu$ becomes much more involved:
we will see that it can be decomposed into
\begin{align}
\label{coccinelle3}
\begin{split}
\mu(k,\ell,m,n) = \sum_{\beta,\gamma,\delta,\epsilon = \pm 1} & \Big[
A_{\beta,\gamma,\delta,\epsilon}(k,\ell,m,n) \delta(\beta k + \gamma \ell + \delta m + \epsilon n)
\\
& + B_{\beta,\gamma,\delta,\epsilon}(k,\ell,m,n) \, \mbox{p.v.} \frac{1}{\beta k + \gamma \ell + \delta m + \epsilon n} \Big]
+ C(k,\ell,m,n),
\end{split}
\end{align}
where $A_{\beta,\gamma,\delta,\epsilon}$, $B_{\beta,\gamma,\delta,\epsilon}$, and $C$ are relatively smooth functions
(depending on the potential), and ``p.v.'' stands for principal value.

The structure of the coefficients in \eqref{coccinelle3} plays an important role.
In particular, we will see that the structure of the coefficients $B_{\beta,\gamma,\delta,\epsilon}$
will lead to some special cancellation of the worst terms appearing in the estimate for $\partial_k\widetilde{f}$.
Further null structures at low frequencies in some of the coefficients $B_{\beta,\gamma,\delta,\epsilon}$ and in $C$
will also allow us to close the crucial bounds on $\partial_k\widetilde{f}$ and $\widetilde{f}$ in \eqref{mtbounds}.

\smallskip
\subsubsection{The multilinear oscillatory integral}
The whole challenge is to analyze the right-hand side of~\eqref{coccinelle},
which is a multilinear oscillatory integral with phase $\Phi(k,\ell,m,n)=-k^2+\ell^2-m^2+n^2$,
where $\widetilde{f}$ has limited regularity and the kernel $\mu$ is as above.
It requires a delicate decomposition, which is the heart of the argument,
and will be explained precisely in the following sections. For the moment,
let us simply notice that, in regions in $(\ell,m,n)$ where $\mu$ is smooth and $\Phi$ nondegenerate,
the convergence of the the right-hand of~\eqref{coccinelle} is easy to establish.

First of all, problems arise, of course, close to the singular set of $\mu$
$$
\operatorname{Sing} \mu = \cup_{\beta,\gamma,\delta,\epsilon = \pm 1} \{\beta k + \gamma \ell + \delta m + \epsilon n = 0\}.
$$

Next, to take advantage of oscillations, one can integrate by parts through the formula
$$
\frac{1}{i s \partial_{\mathbf{e}} \Phi} \partial_{\mathbf{e}} e^{i s \Phi} = e^{is \Phi}
$$
if $\mathbf{e}$ is a vector in $(\ell,m,n)$ space. This is however only helpful if this manipulation does not result in the singularity of $\mu$ getting worse. In other words, $\mathbf{e}$ should be tangent to $\{\beta k + \gamma \ell + \delta m + \epsilon n = 0\}$ (where $\beta,\gamma,\delta,\epsilon$ depend on the part of $\mu$ which is considered). In other words, we see that the relevant notion of stationary points in $(\ell,m,n)$ (``space resonances'') is given by stationary points of
$\Phi$ restricted to $\{\beta k + \gamma \ell + \delta m + \epsilon n = 0\}$.

Finally, a last option is to integrate by parts in $s$ through the formula
$$
\frac{1}{i\Phi} \partial_s e^{is\Phi} = e^{is\Phi};
$$
obviously, this is only helpful away from the set $\{ \phi = 0 \}$ (``time resonances'').

Most worrisome are the points which belong to the three categories:
the singular set of $\mu$, space resonances, and time resonances.
It turns out that these are of the form $\ell,m,n = \pm k$ and will ultimately lead to an ODE giving an oscillatory phase correction.

\smallskip
\subsubsection{The bootstrap argument}
We will prove an a priori estimate for the following  norm
\begin{align}
\label{apriori0}
\| u \|_X = \sup_{t} \left[ \| \widetilde{f}(t) \|_{L^\infty} + \langle t \rangle^{-p_0} \| u(t) \|_{H^3}
  + \langle t \rangle^{-1/4+\alpha} \| \partial_k \widetilde{f}(t) \|_{L^2} \right].
\end{align}
(recall $p_0 = 1/100$). More precisely, we will assume that the initial data $u_0$ satisfies \eqref{mtdata} and that for $\e_1=\e_0^{2/3}$
we have the apriori bound
\begin{equation}
\label{bootstrap}
\|u\|_{X} \leq \e_{1}.
\end{equation}
We will then show that this estimate improves to
\begin{align}
\label{bootstrapest}
\|u\|_{X} \leq C\e_{0} +  C\e_{1}^3,
\end{align}
for some absolute constant $C>0$. For $\e_0$ sufficiently small, this estimate combined with a bootstrap argument, and the choice $\e_1 = 2 C \e_0$,
gives global existence of solutions which are small in the space $X$.
As part of the argument needed to obtain \eqref{bootstrapest} we will establish the asymptotic behavior of solutions
as described in \eqref{mtas+}-\eqref{mtas4} of Theorem \ref{maintheo}.

For simplicity, and without loss of generality, we only consider $t \geq 1$,
assuming that a local solutions has been already constructed on the time interval $[0,1]$ by standard methods.
Using also time reversibility we obtain solutions for all times.

We remark that in the definition \eqref{apriori0} we could equivalently replace $\| \partial_k \widetilde{f}(t) \|_{L^2}$ by
\begin{align*}
{\| \partial_k \mathbf{1}_+ \widetilde{f}(t) \|}_{L^2} + {\| \partial_k \mathbf{1}_- \widetilde{f}(t) \|}_{L^2},
\end{align*}
where $\mathbf{1}_\pm$ is the characteristic function of $\pm k \geq 0$,
and control this quantity instead. Notice this is finite at time $0$ because $\wt{u_0}(0)=0$, see Lemma \ref{Lem0}.

\smallskip
\subsubsection{Structure of the proof}
The rest of the paper is organized as follows:

\setlength{\leftmargini}{1.5em}
\begin{itemize}

\item Section \ref{secspth} contains an exposition of the elements of the spectral theory
of operators $-\partial_x^2 + V$ on $\mathbb{R}$ which will be needed.

\item Section \ref{secprel} is dedicated to three preliminary results: the linear estimate
$$
{\| u(t) \|}_{L^\infty} \lesssim {(1+|t|)}^{-1/2} {\| f \|}_X,
$$
which allows to deduce decay of $u$ from the control of the bootstrap norm, the energy estimate in $H^3$, and a lemma describing precisely the structure of the measure $\mu$ in \eqref{coccinelle2}.

\item Section \ref{secw} gives the control of the weighted norm component of the space $X$.
By weighted norm, we always mean $\| \partial_k \widetilde{f}(t) \|_{L^2}$, which is indeed akin to a weighted norm in physical space.
The control on this norm relies on a precise analysis of the multilinear oscillatory integral, and some key cancellation. 

\item Finally, Section \ref{secLinfty}  gives the control of $\| \widetilde{f}(t) \|_\infty$.
Once again, this is achieved through a precise analysis of the multilinear oscillatory integral.
It allows us to derive an ODE which describes the leading order behavior of $\widetilde{f}$, and whose solutions are bounded.

\end{itemize}

\medskip
\subsection*{Acknowledgements}
We thank Z. Hani for communicating to us that Cuccagna-Georgiev-Visciglia had announced a result for the case of odd solutions
and even potentials \cite{CGVann}. 
We thank A. Stefanov for letting us know about the paper \cite{IPNaumkin}.
P. G. was partially supported by the NSF grant DMS-1501019.
F. P. was partially supported by the NSF grant DMS-1265875.

\bigskip
\section{Spectral theory in dimension one}\label{secspth}

\medskip
\subsection{Jost solutions}\label{spth1}

Define $f_{+}(x,k)$ and $f_-(x,k)$ by the requirements that
\begin{align}
\label{f+-}
(- \partial_x^2 + V) f_{\pm}  = k^2 f_{\pm}, \quad \mbox{for all $x\in\R$, \quad and} \quad
\left\{
\begin{array}{ll}
f_{+}(x,k) \sim e^{ixk} & \mbox{as $x \to \infty$}
\\
f_{-}(x,k) \sim e^{-ixk} & \mbox{as $x \to - \infty$}.
\end{array}
\right.
\end{align}
Define
\begin{align}
\label{m+-}
m_{+}(x,k) = e^{-ikx} f_{+}(x,k) \quad \mbox{and} \quad m_{-}(x,k) = e^{ikx} f_{-}(x,k).
\end{align}
We will need precise bounds on $m_\pm$ and their derivatives, and for this we define
\begin{align}
\label{defWpm}
\mathcal{W}_{+}^s(x)= \int_{x}^{+ \infty} \langle y \rangle^s |V(y)| \, dy, \quad
 \mathcal{W}_{-}^s(x)= \int_{-\infty}^{x} \langle y \rangle^s |V(y)| \, dy.
\end{align}

Let us recall that we say that $V \in L^1_{\gamma}$ if $\langle x\rangle^\gamma |V| \in L^1$.

\begin{lem}
\label{lemm+-}
 For every $s\geq 0$, assuming that $V \in L^1_{s+1}$, we have the following estimates that are uniform in $x$ and $k$,
\begin{align}
\label{mgood}
&| \partial_{k}^s( m_{\pm}(x,k)- 1) | \lesssim { 1 \over \langle k \rangle} \mathcal{W}_{\pm}^{s+1}(x), \quad \pm x \geq -1,\\
\label{mbad}
& | \partial_{k}^s( m_{\pm}(x,k)- 1) | \lesssim { 1 \over \langle k \rangle}   \langle x \rangle^{s+1}, \quad \pm x \leq 1.
  \end{align}
  Moreover, we also have the following control of the $x$ derivatives:
  \begin{align*}
&|\partial_{x} \partial_{k}^s m_{\pm}(x,k) | \lesssim  \mathcal{W}_{\pm}^{s}(x), \quad \pm x \geq -1,\\
& |\partial_{x} \partial_{k}^sm_{\pm}(x,k) | \lesssim  \langle x \rangle^{s}, \quad \pm x \leq 1.
  \end{align*}
  \end{lem}

 The proof of these estimates is sketched in Appendix A.

\medskip
\subsection{Transmission, Reflection, and Scattering Matrix}\label{spth2} A classical reference for the formulas which we recall here is~\cite{DeiTru} (see also \cite{Weder2}, \cite{Yafaev} for example).
Denote $T(k)$ and $R_{\pm}(k)$ respectively the {\it transmission} and {\it reflection} coefficients associated to the potential $V$.
These coefficients are such that
\begin{align}
\label{f+f-}
\begin{split}
&f_+ (x,k) = \frac{1}{T(k)} f_-(x,-k) + \frac{R_-(k)}{T(k)} f_-(x,k),
\\
&f_- (x,k) = \frac{1}{T(k)} f_+(x,-k) + \frac{R_+(k)}{T(k)} f_+(x,k)
\end{split}
\end{align}
or, equivalently,
\begin{align*}
&f_+ (x,k) \sim \frac{1}{T(k)} e^{ikx} + \frac{R_-(k)}{T(k)} e^{-ikx} \quad \mbox{as $x \to - \infty$},
\\
&f_- (x,k) \sim \frac{1}{T(k)} e^{-ikx} + \frac{R_+(k)}{T(k)} e^{ikx} \quad \mbox{as $x \to \infty$}.
\end{align*}
Moreover, they are given by the formulas, see \cite[pp. 145-146]{DeiTru},
\begin{align}
\label{TRformula}
\begin{split}
& \frac{1}{T(k)} = 1 - \frac{1}{2ik} \int V(x) m_{\pm} (x,k)\,dx,
\\
& \frac{R_{\pm}(k)}{T(k)} = \frac{1}{2ik} \int e^{\mp 2ikx} V(x) m_{\mp}(x,k)\,dx,
\end{split}
\end{align}
and satisfy
\begin{align}
 \label{TRconj}
\begin{split}
& T(-k) = \overline{T(k)}, \qquad R_{\pm}(-k) = \overline{R_{\pm}(k)},
\\
& |R_{\pm}(k)|^2 + |T(k)|^2 = 1, \qquad T(k)\overline{R_-(k)} + R_+(k)\overline{T(k)} = 0.
\end{split}
\end{align}
In the present paper, we  recall that we consider the generic case
\begin{align}
\label{Vgeneric}
\int V(x) m_{\pm} (x,0)\,dx \neq 0,
\end{align}
for which
\begin{align}
\label{TRsmallk0}
T(0) = 0 \qquad \mbox{and} \qquad R_{\pm}(0) = -1.
\end{align}
From the formula \eqref{TRformula} above giving $T$ and $R_{\pm}$ and the
estimates of Lemma \ref{lemm+-}, we obtain the following:
\begin{lem}
\label{lemcoeff}
Assuming that $V \in L^1_{4}$, we have the uniform estimates for $k \in \mathbb{R}$:
\begin{equation}
\label{TRk}
|\partial_k^j T(k)| + |\partial_k^j R_{\pm} (k) | \lesssim {1 \over \langle k \rangle}, \qquad 1\leq j \leq 3.
\end{equation}
\end{lem}

Given $T$ and $R_\pm$ as above one defines the scattering matrix associated to the potential $V$ by
\begin{align}
\label{scatmat}
S(k) :=
\left( \begin{array}{cc}
T(k)  & R_+(k) \\ R_-(k) & T(k)
\end{array}
 \right), \qquad
S^{-1}(k) :=
\left( \begin{array}{cc}
\overline{T(k)} & \overline{R_-(k)} \\  \overline{R_+(k)} & \overline{T(k)}
\end{array}
 \right).
\end{align}

\medskip
\subsection{Flat and distorted Fourier transform}
We adopt the following normalization for the (flat) Fourier transform on the line:
$$
\widehat{\mathcal{F}} \phi (k) = \widehat{\phi}(k) = \frac{1}{\sqrt{2\pi}} \int e^{-ikx} \phi(x) \, dx.
$$
As is well-known,
$$
\widehat{\mathcal{F}}^{-1} \phi = \frac{1}{\sqrt{2\pi}} \int e^{ikx} \phi(k) \, dk= \what{\mathcal{F}}^* \phi,
$$
and $\mathcal{F}$ is an isometry on $L^2(\mathbb{R})$.

Setting now
\begin{align}
 \label{psixk}
\psi(x,k) := \frac{1}{\sqrt{2\pi}}
\left\{
\begin{array}{ll}
T(k) f_+(x,k) & \mbox{for $k \geq 0$} \\ \\
T(-k) f_-(x,-k) & \mbox{for $k < 0$},
\end{array}
\right.
\end{align}
the distorted Fourier transform is defined by
\begin{align}
\label{distF}
\widetilde{\mathcal{F}} \phi(k) = \widetilde {\phi}(k) = \int_\mathbb{R} \overline{\psi(x,k)} \phi(x)\,dx.
\end{align}

\medskip
\subsection{Decomposition of $\psi(x,k)$}

Let $\rho$ be a smooth, non-negative function, equal to $0$ outside of $B(0,2)$ and such that $\int \rho = 1$. Define $\chi_{\pm}$ by
\begin{equation}
\label{chi+-}
\chi_+(x) = H * \rho = \int_{-\infty}^{x} \rho(y)\,dy , \quad \mbox{and} \quad \chi_+(x) + \chi_-(x) = 1,
\end{equation}
where $H$ is the Heaviside function, $H=\mathbf{1}_{x \geq 0}$.

With $\chi_\pm$ as above, and using the definition of $\psi$ in \eqref{psixk} and $f_\pm$ and $m_\pm$ in \eqref{f+-}-\eqref{m+-},
as well as the identity \eqref{f+f-} we can write
\begin{align}
\label{psi>}
\begin{split}
\mbox{for} \quad k>0 \qquad \sqrt{2\pi}\psi(x,k) & = \chi_+(x)T(k)m_+(x,k)e^{ixk}
  \\ & + \chi_-(x) \big[ m_-(x,-k)e^{ikx} + R_-(k) m_-(x,k)e^{-ikx} \big],
\end{split}
\end{align}
and
\begin{align}
\label{psi<}
\begin{split}
\mbox{for} \quad k<0 \qquad \sqrt{2\pi}\psi(x,k) & = \chi_-(x)T(-k)m_-(x,-k)e^{ixk}
  \\ & + \chi_+(x) \big[m_+(x,k) e^{ikx} + R_+(-k)m_+(x,-k) e^{-ikx} \big].
\end{split}
\end{align}

We then decompose
\begin{align}
\label{psidec}
\sqrt{2\pi}\psi(x,k) = \psi_S(x,k) + \psi_L(x,k) + \psi_R(x,k) ,
\end{align}
where the singular part (non-decaying in $x$) is
\begin{align}
\label{psiS+-}
\begin{split}
\mbox{for} \quad k>0 \qquad \psi_S(x,k) & := \chi_-(x) \big[ e^{ikx} - e^{-ikx} \big],
\\
\mbox{for} \quad k<0 \qquad \psi_S(x,k) & := \chi_+(x) \big[ e^{ikx} - e^{-ikx} \big],
\end{split}
\end{align}
the singular part with improved low frequencies behavior is
\begin{align}
\label{psiL+-}
\begin{split}
\mbox{for} \quad k>0 \qquad \psi_L(x,k) & := \chi_+(x) T(k) e^{ikx} + \chi_-(x) (R_-(k) + 1) e^{-ixk},
\\
\mbox{for} \quad k<0 \qquad \psi_L(x,k) & := \chi_-(x) T(-k) e^{ikx} + \chi_+(x) (R_+(-k) + 1) e^{-ixk},
\end{split}
\end{align}
and the regular part is
\begin{align}
\label{psiR+-}
\begin{split}
\mbox{for} \quad k>0 \qquad \psi_R(k,x) & := \chi_+(x) T(k) (m_+(x,k)-1) e^{ikx}
  \\ & + \chi_-(x) \big[ (m_-(x,-k) - 1)e^{ikx} + R_-(k)(m_-(x,k) - 1) e^{-ixk} \big],
\\
\mbox{for} \quad k<0 \qquad \psi_R(k,x) & := \chi_-(x) T(-k) (m_-(x,-k)-1) e^{ikx}
  \\ & + \chi_+(x) \big[ (m_+(x,k) - 1)e^{ikx} + R_+(-k)(m_+(x,-k) - 1) e^{-ixk} \big].
\end{split}
\end{align}

\medskip
\subsection{Properties of the distorded Fourier transform}
Let us collect some useful results about the distorded Fourier transform defined in \eqref{distF}; these results
can be obtained as consequences of the general ``Weyl-Kodaira-Titchmarsh theory'', see for example \cite{Dunford} and \cite{Wilcox}.
Direct proofs in our framework can be found in the book \cite{Yafaev}, Chapter 5.

\begin{thm}
\label{theotordu}
Assume that $V \in L^1_1$,  and that $V$ has no bound states, then
 $\widetilde{\mathcal{F}}$ is an isometry on $L^2$,
$$\| \widetilde{\mathcal{F}} f \|_{L^2}
=  \|  f \|_{L^2}, \quad \forall f \in L^2$$
 and $\widetilde{\mathcal{F}}$ is a bijection with
$$
\widetilde{\mathcal{F}}^{-1} \phi(x) = \int_\mathbb{R} \psi(x,k) \phi(k)\,dk.
$$
Moreover,
the distorded  Fourier transform diagonalizes $-\partial_x^2 + V$:
\begin{equation}
\label{diago}
- \partial_x^2 + V = \widetilde{\mathcal{F}}^{-1} k^2 \widetilde{\mathcal{F}}.
\end{equation}
\end{thm}

Note that we can express the wave operators associated to $-\partial_x^2 + V$ with the help of $\widetilde{\mathcal{F}}$, for example,
$
W_+ = \widetilde{\mathcal{F}}^{-1} \widehat{\mathcal{F}}
$
and that these operators enjoy some $L^p$ boundedness properties, \cite{Weder1}, which nevertheless
we will not use here.

We shall only use the following elementary properties:
\begin{lem}\label{Lem0}
Consider a generic potential $V\in L^1_{1}$ with no bound states, then:

\begin{itemize}
\item[(i)] If    $\phi \in L^1$, then $\widetilde{\phi}$ is a continuous, bounded function. Furthermore, $\widetilde{\phi}(0)=0$.

\item [(ii)] There exists $C>0$ such that
\begin{equation}
\label{FH1} \|k \, \widetilde u \|_{L^2} \leq C (1 + \|V\|_{L^1}^{1 \over 2}) \|u\|_{H^1}, \quad \forall u \in H^1.
\end{equation}

\item[(iii)] If $V \in L^1_{3} $, there exists $C>0$ such that
$$ \| \partial_k \widetilde \phi \|_{L^2} \leq C \displaystyle \| \langle x \rangle \phi \|_{L^2}.$$
\end{itemize}

\end{lem}

We will use (ii) to obtain that a control on the regularity of the solution gives decay on the (generalized) Fourier side,
see Proposition \ref{propH3} below.
Also note that for us, the main consequence of (iii) will be that at the initial time one has
\begin{align}
\|\partial_{k} \widetilde f (0,k) \|_{L^2} \lesssim  \| \langle x \rangle u_{0}\|_{L^2} < \infty.
\end{align}
 Control at later times of $\partial_{k} \widetilde f$ will then guarantee decay for the nonlinear solution
through the linear estimate \eqref{propdisp2} in Proposition \ref{propdisp} below.

\begin{proof} $(i)$ Considering for instance the case $k>0$, recall that
\begin{align*}
\widetilde{\phi}(k) & = \int \overline{\psi(x,k)} \phi(x) \,dx \\
& = \frac{1}{\sqrt{2\pi}} \int \overline{ \left[ \chi_+ (x) T(k) m_+(x,k) e^{ikx} + \chi_-(x) (m_-(x,-k) e^{ikx} + R_-(k) m_-(x,k) e^{-ikx} ) \right] } \phi(x) \,dx
\end{align*}
The properties of $m$ and $T$ imply immediately that $\widetilde{\phi}$ is bounded and continuous.
In the generic case, $\widetilde{\phi}(0)$ follows by using $T(0) = 0 $ and $R_{\pm}(0) = -1$, see \eqref{TRsmallk0}-\eqref{TRk}.

\bigskip
\noindent $(ii)$ We note that
\begin{align*}
\|  k \, \widetilde u \|_{L^2}^2 &= \left( \widetilde{\mathcal{F}}u, k^2 \widetilde{\mathcal{F}} u \right)_{L^2}
= \left( \widetilde{\mathcal{F}}u, \widetilde{\mathcal{F}}(-\partial_{x}^2 + V) u \right)_{L^2}
= \left(u, (-\partial_{x}^2 + V) u\right)_{L^2}
\\ &= \| \partial_{x}u \|_{L^2}^2 + \int_{\R} V |u|^2\, dx,
\end{align*}
where we have used  \eqref{diago} for the second equality and the fact that $\widetilde{\mathcal{F}}$ is an isometry
for the third. This yields
$$ \|  k \, \widetilde u \|_{L^2}^2 \leq\|\partial_{x} u\|_{L^2}^2 + \|V\|_{L^1}\|u\|_{L^\infty}^2
  \lesssim (1 + \|V\|_{L^1}) \|u \|_{H^1}^2.$$

\bigskip

\noindent $(iii)$ Assuming that $\langle x \rangle \phi \in L^2$, we aim at proving that $\partial_k \widetilde \phi \in L^2$. Considering the case $k>0$, $\widetilde{\phi}$ is given by the above formula. To alleviate notations, we will focus on the first summand and show that, if $k>0$,
$ \displaystyle \partial_k \int \overline{ \chi_+ (x) T(k) m_+(x,k) e^{ikx} } \phi(x) \,dx \in L^2$. It splits into
\begin{align}
\label{eagle1} \partial_k \int \overline{ \chi_+ (x) T(k) m_+(x,k) e^{ikx} } \phi(x) \,dx &=  T'(k) \int \overline{ \chi_+ (x) m_+(x,k) e^{ikx} } \phi(x) \,dx \\
\label{eagle2} &  \qquad +\int \overline{ \chi_+ (x) T(k) \partial_k  m_+(x,k) e^{ikx} } \phi(x) \,dx \\
\label{eagle3} &  \qquad - i\int \overline{ \chi_+ (x) T(k) m_+(x,k)e^{ikx}x} \phi(x)\,dx.
\end{align}
Note that, though we are only interested in $k>0$, the terms \eqref{eagle1}, \eqref{eagle2}, \eqref{eagle3}
 are well defined for $k \in \mathbb{R}$, so that we can estimate their $L^2(\mathbb{R})$ norms.
We can then view ~\eqref{eagle1} as a pseudo differential operator applied to $\hat{\mathcal{F}}^{-1} \phi$
with $k$ playing the role of the space variable and $x$ the role of the frequency variable.
Let us recall that for a usual pseudo-differential operator defined by
$$ Op_{a}(u) (y) = {1 \over 2 \pi}\int_{\mathbb{R}} e^{iy \xi} a(y, \xi) \what u(\xi) \, d\xi,$$
 we have by  classical $L^2$ continuity results, see  for example  Theorem 2 in \cite{Hwang} or  \cite{Lerner},
that in dimension $1$, $Op_{a}$ is a bounded operator on $L^2$ as soon as
 $a$, $\partial_{y}a$, $\partial_{\xi}a$  and $\partial_{y \xi} a $ are bounded functions.
  By using this criterion with
  $$ a(y, \xi)=T'(y) \chi_{+}(\xi) m_{+}(\xi, y) $$
  we obtain the result from Lemmas \ref{lemm+-} and \ref{lemcoeff}.
   We handle \eqref{eagle2}, \eqref{eagle3} in the same way, this yields
   $$ \left \|  \partial_k \int \overline{ \chi_+ (x) T(k) m_+(x,k) e^{ikx} } \phi(x) \,dx \right\|_{L^2_{k}(\mathbb{R}_{+})}
    \lesssim \| \what{\mathcal{F}}^{-1} \phi \|_{L^2} + \| \mathcal{F}^{-1}(x \phi) \|_{L^2}
     \lesssim  \| \langle x \rangle \phi \|_{L^2}.$$
\end{proof}

\medskip
\subsection{Littlewood-Paley decomposition and other notations}\label{secLP}
In this article we will work with localizations in frequency
defined, as is standard in Littlewood-Paley theory, as follows:
We let $\varphi: \R \to [0,1]$ be an even, smooth function supported in $[-8/5,8/5]$ and equal to $1$ on $[-5/4,5/4]$.
For $k\in\Z$ we define $\varphi_k(x) := \varphi(2^{-k}x) - \varphi(2^{-k+1}x)$, so that the family $(\varphi_k)_{k\in\Z}$
forms a partition of unity,
\begin{equation*}
 \sum_{k\in\Z}\varphi_k(\xi)=1, \quad \xi \neq 0.
\end{equation*}
We also let
\begin{align*}
\varphi_{I}(x) := \sum_{k \in I \cap \Z}\varphi_k, \quad \text{for any} \quad I \subset \R, \qquad
\varphi_{\leq a}(x) := \varphi_{(-\infty,a]}(x), \qquad \varphi_{> a}(x) = \varphi_{(a,\infty]}(x),
\end{align*}
with similar definitions for $\varphi_{< a},\varphi_{\geq a}$.
To these cut-offs we associate frequency projections $P_k$ through
\begin{equation*}
 P_k g:=\mathcal{F}^{-1}\left(\varphi_k(\xi)\what{g}(\xi)\right)
\end{equation*}
and define similarly $P_{I}g:=\mathcal{F}^{-1}\left(\varphi_{I}(\xi)\what{g}(\xi)\right)$,
$P_{\leq k}g:=\mathcal{F}^{-1}\left(\varphi_{\leq k}(\xi) \what{g}(\xi)\right)$, $k\in\Z$ etc.
We will also sometimes denote $\underline{\varphi_k} = \varphi_{[k-2,k+2]}$.

We also denote $H$ the Heavyside function, and $\mathbf{1}_\pm = (1\pm H)/2$ the characteristic function of $\{\pm x > 0\}$.

We will also use the following notation for trilinear operators
\begin{align}
\label{tri}
T_{\alpha}(f_{1}, f_{2}, f_{3}) = \widehat{\mathcal{F}}^{-1}
\iiint_{\R\times\R\times\R} \widehat{\alpha}(k,\ell,m,n) \widehat{f}_{1}(\ell) \widehat{f}_{2}(m) \widehat{f}_{3}(n) \,d\ell dmdn.
\end{align}

\bigskip
\section{Preliminary results}
\label{secprel}

\medskip
\subsection{Linear estimates}\label{secLin}


\begin{prop}
\label{propdisp}
\begin{itemize}

\item[(i)] For any $t \geq 0$,
\begin{align}
 \label{propdisp1}
\| e^{-it\partial_x^2} f \|_{L^\infty} \lesssim \frac{1}{\sqrt t} \| \widehat{f} \|_{L^\infty}
+ \frac{1}{t^{\frac{3}{4}}} \| \partial_k \widehat{f} \|_{L^2}.
\end{align}

\item[(ii)] If $V \in L^1_1$, and does not have bound states, then for any $t \geq 0$,
\begin{align}
 \label{propdisp2}
\| e^{it(-\partial_x^2+V)} f \|_{L^\infty} \lesssim \frac{1}{\sqrt t} \| \widetilde{f}(t) \|_{L^\infty}
+ \frac{1}{t^{\frac{3}{4}}} \| \partial_k \widetilde{f} \|_{L^2}.
\end{align}
\end{itemize}
\end{prop}

\begin{cor} \label{cor_penguin}
We have
$$
\| e^{-it\partial_x^2} \mathbf{1}_+ (D) f \|_{L^\infty} \lesssim \frac{1}{\sqrt t} \| \widehat{f} \|_{L^\infty} + \frac{1}{t^{\frac{3}{4}}} \| \partial_k \widehat{f} \|_{L^2}
$$
\end{cor}

\begin{proof}
For a smooth cutoff function $\chi$, with compact support, and equal to one in a neighborhood of zero, write
$$
f = \chi(\sqrt t D) f + (1 - \chi(\sqrt t D))f.
$$
Estimate separately the two parts: by the Hausdorff-Young inequality,
$$
\| e^{-it \partial_x^2} \mathbf{1}_+ (D) \chi(\sqrt t D) f \|_{L^\infty} \lesssim \| \chi(\sqrt{t} k ) \widehat{f}(k) \|_{L^1} \lesssim \frac{1}{\sqrt t} \| \widehat{f} \|_{L^\infty},
$$
while Proposition~\ref{propdisp} implies
\begin{align*}
\| e^{-it \partial_x^2} \mathbf{1}_+ (D) (1-\chi(\sqrt t D)) f \|_{L^\infty} & \lesssim \frac{1}{\sqrt t} \| \widehat{f} \|_{L^\infty}
  + \frac{1}{t^{\frac{3}{4}}} \big\| \partial_k \big[(1- \chi(\sqrt t k)) \widehat{f}(k)\big] \big\|_{L^2}
\\
& \lesssim  \frac{1}{\sqrt t} \| \widehat{f} \|_{L^\infty} +  \frac{1}{t^{\frac{3}{4}}} \| \partial_k \widehat{f} \|_{L^2},
\end{align*}
by Hardy's inequality.
\end{proof}

\medskip

\begin{proof}[Proof of Proposition \ref{propdisp} (i)] This is a classical estimate; however, we give a proof which is a slightly adapted version of the Van der Corput lemma, which we will extend to prove (ii).
 $$ \sqrt{2\pi} e^{-it\partial_x^2} f = \int_{\mathbb{R}} e^{ixk + ik^2 t}  \widehat f(k) \, dk= e^{-i  {x^2 \over 4t }}
 I(t,x)$$
 with
 $$ I(t,x)= \int_{\mathbb{R}}e^{it(k - X)^2} \widehat f(k) \, dk, \quad X=  -{x \over 2 t}.$$
   For $\epsilon= {1 \over \sqrt{t}}, $ we write
   $$ I= I_{1}+ I_{2}= \int_{X- \epsilon}^{X+\epsilon}  e^{it(k - X)^2} \widehat f(k) + \int_{|k-X|\geq \epsilon}
   e^{it(k - X)^2} \widehat f(k).$$
 For $I_{1}$, we simply use that by the choice of $\epsilon$,
 $$ |I_{1}| \lesssim \epsilon |\widehat f(X)| + \epsilon \sup_{[X-\epsilon, X+\epsilon]} |\widehat f(k) - \widehat f(X)|
 \lesssim \epsilon |\widehat f(X) | + \epsilon \sqrt \epsilon \| \partial_{k} \widehat f \|_{L^2}
 \lesssim { 1 \over  \sqrt t} | \widehat f (X) |+ {1 \over t^{ 3 \over 4}} \| \partial_{k} \widehat f \|_{L^2}.$$
For $I_{2}$, we integrate by parts:
$$I_{2}=  \int_{|k-X| \geq \epsilon} \partial_{k} ( e^{it(k - X)^2}) { 1 \over 2 it(k-X)} \widehat f(k) \, dk$$
to find that
\begin{multline*}
 |I_{2}| \lesssim   { 1 \over t \epsilon} ( | \widehat f(X+\epsilon)| +   | \widehat f(X-\epsilon)| )
 +  { 1 \over t}\int_{|k-X| \geq \epsilon} { 1 \over |k-X|} |\partial_{k} \widehat f(k)| \,dk +
  { 1 \over t} \int_{|k-X| \geq \epsilon}  { 1 \over |k-X|^2}  |\widehat{f}(k)| \, dk.
\end{multline*}
By Cauchy-Schwarz, we also have that
 $$   { 1 \over t}\int_{|k-X| \geq \epsilon} { 1 \over |k-X|} |\partial_{k} \widehat f(k)| \,dk
  \lesssim  { 1 \over t} \Big( \int_{|k-X| \geq \epsilon} { dk \over |k-X|^2} \Big)^{1 \over 2} \|\partial_{k} \widehat f\|_{L^2}
   \lesssim  { 1 \over t} {1 \over \sqrt \epsilon}  \|\partial_{k} \widehat f\|_{L^2},$$
    and we can estimate
 $$  { 1 \over t} \int_{|k-X| \geq \epsilon}  { 1 \over |k-X|^2}  |\widehat{f}(k)| \, dk \lesssim {1 \over t}
  \left( {  1 \over \epsilon} |\widehat f(X)| + \int_{|k-X| \geq \epsilon} { dk \over |k-X|^{ 3 \over 2}} \| \partial_{k} \widehat f\|_{L^2} \right)
   \lesssim   {1 \over t \epsilon} |\widehat f(X)| +  { 1 \over t \sqrt \epsilon}  \| \partial_{k} \widehat  f\|_{L^2}.$$
Since $\epsilon = \frac{1}{\sqrt t}$, we have thus obtained that
   $$  |I_{2}| \lesssim {1 \over \sqrt t}| \widehat f(X)|  + { 1 \over t^{3 \over 4}} \| \partial_{k} \widehat  f\|_{L^2},$$
which gives the desired estimate for $I$.
\end{proof}

\medskip

\begin{proof}[Proof of Proposition \ref{propdisp} (ii)]
 To handle the  general case, we shall use  the distorted Fourier transform,
  $$  e^{it(-\partial_x^2 + V)} f = \int_{\mathbb{R}} \psi(x,k) e^{ik^2 t} \widetilde f(k)\, dk$$
 and we shall deduce the estimate from the following lemma that is a generalization of the above estimate.
 \begin{lem}
 \label{lemstat}
  Consider a function $a(x,k)$ defined on $I \times \mathbb{R}_{+}$ and    such that
  \begin{equation}
  \label{hypest}
   | a(x, k)| +|k| |\partial_{k}a(x,k)| \lesssim 1, \quad \forall x \in I, \, \forall k \in \mathbb{R}_{+}
  \end{equation}
  and for every  $X \in \mathbb{R}$, consider the oscillatory  integral
  $$
  I(t,X,x)=\int_{0}^{+\infty} e^{it(k-X)^2} a(x,k) \widetilde f(k) \, dk, \quad t>0, \, x \in I.
 $$
   Then, we have the  estimate
  \begin{equation}
  \label{decaylem} | I(t,X,x)|  \lesssim \frac{1}{\sqrt t} \| \widetilde{f}(t) \|_{L^\infty} + \frac{1}{t^{\frac{3}{4}}} \| \partial_k \widetilde{f} \|_{L^2}
  \end{equation}
   which is uniform in $X\in \mathbb{R}$, $t >0$ and $x\in I$.
 \end{lem}

 Let us first use the lemma to prove the proposition.

 We focus on the case $x \geq 0$, the other case being similar.
 We will only use the following estimates which hold for $V \in L^1_1$:
 (see \cite{Weder2} Lemma 2.1, and \cite{Weder1} equations (2.6) and (2.9)):
\begin{eqnarray}
 & & \label{m+}   |m_{+}(x,k) - 1 | \lesssim { 1 \over 1+|k|}, \quad x \geq  0, \\
 & & \label{m+2}  | \partial_{k}m_{+}(x,k)| \lesssim {1 \over |k|}, \quad x \geq  0, \\
 & & \label{Tk2}  | \partial_{k} T(k) | + | \partial_{k}R_{+} (k)|  \lesssim { 1 \over |k|}
\end{eqnarray}
(and, obviously, $|T(k)| + |R_+(k)| \lesssim 1$). We split
   $$ u(t,x)= \int_{-\infty}^{0}  \psi(x,k) e^{ik^2 t} \widetilde f(k)\, dk + \int_{0}^{+\infty}
     \psi(x,k) e^{ik^2 t} \widetilde f(k)\, dk = J_{-}+J_{+}.$$
Start with $J_{+}$, which can be written
     $$ J_{+}= e^{ -i {x^2  \over 4t}} I_{+} (t,X,x) \;\;\mbox{with}\;\;  I_{+}(t,X,x)= \int_{0}^{+\infty} e^{it (k-X)^2} T(k) m_{+}(x,k) \widetilde f(k)\, dk,\;\;\mbox{and}\;\; X=  - {x \over 2 t}.$$
    Thanks to \eqref{m+}, \eqref{m+2}, \eqref{Tk2}, we can thus use Lemma \ref{lemstat} with $x \in I= \mathbb{R}_{+}$,
     and  $a(x,k) = T(k) m_{+}(x,k)$.
      This yields
    $$ |J_{+}| \lesssim   \frac{1}{\sqrt t} \| \widetilde{f}(t) \|_{L^\infty} + \frac{1}{t^{\frac{3}{4}}} \| \partial_k \widetilde{f} \|_{L^2}.$$
Let us turn to $J_{-}$. Recall that for  $k<0$,
     $$\sqrt{2\pi} \psi(x,k)=  T(-k) f_{-}(x, -k)=  T(-k) e^{ikx} m_{-}(x,-k)=
      e^{-ikx}R_{+}(-k)m_{+}(x,-k) + e^{ikx}m_{+}(x,k).$$
       We thus split $J_{-}$ into
      \begin{align*}
    \sqrt{2\pi} &  J_{-}   = \int_{-\infty}^0 e^{-i kx} e^{i k^2 t} R_{+}(-k) m_{+}(x,-k)\,
        \widetilde f(k)\, dk  +   \int_{-\infty}^0 e^{i kx} e^{i k^2 t}  m_{+}(x,k) \widetilde f (k)\, dk \\
        & =  e^{ -i {x^2  \over 4t}}  \int_{0}^{+ \infty} e^{it(k + X)^2} R_{+}(k) m_{+}(x,k) \widetilde f(-k)\, dk +
          e^{ -i {x^2  \over 4t}}  \int_{0}^{+ \infty} e^{it(k - X)^2}  m_{+}(x,-k) \widetilde f(-k)\, dk,
\end{align*}
       where we have set  $X= {x \over 2 t}$ and changed $k$ into $-k$ to pass from the first line to the second line.
    Again, thanks to  \eqref{m+}, \eqref{m+2},  \eqref{Tk2}, we can use Lemma \ref{lemstat}
for $x \in \mathbb{R}_{+}$ to also obtain that
$$ |J_{-}| \lesssim   \frac{1}{\sqrt t} \| \widetilde{f}(t) \|_{L^\infty} + \frac{1}{t^{\frac{3}{4}}} \| \partial_k \widetilde{f} \|_{L^2}.$$
This completes the proof of (ii) in Proposition \ref{propdisp}, but there remains to prove Lemma~\ref{lemstat}. \end{proof}

\begin{proof}[Proof of Lemma \ref{lemstat}]
      Let us first assume that $X \geq 0$ so that there is a stationary point for the phase in the integration domain. We split
 $$  I(t, X, x) = I_{1}(t,X,x)+ I_{2}(t,X,x) = \int_{[X- \epsilon, X+\epsilon] \cap \mathbb{R}_{+}}
    +  \int_{ \mathbb{R}_{+} \setminus [X- \epsilon, X+\epsilon]}\dots $$
Choosing again $\epsilon= \frac{1}{\sqrt t}$, we have by  \eqref{hypest} that
  $$   |I_{1}| \lesssim {1 \over \sqrt t} \| \widetilde f\|_{L^\infty}.$$
      For $I_{2}$, we split again
    $$ I_{2}= \int_{X+\epsilon}^{+\infty} +  \int_{0}^{X- \epsilon}= I_{3}+ I_{4}$$
    with the convention that $I_{4}$ is defined only if $X \geq \epsilon$. In order to bound $I_{3}$,  we integrate by parts as previously:
 \begin{multline*} |I_{3}| \lesssim  {1 \over \sqrt t} \| \widetilde f\|_{L^\infty} + { 1 \over t}\int_{X+ \epsilon}^{+\infty} { 1 \over |k-X|} |\partial_{k}(a(x,k) \widetilde f(k))| \,dk + { 1 \over t} \int_{X+ \epsilon}^{+\infty}  { 1 \over |k-X|^2}  | a(x,k)\widetilde{f}(k)| \, dk.
 \end{multline*}
For the last term,  by using again  \eqref{hypest}, we find
$$   { 1 \over t} \int_{X+ \epsilon}^{+\infty}  { 1 \over |k-X|^2}  | a(x,k)\widetilde{f}(k)| \, dk
 \lesssim { 1 \over t \epsilon} \| \widetilde f\|_{L^\infty} = { 1 \over \sqrt{t}} \| \widetilde f \|_{L^\infty}.$$
  For the other term, still using \eqref{hypest},
  \begin{multline*}
  { 1 \over t}\int_{X+ \epsilon}^{+\infty} { 1 \over |k-X|} { 1 \over |k|}| \widetilde f(k)| \,dk
   +  { 1 \over t}\int_{X+ \epsilon}^{+\infty} { 1 \over |k-X|} | \partial_{k} \widetilde f(k)| \,dk
 \\   \lesssim {1 \over t} \big(\int_{X+\epsilon}^{+\infty} { dk \over (k-X)^2}\big)^{1 \over 2}
    \big(\int_{X+\epsilon}^{+\infty} { dk \over k^2}\big)^{1 \over 2} \| \widetilde f \|_{L^\infty}
     + {1 \over t \sqrt \epsilon} \| \partial_{k} \widetilde f\|_{L^2}
      \lesssim {1 \over \sqrt t} \| \widetilde f\|_{L^\infty} +
       { 1 \over t^{3 \over 4}} \| \partial_{k} \widetilde f \|_{L^2}.
     \end{multline*}
      Consequently, we have proven that $I_{3}$ satisfies
    $$ |I_{3}| \lesssim  {1 \over \sqrt t} \| \widetilde f\|_{L^\infty} +
       { 1 \over t^{3 \over 4}} \| \partial_{k} \widetilde f \|_{L^2}.$$
        It remains $I_{4}$.
     If $X \leq 2 \epsilon$, we use the crude estimate
     $$ |I_{4}| \lesssim  \epsilon \| \widetilde f \|_{L^\infty} =  {1 \over \sqrt t} \| \widetilde f \|_{L^\infty}.$$
      If $X \geq 2 \epsilon$, we write
     $$ |I_{4}| \leq  \Big| \int_{0}^\epsilon \dots \Big| + \Big|\int_{\epsilon}^{X- \epsilon} \dots \Big|
      \lesssim \epsilon \| \widetilde f \|_{L^\infty} + | \widetilde I_{4}|$$
       with
       $$ \widetilde I_{4}= \int_{\epsilon}^{X- \epsilon}  e^{it (k-X)^2} a(x,k) \widetilde f(k)\, dk.$$
To bound $\widetilde I_{4}$, we integrate by parts to obtain
    \begin{multline*}
    | \widetilde I_{4}| \lesssim {1 \over t^{1 \over 2}} \| \widetilde f \|_{L^\infty}
     +  { 1 \over t}\int_{\epsilon}^{X- \epsilon} { 1 \over |k-X|} |\partial_{k}(a(x,k) \widetilde f(k))| \,dk
   +
  { 1 \over t} \int_{\epsilon}^{X- \epsilon}  { 1 \over |k-X|^2}  | a(x,k)\widetilde{f}(k)| \, dk.
      \end{multline*}
    For the last term, we get again
  $$   { 1 \over t} \int_{\epsilon}^{X- \epsilon}  { 1 \over |k-X|^2}  | a(x,k)\widetilde{f}(k)| \, dk
   \lesssim {1 \over t} \| \widetilde f \|_{L^\infty} \int_{- \infty}^{X- \epsilon} { 1 \over |k - X|^2} \, dk
    \lesssim { 1 \over \sqrt{t}} \| \widetilde f \|_{L^\infty}.$$
    For the next to last term, we use again \eqref{hypest}, to get
  \begin{multline*}
   { 1 \over t}\int_{\epsilon}^{X- \epsilon} { 1 \over |k-X|} |\partial_{k}(a(x,k) \widetilde f(k))| \,dk
    \lesssim  {1 \over t} \int_{\epsilon}^{X- \epsilon} { 1 \over  |k-X| |k|} \, dk \| \widetilde f \|_{L^\infty}
     + {1 \over t} \int_{\epsilon}^{X- \epsilon} { 1 \over |k-X|} |\partial_{k} f|\, dk.
  \end{multline*}
   Thanks to Cauchy-Schwarz, we still have that the last term above is bounded by $1/( t \epsilon^{1\over 2}) \| \partial_{k}
    \widetilde f \|_{L^2}$, while
    $$  {1 \over t} \int_{\epsilon}^{X- \epsilon} { 1 \over  |k-X| |k|} \, dk
     \lesssim {1 \over t} \Big(\int_{- \infty}^{X- \epsilon} { 1 \over |k-X|^2}\Big)^{1 \over 2}
      \Big( \int_{\epsilon}^{+ \infty} {1 \over |k|^2}\, dk \Big)^{1 \over 2}
       \lesssim { 1 \over t \epsilon}.$$
       Consequently,
     $$ |I_{4}| \lesssim   {1 \over \sqrt{t}} \| \widetilde f\|_{L^\infty} +
       { 1 \over t^{3 \over 4}} \| \partial_{k} \widetilde f \|_{L^2}.$$
        Gathering the previous estimates, we obtain that $I$ satisfies \eqref{decaylem} for $X \geq 0.$

\bigskip

    It remains to consider $X\leq 0$.
            We observe that in this case,  there  is no stationary point on the integration domain, except if $X=0$.
          We thus write
         $$ I=  \int_0^\epsilon  e^{i (k - X)^2 t}  a(x,k) \widetilde f (k) \,dk+  \int_{\epsilon}^{+\infty}  e^{i (k - X)^2 t}  a(x,k) \widetilde f (k) \,dk.$$
      For the first term, we just write
     $$    \Big| \int_{0}^\epsilon  e^{i (k - X)^2 t}  a(x,k) \widetilde f (k) \,dk \Big | \lesssim \epsilon \|\widetilde f\|_{L^\infty}.$$
      For the second term, we integrate by parts and use \eqref{hypest} to get
     \begin{multline*}
    \Big | \int_{\epsilon}^{+\infty}  e^{i (k - X)^2 t}  a(x,k) \widetilde f (k)  \,dk\Big|
     \lesssim { 1 \over \epsilon t} \|\widetilde f \|_{L^\infty}
      +{ 1 \over t} \int_{\epsilon}^{+\infty} { 1 \over |k-X|} { 1 \over |k|} \, dk \,\| \widetilde f \|_{L^\infty} \\
       + { 1 \over t} \int_{\epsilon}^{+\infty} { 1 \over |k-X|} | \partial_{k} \widetilde f(k) | \, dk
        + { 1 \over t}    \int_{\epsilon}^{+\infty} { 1 \over |k-X|^2}\, dk \| \widetilde f \|_{L^\infty}.   \end{multline*}
       This yields from the same arguments as above
    $$   \Big | \int_{\epsilon}^{+\infty}  e^{i (k - X)^2 t}  a(x,k) \widetilde f (k) \,dk \Big|
     \lesssim  {1 \over \sqrt{t}} \| \widetilde f\|_{L^\infty} +
       { 1 \over t^{3 \over 4}} \| \partial_{k} \widetilde f \|_{L^2}.$$
      We have therefore obtained the estimate \eqref{decaylem} in the case $X \leq 0$. This ends the proof. \end{proof}

\medskip
\subsection{Sobolev estimate} 
\begin{prop}
\label{propH3}
If $V \in W^{2,1}$, then under the bootstrap assumption~\eqref{bootstrap},
$$ \|u(t) \|_{H^3} + \|\langle k \rangle^3 \wt{f}(t)\|_{L^2} \leq C\veps_{0}\langle t \rangle^{C \epsilon_1^2}, \quad \forall t \geq 0$$
\end{prop}

\begin{proof}
 Since $V$ is real, we have that
 $$ {d \over dt} {1 \over 2} \|u(t) \|_{L^2}^2 = 0.$$
 Then, we can apply $-\Delta_{V}= - \partial_{x}^2 + V$ to  \eqref{NLS} to get
 $$i\partial_{t}( -\Delta_{V} u)   + \Delta_{V} (-\Delta_{V}) u = - \Delta_{V}(|u|^2 u).$$
 This yields that  for every $M>0$, we have
 $$ {1 \over 2 }  {d \over dt}  \left( (-\Delta_{V}) u, (-\Delta_{V})^2 u)_{L^2} + M  \|u\|_{L^2}^2 \right)
  =  \Re \big( i \Delta_{V}(|u|^2 u),(- \Delta_{V})^2u)\big )_{L^2}.$$
 Next,  we  observe that for some $C>0$ independent of $M$,
 $$   (-\Delta_{V} u, (-\Delta_{V})^2 u)_{L^2} + M  \|u\|_{L^2}^2  \geq \| \nabla \Delta u \|_{L^2}^2 + M \|u \|_{L^2}^2
  - C \left(\|V\|_{W^{2, 1}} + \|V\|_{W^{2, 1}}^3\right) \|u\|_{W^{2, \infty}}^2$$
  and therefore, by Sobolev embedding and interpolation, we get for $M$ sufficiently large
  $$ (-\Delta_{V} u, (-\Delta_{V})^2 u)_{L^2} + M  \|u\|_{L^2}^2
  \gtrsim \|u\|_{H^3}^2.$$
  Moreover, we also have that
  $$  \big( i \Delta_{V}(|u|^2 u),(- \Delta_{V}^2u)\big)_{L^2} \lesssim ( 1 + \|V\|_{W^{2, 1}}^3)\|u\|_{L^\infty}^2 \|u\|_{H^3}^2
   \lesssim \veps_1^2 \langle t\rangle^{- 1} \|u(t)\|_{H^3}^2,$$
   by using the a priori assumption and Proposition \ref{propdisp}.
  Consequently by integrating in time, we obtain that
  $$   \|u(t)\|_{H^3}^2  \lesssim \veps_{0}^2 + \veps_{1}^2 \int_{0}^t  \langle s\rangle^{- 1} \|u(s)\|_{H^3}^2\, ds$$
  and hence, from the Gronwall's inequality, we find
  $$ \|u(t)\|_{H^3}^2  \lesssim \veps_{0}^2 \langle t \rangle^{C \veps_{1}^2}.$$
  which gives the desired estimate for $u$.

  It remains to estimate
 $ \|\langle k \rangle^3 \tilde f(t)\|_{L^2}$. By using the diagonalization property \eqref{diago} and Lemma \ref{Lem0} ii),
 we obtain
 $$   \|\langle k \rangle^3 \tilde f\|_{L^2} \lesssim \| k  \widetilde{\mathcal{F}} (- \partial_{x}^2 + V )f \|_{L^2} + \|f\|_{H^2}
 \lesssim   \|(- \partial_{x}^2 + V )f \|_{H^1}  + \|f\|_{H^2}\lesssim  \|f\|_{H^3}$$
 since $V \in W^{1, 1}.$
\end{proof}

\medskip
\subsection{Decomposition of the nonlinear spectral measure $\mu$}
According to the decomposition of $\psi(k,x)$ in \eqref{psidec}--\eqref{psiR+-},
we can decompose the measure $\mu$ in \eqref{coccinelle2} into three main parts, which will be treated differently.
We have the following:

\begin{prop}\label{muprop}
Let $\psi$ be defined as in \eqref{psixk} and let $\mu$ be the measure defined by
\begin{align}
\label{mu00}
\mu(k,\ell,m,n) & := \int \overline{\psi(x,k) }\psi(x,\ell) \overline{\psi(x,m) }\psi(x,n) \, dx.
\end{align}
We can decompose it as
\begin{align}
\label{mudec}
\begin{split}
(2\pi)^2 \mu(k,\ell,m,n) = \mu_S(k,\ell,m,n) + \mu_L(k,\ell,m,n) + \mu_R(k,\ell,m,n)
\end{split}
\end{align}
where the following holds:

\setlength{\leftmargini}{1.5em}
\begin{itemize}

\item  We can write
\begin{align}
\label{muSdec}
\mu_S(k,\ell,m,n) = \mu_+(k,\ell,m,n) + \mu_-(k,\ell,m,n),
\end{align}
with
\begin{align}
\label{mu+-}
\begin{split}
& \mu_\pm(k,\ell,m,n) := \mathbf{1}_\mp(k,\ell,m,n)
  \sum_{\b,\g,\d,\eps \in \{-1,+1\}} (\beta\g\d\eps) \, \what{\varphi_\pm} (\b k - \g \ell + \d m - \eps n),
\\
& \mathbf{1}_\pm(k,\ell,m,n) = \mathbf{1}_\pm(k) \mathbf{1}_\pm(\ell) \mathbf{1}_\pm(m) \mathbf{1}_\pm(n),
\quad  
\qquad \varphi_\pm := \chi_\pm^4.
\end{split}
\end{align}

\medskip
\item We can write
\begin{align}
\label{muLdec}
 \mu_L(k,\ell,m,n) =  \mu_L^+(k,\ell,m,n) + \mu_L^-(k,\ell,m,n),
\end{align}
where
\begin{align}
\begin{split}
\label{muL+-}
& \mu_L^\pm(k,\ell,m,n) := \sum_{\b,\g,\d,\eps \in \{-1,+1\}} a^\pm_{\beta\g\d\eps}(k,\ell,m,n) \,
  \what{\varphi_\pm} (\b k - \g \ell + \d m - \eps n)
\end{split}
\end{align}
with coefficients $a^\pm_{\b\g\d\eps}$ satisfying
\begin{align}
\label{muLcoeff+-}
\begin{split}
& |a^\pm_{\b\g\d\eps}(k,\ell,m,n)| \lesssim  \min\big( 1, \max(|k|,|\ell|,|m|,|n|) \big),
\end{split}
\end{align}
Moreover, the coefficients $a^\pm_{\b\g\d\eps}$ tensorize in the sense explained in Remark \ref{RemmuL+-} below.

\medskip
\item The regular part $\mu_R$ has the following properties:
let $\theta_i \in \{0,1\}$, $i=1,\dots, 4$, with $\theta_1+\theta_2+\theta_3+\theta_4 \leq 3$, then
\begin{align}
\label{muRprop}
\begin{split}
& \big| \partial_k^{\theta_1} \partial_\ell^{\theta_2} \partial_m^{\theta_3} \partial_n^{\theta_4} \mu_R(k,\ell,m,n) \big|
\lesssim \min(|k|,1)^{1-\theta_1} \min(|\ell|,1)^{1-\theta_2} \min(|m|,1)^{1-\theta_3} \min(|n|,1)^{1-\theta_4}.
\end{split}
\end{align}

\end{itemize}

\end{prop}

We will use this Proposition to decompose
\begin{align}
\label{dtfdec00}
\begin{split}
i \partial_t \widetilde{f}(t,k) & = \frac{1}{4\pi^2} \big[\mathcal{N}_S + \mathcal{N}_L + \mathcal{N}_R\big],
\qquad \mathcal{N}_S = \mathcal{N_+}+\mathcal{N_-},
\\
\mathcal{N}_\ast(t,k) & :=
  \iiint e^{it (-k^2 + \ell^2 - m^2 + n^2)} \widetilde{f}(t,\ell) \overline{\widetilde{f}(t,m)} \widetilde{f}(t,n)
  \mu_\ast (k,\ell,m,n) \,d\ell dm dn.
\end{split}
\end{align}

The singular part $\mu_S$ is a linear combination of singular measures and has a very explicit form, which is very helpful to compute and obtain estimates. The particular structure and signs combination will be important to achieve some key cancellations.
The component $\mu_L$ is a also a linear combination of singular measures, but with coefficients that vanish at
low frequencies. Such vanishing gives additional gains that allow us to close weighted estimates.
Finally, the regular part $\mu_R$ is both smoother than the other components, and has gains at low frequencies.

\begin{proof}[Proof of Proposition \ref{muprop}]
From the definition of $(2\pi)^2 \mu$ we can write it as a sum of terms of the form
\begin{align}
\label{muterms}
\int \overline{\psi_A(x,k)}\psi_B(x,\ell)\overline{\psi_C(x,m)}\psi_D(x,n) \, dx,
\qquad A,B,C,D \in \{S,L,R\},
\end{align}
where we are using our main decomposition of $\psi$ in \eqref{psiS+-}-\eqref{psiR+-}.

\medskip \noindent
{\it The singular part $\mu_S$}.
When all the indexes $A,B,C,D = S$, and the frequencies $k,\ell,m,n$ have the same sign,
these terms give rise to $\mu_S = \mu_++\mu_-$ where
\begin{align}
\label{muS}
\begin{split}
\mu_\pm(k,\ell,m,n) & = \mathbf{1}_{\mp}(k,\ell,m,n) \int_\R \chi_\pm^4(x)  \overline{(e^{ikx} - e^{-ikx})} (e^{i\ell x} - e^{-i\ell x})
  \overline{(e^{imx} - e^{-imx})} (e^{inx} - e^{-inx})\,dx
\\ & = \mathbf{1}_{\mp}(k,\ell,m,n) \sum_{\b,\g,\d,\eps \in \{-1,+1\}} (\beta\g\d\eps) \, \what{\varphi_\pm} (\b k - \g \ell + \d m - \eps n),
\end{split}
\end{align}
having defined $\varphi_{\pm} = (\chi_{\pm})^4$, and with the equality understood in the sense of distributions.

\medskip \noindent
{\it The singular part $\mu_L$}.
This component arises from terms like \eqref{muterms} when (at least) one index is $L$, and the remaining ones (if any) are $S$,
and one has all $\chi_+(x)$ or all $\chi_-(x)$ contributions.
More precisely,
\begin{align}
\label{muL-}
\begin{split}
\mu^\pm_L(k,\ell,m,n) = \sum_{\b,\g,\d,\eps \in \{-1,+1\}} \int_\R \chi_\pm^4(x)  \, a^\pm_{\b\g\d\eps}(k,\ell,m,n) \,
  \overline{e^{\b ikx}} \cdot e^{\g i\ell x} \cdot \overline{e^{\d imx}} \cdot e^{\eps inx} \, dx,
\end{split}
\end{align}
which, for convenience, we write as
\begin{align}
\label{muL-1}
\begin{split}
\mu^\pm_L(k_0,k_1,k_2,k_3) = \sum_{\eps_0,\eps_1,\eps_2,\eps_3 \in \{-,+\}}
  \widehat{\varphi_{\pm}}(\eps_0k_0 - \eps_1k_1 + \eps_2k_2 - \eps_3k_3) \, a^\pm_{\eps_0\eps_1\eps_2\eps_3}(k_0,k_1,k_2,k_3),
\end{split}
\end{align}
recalling that $\varphi_\pm := \chi_\pm^4$, and with the coefficients $a^\pm_{\b\g\d\eps}$ described below.
Let us look at the coefficients in $\mu^-$. One has
\begin{align}
\label{muL-2}
\begin{split}
a^-_{\eps_0\eps_1\eps_2\eps_3}(k_0,k_1,k_2,k_3) = \prod_{j=0}^3 \bold{a}^-_{\eps_j}(k_j)
- (\eps_0\eps_1\eps_2\eps_3)\mathbf{1}_{+}(k_0,k_1,k_2,k_3),
\end{split}
\end{align}
where, recalling also the conjugation property \eqref{TRconj} for $T$ and $R_\pm$,
\begin{equation}
\label{a+-sym}
\bold{a}_{\epsilon_j}^-(k_j) =
\left\{
\begin{array}{llll}
1 & \mbox{if} \qquad \epsilon_j = + & & \mbox{and} \qquad k_j > 0,
\\
R_-((-1)^{j+1} k_j) & \mbox{if} \qquad \epsilon_j = - & & \mbox{and} \qquad k_j > 0,
\\
T((-1)^j k_j)& \mbox{if} \qquad \epsilon_j = + & & \mbox{and} \qquad k_j < 0,
\\
0 & \mbox{if} \qquad \epsilon_j = - & & \mbox{and} \qquad k_j < 0.
\end{array}
\right.
\end{equation}
In other words, we have
\begin{align}
\label{a+-sym2}
\bold{a}^-_+(k_j) = \mathbf{1}_+(k_j) + \mathbf{1}_-(k_j) T((-1)^{j}k_j), \qquad \bold{a}^-_-(k_j) = \mathbf{1}_+(k_j) R_-((-1)^{j+1}k_j)
\qquad j=0,1,2,3,
\end{align}
which leads to the formulas
\begin{align}
\label{a+-sym3}
\begin{split}
a^-_{++++}(k,\ell,m,n) & = [\mathbf{1}_+(k) + \mathbf{1}_-(k) T(k) ] [\mathbf{1}_+(\ell) + \mathbf{1}_-(\ell) T(-\ell) ]
  [\mathbf{1}_+(m) + \mathbf{1}_-(m) T(m)] \\ & \pushright{ \times [\mathbf{1}_+(n) + \mathbf{1}_-(n) T(-n)]
  - \mathbf{1}_{+}(k,\ell,m,n),}
\\
a^-_{+++-}(k,\ell,m,n) & = [\mathbf{1}_+(k) + \mathbf{1}_-(k) T(k)] [\mathbf{1}_+(\ell) + \mathbf{1}_-(\ell) T(-\ell) ]
  [\mathbf{1}_+(m) + \mathbf{1}_-(m) T(m)] \\ & \pushright{ \times \mathbf{1}_+(n)R_-(n)
  + \mathbf{1}_{+}(k,\ell,m,n),}
\\
a^-_{++-+}(k,\ell,m,n) & = [\mathbf{1}_+(k) + \mathbf{1}_-(k) T(k) ] [\mathbf{1}_+(\ell) + \mathbf{1}_-(\ell) T(-\ell) ]
\mathbf{1}_+(m) R_-(-m) \\ & \pushright{ \times [\mathbf{1}_+(n) + \mathbf{1}_-(n)T(-n)] + \mathbf{1}_{+}(k,\ell,m,n),}
\\
& \,\,\, \vdots
\\ \\
a^-_{---+}(k,\ell,m,n) & = \mathbf{1}_+(k)R_-(-k) \mathbf{1}_+(\ell)R_-(\ell)
  \mathbf{1}_+(m)R_-(-m) [\mathbf{1}_+(n) + \mathbf{1}_-(n) T(-n)] \\ & \pushright{+ \mathbf{1}_{+}(k,\ell,m,n),}
\\
a^-_{----}(k,\ell,m,n) & = [R_-(-k) R_-(\ell) R_-(-m) R_-(n) -1 ] \mathbf{1}_{+}(k,\ell,m,n).
\end{split}
\end{align}
Notice that the indicator functions subtracted off at the end of each expression are the contributions from $\mu_-$.
We have similar formulas for the coefficients $a^+_{\b\g\d\eps}(k,\ell,m,n)$:
\begin{align*}
\begin{split}
a^+_{\eps_0\eps_1\eps_2\eps_3}(k_0,k_1,k_2,k_3) = \prod_{j=0}^3 \bold{a}^+_{\eps_j}(k_j)
- (\eps_0\eps_1\eps_2\eps_3)\mathbf{1}_{-}(k_0,k_1,k_2,k_3),
\end{split}
\end{align*}
with
\begin{align}
\label{a+-symbis}
\bold{a}^+_+(k_j) = \mathbf{1}_+(k_j) T((-1)^{j+1}k_j) + \mathbf{1}_-(k_j) = \bold{a}^-_+(-k_j),
 \qquad \bold{a}^+_-(k_j) = \mathbf{1}_-(k_j) R_+((-1)^{j}k_j),
\end{align}
for $j=0,1,2,3$, so that expressions analogous to \eqref{a+-sym3} hold.

We now observe the following tensorization property
\begin{rem}\label{RemmuL+-}
Let us label the set $\{(R_{-}(\pm k) +1) \mathbf{1}_+(k), \, T(\pm k) \mathbf{1}_-(k), \,\mathbf{1}_+(k)\}$
as  $\{a_{i}(k)\}_{-2\leq i \leq 2}$, with $a_{0}(k)= 1$.
Then, directly using the formulas \eqref{psiL+-}, we can expand the coefficients as a sum of tensor products
\begin{equation}
\label{tensored}
a^-_{\eps_0\eps_1\eps_2\eps_3}(k_0,k_1,k_2,k_3)
  = \sum_{\sigma \in F}  C_{\sigma, \eps}^- \, a_{\sigma_{0}}(k_{0}) \cdots a_{\sigma_{3}}(k_{3})
\end{equation}
where $F$ is the set of all quadruples $\sigma = (\sigma_0,\sigma_1,\sigma_2,\sigma_3)$
in the set $\{-2,\dots ,2\}^4 \backslash{(0,0,0,0)}$,
and the coefficients $C_{\sigma, \eps}$ are harmless constants.
An analogous statement holds for $a^+$.
\end{rem}

From \eqref{tensored} and \eqref{TRsmallk0}, we also see that each term in the sum has at least one of the coefficients
$a_{\sigma_{i}}(k)$ vanishing at $k=0$ which gives us the first property in \eqref{muLcoeff+-}.


\medskip \noindent
{\it The regular part $\mu_R$}.
The regular part comes from terms of the form \eqref{muterms} when one of the indices $A,B,C,D$ is $R$,
or there are contributions from both $\chi_+$ and $\chi_-$.
More precisely, we can write
\begin{align}
\mu_R(k,\ell,m,n) = \mu_R^{(1)}(k,\ell,m,n) + \mu_R^{(2)}(k,\ell,m,n)
\end{align}
where, if we let $X_R = \{ (A_0,A_1,A_2,A_3) \, : \, \exists \, j=0,\dots 3 \, : \, A_j = R\}$,
\begin{align}
\label{muR01}
\mu_R^{(1)}(k,\ell,m,n) & := \sum_{(A,B,C,D) \in X_R} \int \overline{\psi_A(x,k)}\psi_B(x,\ell)\overline{\psi_C(x,m)}\psi_D(x,n) \, dx
\end{align}
and
\begin{align}
\label{muR02}
\begin{split}
\mu_R^{(2)}(k,\ell,m,n) := \sum_{A,B,C,D \in \{S,L\}} \int \overline{\psi_A(x,k)}\psi_B(x,\ell)\overline{\psi_C(x,m)}\psi_D(x,n) \, dx
\\ - \mu_S(k,\ell,m,n) - \mu_L(k,\ell,m,n).
\end{split}
\end{align}

To see the validity of \eqref{muRprop} recall the formulas  \eqref{psiS+-}-\eqref{psiR+-} and observe that, in view of \eqref{TRsmallk0} and Lemma \ref{lemcoeff},
\begin{align}
\label{estpsiSL}
|\psi_{S}(x,k)| \lesssim \min(|k||x|,1), \qquad |\psi_{L}(x,k)| \lesssim \min(|k|,1).
\end{align}
Moreover, in view of \eqref{defWpm}-\eqref{mgood} and $V \in L^1_{\gamma}$, 
we have
\begin{align}
\label{mgood'}
\chi_\pm(x)| \partial_k^s(m_\pm(x,k) -1)| \lesssim \frac{1}{\langle k\rangle}\mathcal{W}_\pm^{s+1}(x)
  \lesssim  \frac{1}{\langle k\rangle}  \frac{1}{\langle x\rangle^{\gamma-s-1}},
\end{align}
so that
\begin{align*}
& \Big| \chi_\pm(x) \big[ (m_\pm(x,\pm k) - 1)e^{ikx} + R_\pm(\mp k)(m_\pm(x,\mp k) - 1) e^{-ixk} \big] \Big|
\\
& \lesssim \chi_\pm(x) \big|m_\pm(x,\pm k) - m_\pm(x,\mp k)\big|
  + \chi_\pm(x) \big|m_\pm(x,\mp k) - 1\big| \, \big|e^{ikx}-e^{-ixk}\big|
  \\ & + \big|R_\pm(\mp k)+1\big| \, \chi_\pm(x)\big|m_\pm(x,\mp k) - 1\big|
\lesssim  \frac{|k|}{\langle k\rangle} \frac{1}{\langle x\rangle^{\gamma-2}}
\end{align*}
having used \eqref{mgood'} with $s=1$.
It then follows that
\begin{align}
 \label{estpsiR}
|\psi_{R}(x,k)| \lesssim  \frac{1}{\langle x\rangle^{\gamma-1}}  \frac{1}{\langle k\rangle}
  \min \big(1, |k|\langle x \rangle \big).
\end{align}
having used again the definition \eqref{psiR+-}, \eqref{mgood'}, and Lemma \ref{lemcoeff}.
Combining \eqref{muR01}, \eqref{estpsiSL} and \eqref{estpsiR} we see that
the first property in \eqref{muRprop} holds true for $\mu_R^{(1)}$, provided $\gamma > 6$.
The second property in \eqref{muRprop} can be obtained similarly
by differentiating \eqref{psiS+-}-\eqref{psiR+-}, noticing that each derivative costs a factor of $|x|$,
so that in particular
\begin{align*}
\begin{split}
& |\partial_k \psi_{S}(x,k)| + |\partial_k \psi_{L}(x,k)| \lesssim |x|,
\qquad |\partial_k \psi_{R}(x,k)| \lesssim \frac{1}{\langle x\rangle^{\gamma-3}},
\end{split}
\end{align*}
and using again \eqref{estpsiSL} and \eqref{estpsiR}.

The verification that \eqref{muRprop} also holds for $\mu_R^{(2)}$ can be done similarly
using again \eqref{estpsiSL}, \eqref{TRsmallk0} and Lemma \ref{lemcoeff},
and the fact that $\chi_+ \cdot \chi_-$ is compactly supported in a ball of radius 2, see \eqref{chi+-}.
More precisely, one can write \eqref{muR02} as a linear combination
\begin{align*}
\mu_R^{(2)}(k,\ell,m,n) = \sum_{j=1,2,3} \sum_{\b,\g,\d,\eps \in \{-1,+1\}}
  \int_\R \chi_-^j(x)\chi_+^{4-j}(x) \, b^{j}_{\beta\g\d\eps}(k,\ell,m,n) \, e^{ix(-\b k + \g \ell - \d m + \eps n)} \, dx,
\end{align*}
for some suitable coefficients $b^{j}_{\beta\g\d\eps}$ and estimate
\begin{align*}
|\mu_R^{(2)}(k,\ell,m,n)|  \lesssim \sum_{j=1,2,3} \int_\R \chi_-^j(x)\chi_+^{4-j}(x) \min(|k|,1)\min(|\ell|,1)\min(|m|,1)\min(|n|,1) \, dx.
\end{align*}
The second bound in \eqref{muRprop} can also be obtained similarly.
\end{proof}

\bigskip
\section{Weighted estimate}\label{secw}
 The aim of this section is to prove the following proposition.
 \begin{prop}
\label{weighted}
Under the assumptions of Theorem \ref{maintheo},
consider $u$, solution of \eqref{NLS} satisfying \eqref{bootstrap}, then, there exists $C>0$ such that  we have
$$  \langle t  \rangle^{- {1 \over 4}+ \alpha}\| \partial_{k} \widetilde{f}(t) \|_{L^2}\leq
C( \veps_{0}+  \veps_{1}^3), \quad  \forall t \geq 0.$$
\end{prop}
The remaining of this section is devoted to the proof of this proposition.

Recall the equation
\begin{align}
\label{duhamel0}
i \partial_t \widetilde{f}(t,k) & = \iiint e^{it (-k^2 + \ell^2 - m^2 + n^2)}
  \widetilde{f}(t,\ell) \overline{\widetilde{f}(t,m)} \widetilde{f}(t,n) \mu(k,\ell,m,n) \, d\ell \, dm\, dn := \mathcal{N}(t,k),
\end{align}
with
\begin{align}
\label{mu0}
\mu(k,\ell,m,n) & = \int \overline{\psi(x,k) }\psi(x,\ell) \overline{\psi(x,m) }\psi(x,n) \, dx.
\end{align}

We use Proposition \ref{muprop} to decompose
\begin{align}
\label{dtfdec}
\begin{split}
& i \partial_t \widetilde{f}(t,k) = \mathcal{N}_+ + \mathcal{N}_- + \mathcal{N}_L^+ + \mathcal{N}_L^- + \mathcal{N}_R,
\\
& \mathcal{N}_\ast(t,k) = \frac{1}{(2 \pi)^2}
  \iiint e^{it (-k^2 + \ell^2 - m^2 + n^2)} \widetilde{f}(t,\ell) \overline{\widetilde{f}(t,m)} \widetilde{f}(t,n)
  \mu_\ast (k,\ell,m,n) \,d\ell\, dm \,dn
\end{split}
\end{align}
and move on to prove the desired weighted bound for each term.

\medskip
\subsection{Estimate for $\mathcal{N}_\pm$}\label{secwN+-}
We shall prove that
\begin{equation}
\label{N+-borne}
\| \partial_k \mathcal{N}_{\pm}(t) \|_{L^2} \lesssim { \veps_1^3} \langle t \rangle^{{1 \over 4} - \alpha}.
\end{equation}
Let us concentrate on the case $k>0$, that is on $\mathcal{N}_-$; the case $k<0$ is of course analogous.

By the choice \eqref{chi+-} of $\chi_{-}$, $ \partial_{x}\varphi_{-}$ as defined in \eqref{mu+-},
is a $\mathcal{C}^\infty_{c}$ function,
which we can write as $ \partial_{x} \varphi_{-}= \phi^o - \phi$, where $\phi^o$ and $\phi$
are respectively odd and even and $\mathcal{C}^\infty_{c}$.
Furthermore, since $\phi^o$ is odd,  we can write $\phi^o =  \partial_{x} \psi$ where $\psi \in \mathcal{C}^\infty_{c}$ and $\psi$ is even.
We have thus obtained that
\begin{align*}
\varphi_{-}= \psi + \int_{x}^{+ \infty} \phi(y)\, dy = \psi + \phi \ast \mathbf{1}_-, \qquad \int_\R\phi(y)\,dy =1,
\end{align*}
where we denoted $\mathbf{1}_\pm = (1\pm H)/2$ the characteristic function of $\{\pm x > 0\}$.
Taking the Fourier transform, and using the classical formulas
\begin{align}
 \label{Fsign}
\what{f \ast g} = \sqrt{2\pi} \what{f} \cdot \what{g}, \qquad \what{1} = \sqrt{2\pi} \delta_0, \qquad
\what{\sign x} = \sqrt{\frac{2}{\pi}}\frac{1}{ik},
\end{align}
we see that $\widehat{\mathbf{1}_-} = \sqrt{\frac{\pi}{2}} \delta - \frac{1}{\sqrt{2\pi}} \frac{1}{ik}$, and therefore
\begin{equation*}
\what{\varphi_-} - \what{\psi} = \widehat{\mathcal{F}} \big(\phi \ast \mathbf{1}_- \big)
  = \sqrt{2\pi}  \what{\mathbf{1}_-}(k) \what{\phi}(k)
  = \sqrt{\frac{\pi}{2}} \delta_0 - \frac{\what{\phi}(k)}{ik}.
\end{equation*}
A similar formula can be obtained for $\varphi_+$. Let us  record these formulas:
\begin{equation}
\label{decphi-}
\widehat{\varphi}_{-}(k) = \sqrt{\frac{\pi}{2}} \delta - {\widehat{\phi}(k)\over ik} + \widehat{\psi}(k)
\qquad \mbox{and} \qquad \widehat{\varphi}_{+}(k) = \sqrt{\frac{\pi}{2}} \delta +  {\widehat{\phi}(k)\over ik} + \widehat{\psi}(k),
\end{equation}
where $\phi \in C_c^\infty$ is even and has integral $1$, and we slightly abuse notation by denoting with the same letter $\psi$
a generic $C_c^\infty$ even function.
Then, we define
\begin{align}
\label{f_pm}
\wt{f}_\pm(k) = \wt{f}(k) \cdot \mathbf{1}_\pm(k), \qquad \wt{u}_\pm(k) = e^{itk^2} \wt{f}_\pm(k),
\end{align}
and write 
\begin{align}
\mathcal{N}_-(t,k) &= \mathcal{N}_0(t,k) + \mathcal{N}_V(t,k) + \mathcal{N}_{V,r}
\end{align}
where
\begin{align}
\begin{split}
\mathcal{N}_0(t,k) = \sqrt{\frac{\pi}{2}} \sum_{\b,\g,\d,\eps \in \{-1,+1\}} \b\g\d\eps
  \iint e^{it (-k^2 + (\b k+\d m-\eps n)^2 - m^2 + n^2)} \mathbf{1}_+(k) 
  \\ \times \widetilde{f}_+(t,\g(\b k + \d m - \eps n)) \overline{\widetilde{f}_+(t,m)} \widetilde{f}_+(t,n) \,dm\, dn,
\end{split}
\end{align}
\begin{align}
\begin{split}
\mathcal{N}_V(t,k) = i \sum_{\b,\g,\d,\eps \in \{-1,+1\}} \b\g\d\eps
  \iiint e^{it (-k^2 + (\b k-p+\d m-\eps n)^2 - m^2 + n^2)} \mathbf{1}_+(k) 
  \\ \times \widetilde{f}_+(t,\gamma(\b k-p+\d m-\eps n)) \overline{\widetilde{f}_+(t,m)} \widetilde{f}_+(t,n)
  \frac{\what{\phi}(p)}{p} \, dm\, dn\, dp,
\end{split}
\end{align}
and
\begin{align}
\label{NVr}
\begin{split}
\mathcal{N}_{V,r}(t,k) = \sum_{\b,\g,\d,\eps \in \{-1,+1\}} \b\g\d\eps
  \iiint e^{it (-k^2 + (\b k-p+\d m-\eps n)^2 - m^2 + n^2)} \mathbf{1}_+(k) 
  \\ \times \widetilde{f}_+(t,\gamma(\b k-p+\d m-\eps n)) \overline{\widetilde{f}_+(t,m)} \widetilde{f}_+(t,n)
  \what{\psi(p)} \, dm\, dn\, dp,
\end{split}
\end{align}
having changed variables from $\ell$ to $p = \b k - \g \ell + \d m - \eps n$ in the last two terms.
The term $\mathcal{N}_0$ essentially corresponds to the flat NLS, i.e., the case $V=0$.

\subsubsection{The term $\mathcal{N}_0$}
Changing variables $(m,n) \rightarrow (a,b)$ by letting
$$
\left\{
\begin{array}{l}
m = \delta(- a + b +\beta k) \\
n = \eps(\beta k + b)
\end{array}
\right.
\quad
\mbox{i.e.}
\quad
\left\{
\begin{array}{l}
a = \epsilon n - \delta m \\
b = \epsilon n - \beta k,
\end{array}
\right.
$$
we have
\begin{align}
\label{wN_0}
\begin{split}
\mathcal{N}_0(t,k) = \sqrt{\frac{\pi}{2}} \sum_{\b,\g,\d,\eps \in \{-1,+1\}} \b\g\d\eps
  \iint e^{2it ab} \mathbf{1}_+(k) 
  \wt{f}_+(t,\gamma(\beta k-a)) \overline{\wt{f}_+(t,\delta(b-a+\beta k))} \\ \times \wt{f}_+(t,\eps(b+\beta k)) \, da \, db.
\end{split}
\end{align}
This is analogous to the case of flat cubic NLS where, due to the gauge invariance, the derivative $\partial_k$ simply distributes on the three profiles.
Moreover, let us recall that $\partial_{k} (\widetilde f_{+})=( \partial_{k} \widetilde f )\mathbf{1}_{l \geq 0} : = (\partial_{k} \widetilde f)_{+}$ since $ \widetilde f_{+}(0)= 0$  and let us notice  that the contribution
occurring when $\partial_k$ hits $\mathbf{1}_+(k)$
also  vanishes due to a cancellation. Indeed, we observe that
$$\mathcal{N}_0(t,k) = \sqrt{\frac{\pi}{2}} \mathbf{1}_+(k)  \sum_{\beta \in \{-1, \, 1\}} \beta\, I (t,\beta k),$$
with
\begin{equation}
\label{Idef}
I(t,y) = \sum_{\g,\d,\eps \in \{-1,+1\}} \g\d\eps
 \iint e^{2it ab} \wt{f}_+(t,\gamma(y-a)) \overline{\wt{f}_+(t,\delta(b-a+y))} \wt{f}_+(t,\eps(b+y)) \, da\, db
\end{equation}
and hence that
$   \sum_{\beta \in \{-1, \, 1\}} \beta I (0)=0$.
 Consequently, we have that
 $$ \partial_{k}\mathcal{N}_{0}(t,k)=  \sqrt{\frac{\pi}{2}} \mathbf{1}_{+}(k)  \sum_{\beta \in \{-1, \, 1\}}   \partial_{y} I (t,\beta k).$$
 After redistributing the phases,  we  obtain  that $\partial_{y}I (t, \beta k)$ can be written as a sum of terms of the type
$$ \iint e^{it (-k^2 + (\b k+\d m-\eps n)^2 - m^2 + n^2)} \mathbf{1}_+(k) 
   \partial_{k}\widetilde{f}_+(t,\g(\b k + \d m - \eps n)) \overline{\widetilde{f}_+(t,m)} \widetilde{f}_+(t,n) \,dm \, dn.$$
   The above term can be written as
   $$  \widehat{\mathcal{F}}  \left[e^{it \partial_{x}^2} \mathbf{1}_{+}( D)\left(
    e^{-it \partial_{x}^2} (\widehat{\mathcal{F}}^{-1} \partial_{k}\widetilde f_{+})\left(\gamma \cdot\right)
     \overline{ e^{-it \partial_{x}^2} (\widehat{\mathcal{F}}^{-1} \widetilde f_{+})\left( \delta \cdot \right)}
      e^{-it \partial_{x}^2} (\widehat{\mathcal{F}}^{-1} \widetilde f_{+})(\eps\, \cdot)\right) \right](\beta k)$$
This yields the estimate
\begin{align}
\label{dkI}
{\| \partial_k I (t) \|}_{L^2} \lesssim {\| \widehat{\mathcal{F}}^{-1} \partial_k \wt{f}_{+} \|}_{L^2}
    {\| e^{-it \partial_{x}^2} \widehat{\mathcal{F}}^{-1}(\mathbf{1}_+(k) \wt{f}) \|}_{L^\infty}^2.
\end{align}
Hence,  by using the  (flat) linear estimate of Corollary \ref{cor_penguin} to deduce that
\begin{align}
\label{flatbis}
{\| e^{-it \partial_{x}^2} \widehat{\mathcal{F}}^{-1}(\mathbf{1}_+(k) \wt{f}) \|}_{L^\infty}
    \lesssim \frac{1}{\sqrt{t}} {\| \wt{f}(t) \|}_{L^\infty} + \frac{1}{t^{3/4}} {\| \partial_k \wt{f}(t) \|}_{L^2},
\end{align}
we finally obtain by using the bootstrap assumption that
\begin{equation}
\label{dkN0}
{\| \partial_k \mathcal{N}_0(t) \|}_{L^2} \lesssim \| \partial_{k} I(t) \|  \lesssim \e_1^3 \langle t \rangle^{- { 3 \over 4} - \alpha}.
\end{equation}
Note that by using the above arguments, we have since
$$ I(t,y)= \widehat{\mathcal{F}}  \left[e^{it \partial_{x}^2} \mathbf{1}_{+}(D)\left(
    e^{-it \partial_{x}^2} (\widehat{\mathcal{F}}^{-1}\widetilde f_{+})\left(\gamma\cdot\right)
     \overline{ e^{-it \partial_{x}^2} (\widehat{\mathcal{F}}^{-1} \widetilde f_{+})\left(\delta\cdot \right)}
      e^{-it \partial_{x}^2} (\widehat{\mathcal{F}}^{-1} \widetilde f_{+})(\eps\, \cdot)\right) \right](\beta y)
      $$
  that
 \begin{equation}
 \label{Idet}
 \| I(t) \|_{L^2} \lesssim   {\| \wt{f}_{+} \|}_{L^2}
    {\| e^{-it \partial_{x}^2} \widehat{\mathcal{F}}^{-1}(\mathbf{1}_+(k) \wt{f}) \|}_{L^\infty}^2  \lesssim{ \e_1^3 \over t}.
 \end{equation}

\subsubsection{The term $\mathcal{N}_V$}
Changing variables $(m,n) \rightarrow (a,b)$ by letting
$$
\left\{
\begin{array}{l}
m = \delta(- a + b -p +\beta k) \\
n = \eps(\beta k - p + b)
\end{array}
\right.
\quad
\mbox{i.e.}
\quad
\left\{
\begin{array}{l}
a = \epsilon n - \delta m \\
b = p + \epsilon n - \beta k,
\end{array}
\right.
$$
we can write
\begin{align}
\begin{split}
& \mathcal{N}_V(t,k) =  \sum_{\b \in \{-1,+1\}} \b
  \int e^{it (-k^2 + (p-\beta k)^2)}\, \mathbf{1}_+(k) \, I (t,\beta k - p) \, \frac{\what{\phi}(p)}{p} \, dp,
\end{split}
\end{align}
where $I(t,y)$ is defined in \eqref{Idef}.
By setting $q= p-\beta k$, we can also write that
$$
\mathcal{N}_V(t,k) =  \sum_{\b \in \{-1,+1\}} \b
  \int e^{it (-k^2 + q^2)}\, \mathbf{1}_+(k) \, I (t, -q) \, \frac{\what{\phi}(q+ \beta k)}{q + \beta k} \, dq$$
and we observe,  first changing variable $\gamma \to -\gamma$,
$\delta \to -\delta$, $\epsilon \to -\epsilon$, and then $a \to - a$ and $b \to -b$, that:
\begin{align}
\label{goodsym}
\begin{split}
I(q) & = \sum_{\g,\d,\eps \in \{-1,+1\}} \g\d\eps \iint e^{2it ab} \wt{f}_+(\gamma(q-a)) \overline{\wt{f}_+(\delta(b-a+q))} \wt{f}_+(\eps(b+q)) \, da db
\\
& = - \sum_{\g,\d,\eps \in \{-1,+1\}} \g\d\eps \iint e^{2it ab}
\wt{f}_+(\gamma(-q+ a)) \overline{\wt{f}_+(\delta(-b +a-q))} \wt{f}_+(\eps(-b-q)) \, da db
\\
& = - I(-q).
\end{split}
\end{align}
By using this symmetry property, we find that
$$\mathcal{N}_V(t,k) =  - {1 \over 2 } \sum_{\b  \in \{-1,+1\}} \b
  \int e^{it (-k^2 + q^2)}\, \mathbf{1}_+(k) \, I(t,q) \, \left ( \frac{\what{\phi}(q+ \beta k)}{q + \beta k}
  -\frac{\what{\phi}(-q+ \beta k)}{-q + \beta k}\right)
   \, dq$$
and by writing out explicitly the terms corresponding to $\beta = 1$ and $-1$ we finally  get
\begin{align}
\begin{split}
\mathcal{N}_V(t,k) & = -  {1 \over 2}  \int e^{it (-k^2 + q^2)} \,\mathbf{1}_{+}(k) I(t,q)
  \\ & \times \Big[ {\what{\phi}(q+k) \over q+k} - {\what{\phi}(-q+k) \over -q+ k} - {\what{\phi}(q-k)
   \over q-k} +{ \what{\phi}(-q-k)  \over -q -k}\Big] \, dq.
\end{split}
\end{align}
 Since $\phi$ is even,
 this yields $\mathcal{N}_V(t,k) \equiv 0$.

\subsubsection{The term $\mathcal{N}_{V,r}$}\label{NVrsec}
As above, we can write
\begin{align}
\label{NVr2}
\mathcal{N}_{V, r}(t,k) =  \sum_{\b, \in \{-1,+1\}} \b
  \int e^{it (-k^2 + (p-\beta k)^2)}\, \mathbf{1}_+(k) \, I(t,\beta k- p) \, \what{\psi}(p) \, dp
\end{align}
where now $\widehat{\psi}$ is even (above $\widehat{\phi}(p)/p$ was odd) and in the Schwartz class.
 By computing $\partial_{k}$, we find
 \begin{align}
& \partial_k \mathcal{N}_{V, r} = \mathcal{N}_{1} + \mathcal{N}_{2}
\end{align}
where
\begin{align}
\begin{split}
\mathcal{N}_1(t,k) & =   -2 i t  \mathbf{1}_{+}(k)\sum_{\b  \in \{-1,+1\}}  \int e^{it (-k^2 + q^2)}
  \, I (t,q) \, (q+ \beta k) \what{\psi}(q+ \beta k) \, dq,
\\
\mathcal{N}_2(t,k) & =  -\mathbf{1}_{+}(k) \sum_{\b  \in \{-1,+1\}}  \int e^{it (-k^2 + (p-\b k)^2)} \,
  \partial_y I (t,p-\beta k) \, \what{\psi}(p) \, dp,
\end{split}
\end{align}
having changed variables to $q = p - \beta k$ for the first term.

Let us start with the estimate of $\mathcal{N}_{2}$. We first observe that
since $\psi$ is a Schwartz class function, we obtain  from the Young inequality that
   $$ \|\mathcal{N}_2\|_{L^2} \lesssim   \|\partial_{y} I \|_{L^2}$$
and hence, by using \eqref{dkN0}, we find
$$\|  \mathcal{N}_2(t) \|_{L^2} \lesssim \e_1^3 \langle t \rangle^{- { 3 \over 4} - \alpha}.$$

To handle $\mathcal{N}_{1}$, we shall integrate by parts in $q$ using that
 $ {1 \over q}\partial_{q} (e^{it q^2})= 2i t  e^{it q^2}.$
  This yields
 \begin{align*}
 \begin{split}  \mathcal{N}_{1}(t,k) &  =
  \mathbf{1}_{+}(k)\sum_{\b  \in \{-1,+1\}}  \int e^{it (-k^2 + q^2)}
  \,  {I (t,q) \over q} \, \psi_{1}(q+ \beta k)    \, dq
  \\
  & \mbox{\hspace{4cm}}  + \mbox{p.v. }\int e^{it (-k^2 + q^2)}
  \,\partial_{q} \left( {I (t,q) \over q} \right) \,  \psi_{2}(q+ \beta k)\, dq \\
  & = \mathcal{N}_{1, 1}+ \mathcal{N}_{1,2}
 \end{split}
  \end{align*}
  where
  $$ \psi_{1}(y)=  \what{\psi}(y) + y \partial_{y}  \what{\psi}(y), \quad \psi_{2}(y)
  =y \,\what{\psi}(y).$$
  The above integration by parts can be justified by  integrating by parts for $|q| \geq \eps >0$ \texttt{}and passing to the limit $\eps\rightarrow 0$.
  Indeed, since $I(t,q)$ is an odd function thanks to \eqref{goodsym}, we observe that
   the boundary term
  $$ e^{i t( -k^2 + \eps^2)}\left( {I (t, \eps) \over \eps} \what{\psi}(\eps + \beta k) - {I(t, - \eps) \over - \eps} \what{\psi}(- \eps + \beta k)
  \right) = e^{i t( -k^2 + \eps^2)}  {I (t, \eps) \over \eps} \left( \what{\psi}( \eps + \beta k) - \what{\psi} (- \eps + \beta k)\right)$$
  tends to zero when $\eps$ tends to zero.
 Since $\psi_{1}$ is in the Schwartz class, we get as before
 $$ \| \mathcal{N}_{1, 1}\|_{L^2} \lesssim \| I(t, q)/q \|_{L^2}.$$
 Next, again since $I(t,0)= 0$, we can
 use  the Hardy inequality and \eqref{dkN0} to get that
  $$  \| \mathcal{N}_{1, 1}\|_{L^2} \lesssim  \| \partial_{q} I (t)\|_{L^2}
   \lesssim   \eps_1^3 \langle t \rangle^{- { 3 \over 4} - \alpha}.$$
   For the second term, we can symmetrize by using that the function $\partial_{q}(I/q)$ is odd, to obtain
  \begin{multline*}
   \mathcal{N}_{1,2}  ={1 \over 2}  \mathbf{1}_{+}(k)\sum_{\b  \in \{-1,+1\}} \mbox{p.v. }\int e^{it (-k^2 + q^2)}
  \,\partial_{q} \left( {I (t,q) \over q} \right) \,\left(   \psi_{2}(q+ \beta k) - \psi_{2}(-q+ \beta k) \right)\, dq \\
 =  {1 \over 2}  \mathbf{1}_{+}(k)\sum_{\b  \in \{-1,+1\}} \mbox{p.v. }\int e^{it (-k^2 + q^2)}
  \, \left(  \partial_{q} I (t,q) - { I(t,q) \over q} \right) \,\left(  { \psi_{2}(q+ \beta k) - \psi_{2}(-q+ \beta k) \over q} \right)\, dq
  \end{multline*}
  Again, since $\psi_{2}$ is a Schwartz class function,  we have that
  $$ \sup_{k} \int_{\mathbb{R}}  \left|{ \psi_{2}(q+ \beta k) - \psi_{2}(-q+ \beta k) \over q} \right|\, dq
   + \sup_{q} \int_{\mathbb{R}} \left|{ \psi_{2}(q+ \beta k) - \psi_{2}(-q+ \beta k) \over q} \right|\, dk <+ \infty$$
   and therefore, we obtain that
  $$  \| \mathcal{N}_{1, 2}(t) \|_{L^2} \lesssim \| \partial_k I(t) \|_{L^2}+\left\|  { I(t, k)  \over k } \right\|_{L^2} \lesssim
   \e_1^3 \langle t \rangle^{- { 3 \over 4} - \alpha}
  $$
  by using again the Hardy inequality and \eqref{dkN0}.

  We have thus obtained that
  $$ \| \partial_{k} \mathcal{N}_{V,r}\|_{L^2} \lesssim   \e_1^3 \langle t \rangle^{- { 3 \over 4} - \alpha}.$$

Gathering all the above estimates, we find \eqref{N+-borne}.

\medskip
\subsection{Estimate for $\mathcal{N}_L^\pm$}

As before, we only treat $\mathcal{N}_L^-$.  By \eqref{decphi-}, we can write
\begin{align}
\label{NLdec}
\mathcal{N}_L^-(t,k) & = \mathcal{N}_{L,0}(t,k) + \mathcal{N}_{L,V}(t,k)+ \mathcal{N}_{L, V,r}
\end{align}
where
\begin{align*}
\begin{split}
\mathcal{N}_{L,0}(t,k) = \sqrt{\frac{\pi}{2}} \sum_{\b,\g,\d,\eps \in \{-1,+1\}} \iint a^-_{\beta\g\d\eps}(k,\g(\b k+\d m-\eps n),m,n)
  e^{it (-k^2 + (\b k+\d m-\eps n)^2 - m^2 + n^2)}
  \\ \times 
  \widetilde{f}(t,\g(\b k+\d m-\eps n)) \overline{\widetilde{f}(t,m)} \widetilde{f}(t,n) \,dm\, dn,
\end{split}
\end{align*}
\begin{align*}
\begin{split}
\mathcal{N}_{L,V}(t,k) =  \sum_{\b,\g,\d,\eps \in \{-1,+1\}} \iiint  a^-_{\beta\g\d\eps}(k,\gamma(\b k-p+\d m-\eps n),m,n)
 e^{it (-k^2 + (\b k-p+\d m-\eps n)^2 - m^2 + n^2)}
  \\ \times
 \widetilde{f}(t,\gamma(\b k-p+\d m-\eps n)) \overline{\widetilde{f}(t,m)} \widetilde{f}(t,n)
  \frac{\what{\phi}(p)}{p} \, dm \,dn\, dp,
\end{split}
\end{align*}
and
\begin{align}
\label{NLvr}
\begin{split}
\mathcal{N}_{L,V, r}(t,k) =  \sum_{\b,\g,\d,\eps \in \{-1,+1\}} \iiint  a^-_{\beta\g\d\eps}(k,\gamma(\b k-p+\d m-\eps n),m,n)
 e^{it (-k^2 + (\b k-p+\d m-\eps n)^2 - m^2 + n^2)}
  \\ \times
 \widetilde{f}(t,\gamma(\b k-p+\d m-\eps n)) \overline{\widetilde{f}(t,m)} \widetilde{f}(t,n)
   \what{\psi}(p) \, dm\, dn\, dp.
\end{split}
\end{align}

\subsubsection{The $\mathcal{N}_{L,0}$ contribution} \label{nlzero}
This is similar to the term $\mathcal{N}_0$ in \eqref{wN_0}. Indeed, by using the expansion \eqref{tensored} of the symbols $a_-$, the problem reduces to estimating terms of the form
\begin{align}
\begin{split}
a_{\sigma_{1}}(\beta k) \iint   e^{it (-k^2 + (\b k+\d m-\eps n)^2 - m^2 + n^2)}
 g_{\sigma_{2}}(t,\g(\b k+\d m-\eps n)) g_{\sigma_{3}}(t,m) g_{\sigma_{4}}(t,n) \,dm \,dn,
\end{split}
\end{align}
where we have set
\begin{equation}
\label{defg} g_{\sigma_{i}}(t, k)= a_{\sigma_{i}}(k) \widetilde{f}(t,k).
\end{equation}
The bounds on $a_{\sigma_i}$ as well as the bootstrap assumption on $f$ imply that
 $$ \|\widehat{\mathcal{F}}^{-1} e^{-it k^2} g_{\sigma_{i}}(t)\|_{L^\infty} \lesssim  \frac{\e_1}{\sqrt t}, \quad \| \partial_{k} g_{\sigma_{i}}(t) \|_{L^2}
  \lesssim\e_1 \langle t \rangle^{\frac{1}{4}-\alpha},$$
  and  the estimates follow exactly as above, giving
  $$ \|\partial_{k} \mathcal{N}_{L,0} (t) \|_{L^2} \lesssim \veps_1^3 \langle t \rangle^{- { 3 \over 4} - \alpha}.$$

\subsubsection{The $\mathcal{N}_{L,V}$ contribution}
The main idea here is to use the vanishing of the $a^-$ coefficients, see  \eqref{muLcoeff+-},
in order to perform various integration by parts.
We begin by changing variables as we did before: $(m,n) \to (a,b)$ with letting $(m,n) = (\d (-a+b-p+\b k),\eps (\b k - p + b))$ so that
\begin{align}
\begin{split}
\mathcal{N}_{L,V}(t,k) = i \sum_{\b,\g,\d,\eps \in \{-1,+1\}}
  \int e^{it (-k^2 + (p-\b k)^2)}  I_{\beta\g\d\eps}(\beta k-p) \frac{\what{\phi}(p)}{p} \, dp,
\\
\end{split}
\end{align}
where
\begin{align}
\begin{split}
I_{\beta\g\d\eps}(y) & = \iint e^{2itab} a^-_{\beta\g\d\eps}(k,\gamma(y-a),\delta(b-a+y), \eps(y+ b))
  \\ & \times \widetilde{f}(t,\gamma(y-a)) \overline{\widetilde{f}(t,\delta(b-a+y)} \widetilde{f}(t,\eps(y+ b))\, da \, db.
\end{split}
\end{align}
Applying $\partial_k$ gives two types of terms:
\begin{align}
\label{dkNLV}
\begin{split}
& \partial_k \mathcal{N}_{L,V} = \mathcal{N}_{L,1} + \mathcal{N}_{L,2},
\\
& \mathcal{N}_{L,1} = 2t \sum_{\b,\g,\d,\eps \in \{-1,+1\}} \beta
  \int e^{it (-k^2 + q^2)} \, I_{\beta\g\d\eps}(-q) \, \what{\phi}(q+\beta k)\,dq,
\\
& \mathcal{N}_{L,2} =  i\sum_{\b,\g,\d,\eps \in \{-1,+1\}}\beta
  \int e^{it (-k^2 + (p-\b k)^2)} \, \partial_y I_{\beta\g\d\eps}(p-\beta k) \, \frac{\what{\phi}(p)}{p} \,dp.
\end{split}
\end{align}

\medskip
\noindent
{\it The term $\mathcal{N}_{L,2}$}. We start with this term, which can be easily bounded. Proceeding as in Section~\ref{nlzero}, we observe that $I_{\beta \gamma \delta \epsilon}$ can be written as $I(t,y)$ in~\eqref{Idef} if one replaces $\widetilde{f}$ by $g_\sigma$. By the boundedness properties of $a_\sigma(D)$ exploited in Section~\ref{nlzero}, we can follow the argument used when estimating $I(t,y)$ to deduce the equivalent of~\eqref{dkN0}, namely
$$
\| \partial_k I_{\beta \gamma \delta \epsilon} \|_{L^2} \lesssim t^{-\frac{3}{4}-\alpha} \e_1^3.
$$
Now observe that
$$
\mathcal{N}_{L,2} = \sum_{\beta,\gamma,\delta,\epsilon}\widehat{ \mathcal{F}} e^{it \partial_x^2} \left[ e^{-it\partial_x^2} \widehat{ \mathcal{F}}^{-1} (\partial_k I_{\beta \gamma \delta \epsilon}) \widehat{\mathcal{F}}^{-1}  \frac{\widehat{\phi}(k)}{k}\right].
$$
Since $\widehat{\mathcal{F}}^{-1}  \frac{\widehat{\phi}(k)}{k}$ is a bounded function, we obtain the desired estimate:
\begin{align*}
\| \mathcal{N}_{L,2} \|_{L^2} \lesssim \| \partial_k I_{\beta \gamma \delta \epsilon} \|_{L^2} \lesssim t^{-\frac{3}{4}-\alpha} \e_1^3.
\end{align*}

\medskip
\noindent
{\it The term $\mathcal{N}_{L,1}$}.
This is the term where we exploit the vanishing of the $\mu_L$ part of the spectral measure, see  \eqref{muLcoeff+-}. The desired bound will be achieved if we can show (for any choice of $\beta$, $\gamma$, $\delta$, and $\epsilon$)
\begin{align}
\label{NL10}
{\Big\| \int_0^t \mathcal{M}(s,k) \, ds \Big\|}_{L^2_k} \lesssim \e_1^3 \langle t \rangle^{1/4-\alpha},
\end{align}
where
$$
\mathcal{M}(t,k) = t \, \iiint e^{it (-k^2 + q^2 + 2ab)} \widetilde{f}(t,\gamma(-q-a)) \overline{\widetilde{f}(t,\delta(b-a-q)} \widetilde{f}(t,\eps(-q+ b)) \mu(k,a,b,q) \,da\,db\,dq
$$
with
$$
\mu(k,a,b,q) = a^-_{\beta\g\d\eps}(k,\gamma(-q-a),\delta(b-a-q), \eps(-q+ b)) \widehat{\phi}(q+\beta k).
$$
Using the notation for Littlewood-Paley cutoffs from \ref{secLP},
we decompose dyadically with respect to the output variable $k$, and the maximum of the input variables.
More precisely, we decompose $\mathcal{M} = \sum_{K,J} \mathcal{M}_{K,J}(t,k)$ by setting
\begin{align}
\label{NL12}
\begin{split}
\mathcal{M}_{K,J}(k) := t \, \iiint e^{it (-k^2 + q^2 + 2ab)} \,\widetilde{f}(\gamma(-q-a)) \overline{\widetilde{f}(\delta(b-a-q)} \widetilde{f}(\eps(-q+ b)) \mu_{K,J}(k,a,b,q) \, da \,db\, dq,
\end{split}
\end{align}
with
\begin{equation}
\begin{split}
\label{NL12'}
\mu_{K,J}(k,a,b,q) & :=  a^-_{\beta\g\d\eps}(k,\gamma(-q-a),\delta(b-a-q), \eps(-q+ b)) \widehat{\phi}(q+\beta k) \\
& \qquad \qquad \qquad \qquad \times \varphi_K(k) \varphi_J\big( |(q+a,b-a-q,q-b)| \big)
\end{split}
\end{equation}
We then distinguish two main cases depending on the relative sizes of $J$ and $K$ by splitting
\begin{align}
\label{NL1split}
\begin{split}
& \mathcal{M} = \mathcal{M}_1 + \mathcal{M}_2,
\qquad \mathcal{M}_1 := \sum_{J \geq K-10} \mathcal{M}_{K,J},
\qquad \mathcal{M}_2 := \sum_{J < K-10} \mathcal{M}_{K,J}.
\end{split}
\end{align}
The first term corresponds to the case when the maximum of three input variables is larger or comparable to the output frequency $k$,
while in the term $\mathcal{M}_2$ the frequency $k$ is dominant.

\medskip
\noindent
{\it Case 1: Estimate of $\mathcal{M}_1$.}
We begin by treating the case when $k$ is not the dominant frequency and distinguish several subcases.
Note that since $K \leq J+10$ we have
\begin{align*}
\mu_{K,J}(k,a,b,q) = \mu_{K,J}(k,a,b,q) \varphi_{\leq J+20}(q+\beta k).
\end{align*}

\medskip
\noindent
{\it Subcase 1.0: Small times $t \leq 1$.} It is easy to see that
$$
\| \mathcal{M}_1(t) \|_{L^2} \lesssim \| u(t) \|_{H^3}^3 \lesssim \e_1^3.
$$
Therefore, we can assume in the following that $t \geq 1$.

\medskip
\noindent
{\it Subcase 1.1: Low Frequencies $2^J \leq t^{-6/13}$.}
Due to the bound \eqref{muLcoeff+-}, for $K \leq J + 10$
on the support of $\mu_{K,J}$ we have
\begin{align*}
|\mu_{K,J}(k,a,b,q) | \lesssim 2^J.
\end{align*}
Using the support properties of $\mu_{K,J}$ we can then estimate
\begin{align*}
{\| \mathcal{M}_{K,J}(t) \|}_{L^2} & \lesssim t  \cdot 2^{K/2} \sup_k
  \iiint \left| \mu_{K,J}(k,a,b,q) \widetilde{f}(\gamma(-q-a)) \overline{\widetilde{f}(\delta(b-a-q)} \widetilde{f}(\eps(-q+ b)) \right| \,da \,db \, dq \\
& \lesssim t 2^{K/2} 2^{J} {\| \wt{f} \|}_{L^\infty}^3 \iiint \big| \underline{\varphi_{J}}\big(|(q+a,b-a-q,q-b)|\big) \big| \,da \, db \, dq
\\
& \lesssim t  2^{K/2} 2^J 2^{3J} \e_1^3.
\end{align*}
Summing over $K \leq J+10$ with $2^J \leq t^{-\frac{6}{13}}$ gives us
\begin{align}
\label{NL14}
\sum_{\substack{K-10 \leq J \\ 2^J \leq t^{-6/13}}} {\| \mathcal{M}_{K,J}(t) \|}_{L^2} \lesssim  t^{-1} \e_1^3.
\end{align}
From now on we may assume $2^J \geq t^{-6/13}$.
In the next step we compare the size of the integration variables $a$ and $b$ to $2^J$.
Without loss of generality we may assume that $\max\{|a|,|b|\} = |b|$, and consider terms of the form
\begin{align}
\label{NL15}
\begin{split}
\mathcal{M}_{K,J,B}(t,k) := t \iiint e^{it (-k^2 + q^2 + 2ab)} \widetilde{f}(\gamma(-q-a)) \overline{\widetilde{f}(\delta(b-a-q)} \widetilde{f}(\eps(-q+ b)) \\ \times \mu_{K,J}(k,a,b,q) \varphi_B(b)\, da \, db \, dq.
\end{split}
\end{align}

\medskip
\noindent
{\it Subcase 1.2: $B \geq J-20$ and $J \leq 0$}.
In this case we resort to the identity $(1/2itb) \partial_a e^{it (-k^2 + q^2 + 2ab)} = e^{it (-k^2 + q^2 + 2ab)}$ to integrate by parts in $a$, leading to
\begin{align*}
& \mathcal{M}_{K,J,B} = \mathcal{M}_{K,J,B}^1 + \mathcal{M}_{K,J,B}^2 + \{ \mbox{similar terms} \},
\end{align*}
with
\begin{align*}
& \mathcal{M}_{K,J,B}^1 := \iiint e^{it (-k^2 + q^2 + 2ab)}
\widetilde{f}(\gamma(-q-a)) \overline{\widetilde{f}(\delta(b-a-q))} \widetilde{f}(\eps(-q+ b)) \, m_1(k,a,b,q) \,da\, db\, dq,
\\
& m_1(k,a,b,q) := \partial_a \big[ \mu_{K,J}(k,a,b,q) \big]
  \frac{\varphi_B(b)}{2ib},
\end{align*}
and
\begin{align*}
& \mathcal{M}_{K,J,B}^2 := \iiint e^{it (-k^2 + q^2 + 2ab)}
  \partial_a  \widetilde{f}(\gamma(-q-a)) \overline{\widetilde{f}(\delta(b-a-q))} \widetilde{f}(\eps(-q+ b)) m_2(k,a,b,q) \, da\, db\, dq,
\\
& m_2(k,a,b,q) := \mu_{K,J}(k,a,b,q) \frac{\varphi_B(b)}{2ib},
\end{align*}
with similar terms arising when $\partial_a$ hits the second and third profiles $\widetilde{f}$.

We will now denote, for any symbol $m$,
\begin{align*}
m^{\sharp}(k,\ell,m,n) = m(k,a,b,q),
\end{align*}
where $a,b,q$ are given by the change of variables $(\ell,m,n)= (\gamma(-q-a),\delta(b-a-q),\e(-q+b))$ performed before.
Notice that, in view of the support restrictions (in particular $J \sim B$) and Proposition~\ref{muprop},
\begin{align*}
\| \widehat{\mathcal{F}}^{-1} m_1^\sharp \|_{L^2_w L^1_{x,y,z}} & \lesssim \| \widehat{\mathcal{F}}^{-1} m_1^\sharp \|_{L^2_{w,x,y,z}}^{1/4}
\| |(x,y,z)|^2 \widehat{\mathcal{F}}^{-1} m_1^\sharp \|_{L^2_{w,x,y,z}}^{3/4} \\
& = \| m_1^\sharp \|_{L^2_{k,\ell,m,n}}^{1/4} \| \nabla_{\ell,m,n}^2 m_1^\sharp \|_{L^2_{k,\ell,m,n}}^{3/4}
\lesssim (2^{(K+J)/2})^{1/4} (2^{(K-3J)/2})^{3/4}
\\ & \lesssim 2^{\frac{K}{2}-J}.
\end{align*}

Since
$$
\mathcal{M}^1_{K,J,B} = e^{-itk^2} \widehat{\mathcal{F}}
T_{\widehat{\mathcal{F}}^{-1} m_1^{\sharp}} (\widehat{\mathcal{F}}^{-1} \widetilde{u},
\widehat{\mathcal{F}}^{-1} \widetilde{u} ,\widehat{\mathcal{F}}^{-1} \widetilde{u}),
$$
see the notation in \ref{secLP},
we can bound, by Lemma \ref{lemmult}, the above estimate on $\widehat{\mathcal{F}}^{-1} m_1^\sharp$, and the linear estimate~\eqref{propdisp1},
\begin{align*}
\| \mathcal{M}^1_{K,J,B}\|_{L^2} & \lesssim \|  \widehat{\mathcal{F}}^{-1} m_1^\sharp \|_{L^2_w L^1_{x,y,z}} \|  \widehat{\mathcal{F}}^{-1}\widetilde{u} \|_{L^\infty}^3  \\
&  \lesssim 2^{\frac{K}{2}-J} \left[ \frac{\| \widetilde f\|_{L^\infty}}{\sqrt{t}} + \frac{\| \partial_k \widetilde{f} \|_{L^2}}{t^{3/4}}\right]^3 \lesssim  2^{\frac{K}{2}-J} t^{-3/2} \e_1^3.
\end{align*}
Therefore, summing over all indices in the current configuration, we obtain the bound
$$
\sum_{\substack{K\leq J+10\\1 > 2^J> t^{-6/13} \\ J \sim B}}  \| \mathcal{M}_{K,J,B}^2 \|_{L^2} \lesssim \sum_{\substack{K\leq J+10\\1 > 2^J> t^{-6/13} \\ J \sim B}} 2^{\frac{K}{2}-J} t^{-3/2} \e_1^3 \lesssim t^{-\frac{33}{26}} \e_1^3,
$$
which suffices!

Turning to $\mathcal{M}^2_{K,J,B}$, one proceeds similarly by observing first that
$$
\| \widehat{\mathcal{F}}^{-1} m_2^\sharp \|_{L^2_{w,x} L^1_{y,z}} \lesssim \| \widehat{\mathcal{F}}^{-1} m_2^\sharp \|_{L^2_{w,x,y,z}}^{1/2}
  \| |(y,z)|^2 \widehat{\mathcal{F}}^{-1} m_1^\sharp \|_{L^2_{w,x,y,z}}^{1/2} \lesssim 2^{(K+J)/2}.
$$
Therefore,
$$
\| \mathcal{M}^2_{K,J,B}\|_{L^2}  \lesssim \|  \widehat{\mathcal{F}}^{-1} m_2^\sharp \|_{L^2_{w,x} L^1_{y,z}} \|  \widehat{\mathcal{F}}^{-1}\widetilde{u} \|_{L^\infty}^2 \| \partial_k \widetilde{f} \|_{L^2} \lesssim 2^{(K+J)/2} t^{-\frac{3}{4} - \alpha} \e_1^3,
$$
which, after summing over all indices in the current configuration, leads to the acceptable bound
$$
\sum_{\substack{K\leq J+10\\1 > 2^J> t^{-6/13} \\ J \sim B}}  \| \mathcal{M}_{K,J,B}^2 \|_{L^2}
\lesssim \sum_{\substack{K\leq J+10\\1 > 2^J> t^{-6/13}}} 2^{(K+J)/2} t^{-\frac{3}{4} - \alpha}\e_1^3 \lesssim t^{-\frac{3}{4} - \alpha}\e_1^3.
$$

\medskip
\noindent
{\it Subcase 1.3: $B \geq J-20$ and $J \geq 0$}.
Integrating by parts in $b$ as in the case $J \leq 0$, matters reduce to estimating
$$\sum_{\substack{K \leq J + 10, \\ B \sim J, \, J \geq 0}} \mathcal{M}^1_{K,J,B}
+ \sum_{\substack{K \leq J + 10, \\ B \sim J, \, J \geq 0}} \mathcal{M}^2_{K,J,B}.$$
We will only discuss the latter sum, which is slightly more delicate.
Arguing as in Section \ref{nlzero} to replace $f$ by $g_{\sigma_i}$, observe that
\begin{align*}
\sum_{\substack{K \leq J + 10, \\ B \sim J, \, J \geq 0}} \mathcal{M}^2_{K,J,B}
= \iiint e^{it(-k^2 + \ell^2 - m^2 + n^2)} g_{\sigma_1}(\ell) g_{\sigma_2}(m) g_{\sigma_3}(n) \varphi_J(|(\ell,m,n)|)
\\ \times \varphi_{\leq J+10}(k) \frac{\underline{\varphi_J}(-\gamma \ell + \delta m)}{-\gamma \ell + \delta m} \,d\ell\,dm\,dn,
\end{align*}
which can also be written as
\begin{align*}
& \sum_{\substack{K \leq J + 10 \\ B \sim J, \, J \geq 0}} \mathcal{M}^2_{K,J,B}
= e^{-itk^2} \int \widehat{\phi}(p) \mathcal{T}_{J,p}(g_{\sigma_1}, g_{\sigma_2}, g_{\sigma_3})\,dp
\\
&  \mathcal{T}_{J,p}(g_{\sigma_1}, g_{\sigma_2}, g_{\sigma_3})
:= \iint e^{it(\ell^2 - m^2 + n^2)} g_{\sigma_1}(\ell) g_{\sigma_2}(m) g_{\sigma_3}(n) \nu_{J,p}(k,m,n) \,dm\,dn,
\\
& \nu_{J,p}(k,m,n) := \varphi_J(|(\ell,m,n)|) \varphi_{\leq J+10}(k)
  \frac{\underline{\varphi_J}(-\gamma \ell + \delta m)}{-\gamma \ell + \delta m}
\end{align*}
where in the last integral, $\ell$ always stands for $\ell = \gamma(\beta k + \delta m -\epsilon n - p)$. Observe that the Fourier transform of the kernel $\nu_{J,p}$ is easily bounded by
$$
\| \widehat{ \mathcal{F}}^{-1} \nu_{J,p} \|_{L^1} \lesssim 2^{-J}.
$$
It is then easy to conclude by Lemma~\ref{lemmult2} that
$$
\Big\| \sum_{\substack{K \leq J + 10 \\ B \sim J, \, J \geq 0}} \mathcal{M}^2_{K,J,B} \Big\|_{L^2}
\lesssim \sum_{J \geq 0} 2^{-J} \frac{\varepsilon_1^3}{t^{\frac{3}{4}+\alpha}} \lesssim \frac{\varepsilon_1^3}{t^{\frac{3}{4}+\alpha}},
$$
which leads to the desired estimate.

\medskip
\noindent
{\it Subcase 1.4: $B \leq J-20$.}
We now consider the term $\mathcal{M}_{K,J,B}$, when $B \leq J-20$. Here again, the difficulty lies in estimating the contribution of $J \leq0$; we will focus on it and omit the case $J \geq 0$. Observe that on the support of this oscillatory integral we must have $|a|+|b| \approx 2^B \ll \max\{|a+q|,|b-a-q|,|b-q|\} \approx 2^J$. It then follows that $|q| \approx 2^J$.
We can then integrate by parts in $q$. More precisely we can write
\begin{align}
\label{NL18}
\begin{split}
\sum_{B \leq J-20} \mathcal{M}_{K,J,B}(t,k) = t \iiint e^{it (-k^2 + q^2 + 2ab)} \,
 \widetilde{f}(\gamma(-q-a)) \overline{\widetilde{f}(\delta(b-a-q))} \widetilde{f}(\eps(-q+ b))\\
  \mu_{K,J}(k,a,b,q) \varphi_{\leq J-20}(b) \underline{\varphi}_J(q) \, da\, db\, dq,
\end{split}
\end{align}
and, similarly to what was done above, obtain
\begin{align*}
& \sum_{B \leq J-20} \mathcal{M}_{K,J,B}(t,k) = \mathcal{M}_{K,J}^3 + \mathcal{M}_{K,J}^4,
\end{align*}
with
\begin{align*}
& \mathcal{M}_{K,J}^3 := \iiint e^{it (-k^2 + q^2 + 2ab)}  \widetilde{f}(\gamma(-q-a)) \overline{\widetilde{f}(\delta(b-a-q))} \widetilde{f}(\eps(-q+ b))  m_3(k,a,b,q) \,da\, db\, dq,
\\
& m_3(k,a,b,q) = \partial_q \Big[  \mu_{K,J}(k,a,b,q) \frac{\underline{\varphi}_J(q)}{2iq} \Big] \varphi_{\leq J-20}(b),
\end{align*}
and
\begin{align*}
& \mathcal{M}_{K,J}^4 := \iiint e^{it (-k^2 + q^2 + 2ab)}
  \partial_q\big[ \wt{f}(\g(-q-a)) \overline{\wt{f}(\d(b-a-q))} \wt{f}(\eps(-q+b)) \big] m_4(k,a,b,q) \, da db dq,
\\
& m_4(k,a,b,q) := \mu_{K,J}(k,a,b,q)  \varphi_{\leq J-20}(b) \frac{\underline{\varphi}_J(q)}{2iq}.
\end{align*}

Direct computations show that the following bounds hold:
\begin{align*}
& {\big\| \what{\mathcal{F}}^{-1} m_3^{\sharp} \big\|}_{L^2_w L^1_{x,y,z}} \lesssim 2^{\frac{K}{2}-J}.
\\
& {\big\| \what{\mathcal{F}}^{-1} m_4^{\sharp} \big\|}_{L^2L^2L^1L^1} \lesssim 2^{(K+J)/2}.
\end{align*}
We can then proceed exactly as we did for the terms $\mathcal{M}_{K,J,B}^1$ and $\mathcal{M}_{K,J,B}^2$ above,
applying Lemma \ref{lemmult} and obtaining the desired bounds.
This shows that the term $\mathcal{M}_1$ in \eqref{NL1split} satisfies the estimate \eqref{NL10}.

\bigskip
\noindent
{\it Case 2: Estimate on $\mathcal{M}_2$}.
In this case the variable $k$ dominates all the others. Again we distinguish the case of small and high frequencies.

\medskip
\noindent
{\it Subcase 2.1: $t \leq 1$ or $2^K \leq t^{-6/13}$.} Here we can proceed exactly as in Subcase 1.1 above to deduce the desired estimate.

\medskip
\noindent
{\it Subcase 2.2: $2^K \geq t^{-6/13}$.}
In this case we integrate by parts in time. Let us denote the oscillating phase in \eqref{NL12} by
$$
\Phi = \Phi(k,a,b,q) = -k^2 + q^2 + 2ab,
$$
and observe that for $K \geq J+10$, on the support of the integral, we have $|k| \gg |a|, |b|, |q|$
and, in particular, $|\Phi| \gtrsim k^2$.
Integrating by parts in time via the identity $\partial_s e^{is\Phi} = (1/i\Phi)e^{is\Phi}$, we get
\begin{align}
\label{NL19}
\begin{split}
& \int_0^t \mathcal{M}_{K,J}(s,k)\,ds = t S^1(t,k) - S^1(1,k) + \int_1^t S^1(s,k) \, ds + \int_1^t S^2(s,k) \, ds + \{ \mbox{similar terms} \},
\\
& S^1(t,k) := \iiint e^{is\Phi} \wt{f}(\g(-q-a)) \overline{\wt{f}(\d(b-a-q))} \wt{f}(\eps(-q+b)) \sigma(k,a,b,q) \, da\, db\, dq,
\\
& S^2(t,k) := \iiint t \,e^{it\Phi}\partial_t \wt{f}(\g(-q-a)) \overline{\wt{f}(\d(b-a-q))} \wt{f}(\eps(-q+b)) \sigma(k,a,b,q) \,da\,db\,dq,
\end{split}
\end{align}
with similar terms arising when $\partial_t$ hits the second or the third profile $\widetilde{f}$, and
\begin{align*}
\sigma(k,a,b,q) := \frac{1}{i \Phi(k,a,b,q)} \mu_{K,J}(k,a,b,q).
\end{align*}
It is not hard to verify that
satisfies
\begin{align*}
\begin{split}
& {\big\| \what{\mathcal{F}}^{-1} \sigma^\sharp \big\|}_{L^2_{w,x} L^1_{y,z}} \lesssim 2^{\frac{3}{2}(J-K)},
\end{split}
\end{align*}
for $J\leq 0$. For $J>0$ the bound above would have an extra factor of $2^{J/2}$:
This loss can be tolerated for $2^J \leq t^{1/3}$ by proceeding as we do below,
while for $2^J \leq t^{1/3}$ one can rely on the a priori $H^3$ bound of Proposition \ref{propH3} to obtain the desired estimate.
We leave the details of this simpler case to the reader.

Using Lemma \ref{lemmult} we have
\begin{align*}
{\| S^1(t) \|}_{L^2} \lesssim  2^{\frac{3}{2}(J-K)} {\| \what{\mathcal{F}}^{-1}\wt{u}\|}_{L^2} {\| \what{\mathcal{F}}^{-1}\wt{u}\|}_{L^\infty}^2 \lesssim 2^{\frac{3}{2}(J-K)} t^{-1} \e_1^3,
\end{align*}
which after summation in $J,K$ over the current range of indices, leads to the acceptable contribution
$$
\sum_{\substack{K \geq J+10\\2^K\geq t^{-6/13}}} 2^{\frac{3}{2}(J-K)} t^{-1} \e_1^3 \lesssim \frac{\log t}{t} \e_1^3.
$$
Finally, recalling that $\partial_t f = e^{-it(-\partial_x^2 + V)} |u|^2 u$, we can estimate
\begin{align*}
{\| S^2(t) \|}_{L^2} \lesssim  t 2^{\frac{3}{2}(J-K)} \| \partial_t \widetilde f \|_{L^2} \| \what{\mathcal{F}}^{-1}\wt{u} \|_{L^\infty}^2
\lesssim 2^{\frac{3}{2}(J-K)} t^{-1} \e_1^3,
\end{align*}
which again largely suffices since
$$
\sum_{\substack{K \geq J+10\\1>2^K>t^{-6/13}}} 2^{\frac{3}{2}(J-K)} t^{-1} \e_1^3 \lesssim \frac{\log t}{t} \e_1^3.
$$
This concludes the proof of \eqref{NL10}, and of the weighted $L^2$-bound for $\mathcal{N}_{L,V}$.

To complete the estimate of $\mathcal{N}_L^-$, see \eqref{NLdec}, one needs to control the smoother remainder term
$\mathcal{N}_{L,V,r}$ in \eqref{NLvr}. This can be estimated exactly as in \ref{NVrsec} where we treated the similar
term $\mathcal{N}_{V,r}$, see the formula \eqref{NVr}. Therefore, we omit the details.

\medskip
\subsection{Estimates for $\mathcal{N}_R$}
We now look at the regular part
\begin{align}
\label{N_R0}
\begin{split}
4\pi^2\mathcal{N}_R(t,k) =
  \iiint e^{it \Phi(k,\ell,m,n)} \widetilde{f}(t,\ell) \overline{\widetilde{f}(t,m)} \widetilde{f}(t,n) \mu_R (k,\ell,m,n) \,d\ell\,dm\,dn,
\\ \Phi(k,\ell,m,n) = -k^2 + \ell^2 - m^2 + n^2,
\end{split}
\end{align}
where the measure $\mu_R$ is defined in Proposition \ref{muprop}, and want to show that this is a remainder term.
In particular we will establish the following Lemma which contains also an estimate for the $L^\infty_k$ norm of $\mathcal{N}_R(t,\cdot)$
to be used in the next section.

\begin{lem}\label{NR}
Under the a priori assumptions \eqref{bootstrap} we have
\begin{align}
\label{NRest}
& {\| \mathcal{N}_R(t,k) \|}_{L^\infty_k} \lesssim \e_1^3 \langle t \rangle^{-5/4},
\\
\label{NRest2}
& {\Big\| \int_0^t \partial_k \mathcal{N}_R(t,s) \, ds \Big\|}_{L^2} \lesssim \e_1^3 \langle t \rangle^{1/4 - \alpha}.
\end{align}
\end{lem}

\begin{proof}
We will use \eqref{muRprop} from Proposition \ref{muprop}:
\begin{align}
\label{NR3}
\begin{split}
& \big| \partial_k^{\theta_1} \partial_\ell^{\theta_2} \partial_m^{\theta_3} \partial_n^{\theta_4} \mu_R(k,\ell,m,n) \big|
\lesssim \min(|k|,1)^{1-\theta_1} \min(|\ell|,1)^{1-\theta_2} \min(|m|,1)^{1-\theta_3} \min(|n|,1)^{1-\theta_4}
\end{split}
\end{align}
for $\theta_1, \theta_2 , \theta_3, \theta_4 = 0$ or $1$, $\theta_1+\theta_2+\theta_3+\theta_4 \leq 3$.

We decompose
\begin{align}
\label{NR4}
\begin{split}
& \mathcal{N}_R(t,k) := \sum_{K,L,M,N \in \Z} \mathcal{N}_{KLMN}(t,k)
\\
& \mathcal{N}_{KLMN}(t,k) =
  \iiint e^{it \Phi(k,\ell,m,n)} \widetilde{f}(t,\ell) \overline{\widetilde{f}(t,m)} \widetilde{f}(t,n)
  \underline{\varphi}_{KLMN}\mu_R (k,\ell,m,n) \,d\ell\,dm\,dn,
\\
& \underline{\varphi}_{KLMN}(k,\ell,m,n) := \varphi_K(k) \varphi_L(\ell) \varphi_M(m) \varphi_N(n).
\end{split}
\end{align}
Without loss of generality, for the rest of this proof we will assume $$L\leq M \leq N.$$

Let us begin by recording some basic estimates: under our bootstrap assumptions, see \eqref{apriori0}-\eqref{bootstrap}, we have
\begin{equation}
\begin{split}
\label{bootbound}
& \int_\R |\varphi_K(k) \wt{f}(k) | \,dk \lesssim \min(2^K, 2^{-5K/2} t^{p_0}) \e_1,
\\
& \int_\R | \varphi_K(k) k^{-1} \wt{f}(k) | \, dk \lesssim \min\big( 2^{K/2}\langle t \rangle^{1/4-\alpha}, 2^{-K/2} \big) \e_1,
\\
& \int_\R | \partial_k [\varphi_K(k) k^{-1} \wt{f}(k)]| \, dk \lesssim 2^{-K/2}\langle t \rangle^{1/4-\alpha} \e_1,
\end{split}
\end{equation}
where we used Hardy's inequality in deriving the last two estimates.

\medskip \noindent
{\it Proof of \eqref{NRest}}. The case $|t|<1$ is immediate, so we will assume that $t\geq1$.
Integrating by parts and using the bounds~\eqref{bootbound} and~\eqref{NR3},

\begin{align*}
 |\mathcal{N}_{KLMN}(t,k) | & \lesssim \frac{1}{t^2} \iiint \Big| \widetilde{f}(t,\ell) \partial_m\partial_n \Big( \frac{1}{m} \overline{\widetilde{f}(t,m)} \frac{1}{n} \widetilde{f}(t,n) \underline{\varphi}_{KLMN}(k,\ell,m,n) \mu_R (k,\ell,m,n) \Big) \Big| \,d\ell\, dm\, dn \\
& \lesssim \frac{\e_1^3}{t^2} 2^{K_- + L_- + M_- + N_-} \min(2^L,2^{-5L/2} t^{p_0})
\cdot 2^{-M/2} t^{\frac{1}{4} - \alpha}  \cdot 2^{-N/2} t^{\frac{1}{4} - \alpha}.
\end{align*}
Summing over $L,M$ and $N$ gives the desired bound:
$$
 \sum_{L<M<N}|\mathcal{N}_{KLMN}(t,k) | \lesssim \frac{\e_1^3}{t^{5/4}}.
$$

\bigskip \noindent
{\it Proof of \eqref{NRest2}}.
We now prove the weighted $L^2$ bound. Adopting the notation \eqref{NR4} we calculate
\begin{align}
\begin{split}
& 4\pi^2 \partial_k \mathcal{N}(t,k) = I(t,k) + II(t,k),
\\
& I(t,k) := \iiint e^{it \Phi(k,\ell,m,n)} \widetilde{f}(t,\ell) \overline{\widetilde{f}(t,m)} \widetilde{f}(t,n)
  \partial_k\mu_R (k,\ell,m,n) \,d\ell \,dm\, dn,
\\
& II(t,k) := -2itk \iiint e^{it \Phi(k,\ell,m,n)} \widetilde{f}(t,\ell) \overline{\widetilde{f}(t,m)} \widetilde{f}(t,n)
  \mu_R (k,\ell,m,n) \,d\ell \, dm\, dn.
\end{split}
\end{align}
We will focus on the more complicated estimate of $II(t,k)$. Again we decompose according to \eqref{NR4}:
\begin{align}
\begin{split}
& II(t,k) := \sum_{K,L,M,N \in \Z} II_{KLMN}(t,k)
\\
&  II_{KLMN}(t,k) =
  -2itk \iiint e^{it \Phi(k,\ell,m,n)} \widetilde{f}(t,\ell) \overline{\widetilde{f}(t,m)} \widetilde{f}(t,n)
  \underline{\varphi}_{KLMN}\mu_R (k,\ell,m,n) \,d\ell \,dm\, dn.
\end{split}
\end{align}
We now distinguish between the cases $K \geq N + 10$ and $K < N + 10$.

\medskip \noindent
{\it Case 1: $K \geq N+10$}.
In this case we have $|\Phi| \gtrsim k^2 \approx 2^{2K}$, and we can resort to integration by parts in $s$:
\begin{align}
\label{NR10}
& \int_0^t II_{KLMN}(s,k) \,ds = -2t A(t,k) + 2\int_0^t A(s,k) \, ds + \int_0^t 2s B(s,k) \, ds,
\\
\nonumber
& A_{KLMN}(t,k) = \iiint e^{it \Phi(k,\ell,m,n)} \frac{k}{\Phi(k,\ell,m,n)}
  \widetilde{f}(t,\ell) \overline{\widetilde{f}(t,m)} \widetilde{f}(t,n)
  \underline{\varphi}_{KLMN}\mu_R (k,\ell,m,n) \,d\ell \, dm \, dn,
\\
\nonumber
& B_{KLMN}(t,k) = \iiint e^{it \Phi(k,\ell,m,n)} \frac{k}{\Phi(k,\ell,m,n)}
   \partial_t \big[ \widetilde{f}(t,\ell) \overline{\widetilde{f}(t,m)} \widetilde{f}(t,n) \big]
  \underline{\varphi}_{KLMN}\mu_R (k,\ell,m,n) \,d\ell \, dm\, dn.
\end{align}

To estimate $A$ we integrate by parts in the frequencies $m$ and $n$ similarly to what was done above in the proof of  \eqref{NRest2}.
Using the bootstrap bounds~\eqref{bootbound} and the bounds on $\mu_R$~\eqref{NR3}, we get
\begin{align*}
\begin{split}
&| A_{KLMN}(t,k)| \\
&\quad  \lesssim \frac{1}{|t|^2} \iiint \Big| \widetilde{f}(t,\ell)
  \partial_m\partial_n \Big( \frac{k\underline{\varphi}_{KLMN}(k,\ell,m,n) }{\Phi(k,\ell,m,n)}
  \, \frac{1}{m} \overline{\widetilde{f}(t,m)} \frac{1}{n} \widetilde{f}(t,n) \mu_R (k,\ell,m,n) \Big) \Big| \,d\ell\, dm\, dn \\
& \quad \lesssim \frac{ \e_1^3}{|t|^2} 2^{K_- + L_- + M_- + N_-} 2^{-K} \min(2^L,2^{-5L/2} t^{p_0})
  \cdot 2^{-M/2} t^{\frac{1}{4} - \alpha} \cdot 2^{-N/2} t^{\frac{1}{4} - \alpha}.
\end{split}
\end{align*}
Using the above bound and summing over the current configuration,
$$
\sum_{\substack{L\leq M\leq N\\K\geq N+10}}\| A_{KLMN}(t,k) \|_{L^2}
  \lesssim \sum_{\substack{L\leq M\leq N \\ K \geq N+10}} 2^{K/2} \| A_{KLMN}(t,k) \|_{L^\infty} \lesssim \frac{\e_1^3}{|t|^{5/4}}.
$$

Turning to $B_{KLMN}$, split it first into
\begin{align*}
 B_{KLMN}(t,k) & = \iiint e^{it \Phi(k,\ell,m,n)} \frac{k}{\Phi(k,\ell,m,n)}
   \partial_t \widetilde{f}(t,\ell) \overline{\widetilde{f}(t,m)} \widetilde{f}(t,n)  \underline{\varphi}_{KLMN}\mu_R (k,\ell,m,n) \,d\ell \, dm\, dn \\
& \qquad +  \iiint e^{it \Phi(k,\ell,m,n)} \frac{k}{\Phi(k,\ell,m,n)}
    \widetilde{f}(t,\ell) \overline{\widetilde{f}(t,m)} \partial_t\widetilde{f}(t,n)  \underline{\varphi}_{KLMN}\mu_R (k,\ell,m,n) \,d\ell \, dm\, dn \\
& \qquad + \{ \mbox{similar term} \} \\
& = B^1_{KLMN}(t,k) + B^2_{KLMN}(t,k) + \{ \mbox{similar term} \}.
\end{align*}

Let us consider first $B^1_{KLMN}$. Integrating by parts in $m$ and $n$,
and using that $\| \partial_t \wt{f}\|_{L^2} \lesssim \e_1^3t^{-1}$, we see that
\begin{align*}
\begin{split}
| B^1_{KLMN}(t,k) |  \lesssim \frac{1}{|t|^2} \iiint \Big| \partial_t \widetilde{f}(t,\ell)
  \partial_m\partial_n \Big( \frac{k\underline{\varphi}_{KLMN}(k,\ell,m,n) }{\Phi(k,\ell,m,n)}
  \, \frac{1}{m} \overline{\widetilde{f}(t,m)} \frac{1}{n} \widetilde{f}(t,n) \mu_R (k,\ell,m,n) \Big) \Big| \\\,d\ell\, dm \,dn
\\
\lesssim \frac{\e_1^3}{|t|^2} 2^{K_-+ L_- + M_- + N_-} \cdot 2^{-K} \cdot \frac{2^{L/2}}{t} \cdot 2^{-M/2} t^{\frac{1}{4} - \alpha}
\cdot 2^{-N/2} t^{\frac{1}{4} - \alpha}.
\end{split}
\end{align*}
This $L^\infty$ bound leads to an acceptable contribution:
\begin{align*}
\sum_{\substack{L<M<N \\ K \geq N+10}} \| B^1_{KLMN}(t,k) \|_{L^2}
  \lesssim \sum_{\substack{L<M<N \\ K \geq N+10}} 2^{K/2} \| B^1_{KLMN}(t,k) \|_{L^\infty} \lesssim \frac{\e_1^3}{t^{5/2}}.
\end{align*}

To estimate $B^2_{KLMN}$ first notice that
\begin{align*}
{\| \langle k \rangle\partial_t \wt{f}\|}_{L^2} \lesssim {\big\|\langle k \rangle \wt{u^3} \big\|}_{L^2}
  \lesssim {\big\| u^3 \big\|}_{H^1} \lesssim \e_1^3 t^{p_0-1},
\end{align*}
having used \eqref{FH1}.
Then we can integrate by parts in $\ell$ and $m$ to obtain
\begin{align*}
\begin{split}
| B^2_{KLMN}(t,k) |  \lesssim \frac{1}{|t|^2} \iiint \Big|
  \partial_\ell\partial_m \Big( \frac{k\underline{\varphi}_{KLMN}(k,\ell,m,n) }{\Phi(k,\ell,m,n)} \, \frac{1}{\ell} \overline{\widetilde{f}(t,\ell)} \frac{1}{m} \widetilde{f}(t,m) \mu_R (k,\ell,m,n) \Big)\partial_t  \widetilde{f}(t,n)  \Big|\\
\,d\ell\, dm \,dn
\\
\lesssim \frac{\e_1^3}{|t|^2} 2^{K_-+ L_- + M_- + N_-} \cdot 2^{-K} \cdot 2^{-L/2} t^{\frac{1}{4} - \alpha}
  \cdot 2^{-M/2} t^{\frac{1}{4} - \alpha}
  \cdot 2^{-N/2} t^{p_0-1},
\end{split}
\end{align*}
which, after summing over all current indices,  leads to an acceptable contribution:
\begin{align*}
\sum_{\substack{L<M<N \\ K \geq N+10}} \| B^2_{KLMN}(t,k) \|_{L^2}
  \lesssim \sum_{\substack{L\leq M\leq N \\ K \geq N+10}} 2^{K/2} \| B^2_{KLMN}(t,k) \|_{L^\infty} \lesssim \frac{\e_1^3}{t^{5/2}}.
\end{align*}

\medskip
\noindent
{\it Case 2: $K < N+10$}.
We distinguish two subcases depending on the size of $N$.

\smallskip
\noindent
{\it Subcase 2.1: $2^N \geq  t^{1/4}$}.
We integrate by parts in $\ell$ and $m$, and use again \eqref{bootbound}, to obtain
\begin{equation*}
\begin{split}
&|II_{KLMN}(t,k)| \\
&\lesssim  \frac{1}{|t|^2} \iiint \Big|
\partial_\ell\partial_m \Big(k\underline{\varphi}_{KLMN}(k,\ell,m,n) \, \frac{1}{\ell} \overline{\widetilde{f}(t,\ell)} \frac{1}{m} \widetilde{f}(t,m) \mu_R (k,\ell,m,n) \Big)  \widetilde{f}(t,n)  \Big| \,d\ell\, dm \,dn \\
& \lesssim  \frac{\e_1^3}{|t|}\cdot 2^{K} \cdot 2^{K_-+ L_- + M_- + N_-}
  \cdot 2^{-L/2} t^{\frac{1}{4} - \alpha}  \cdot 2^{-M/2} t^{\frac{1}{4} - \alpha} \cdot 2^{-5N/2} t^{p_0},
\end{split}
\end{equation*}
which, after using this to estimate the $L^2$ norm and summing over all current indices, gives an acceptable contribution
\begin{align*}
\sum_{\substack{L\leq M<N\leq  \\ 2^N \geq t^{1/4}}} \| II_{KLMN}(t,k) \|_{L^2} \lesssim
  \sum_{\substack{L\leq M\leq N \\ 2^N \geq t^{1/4}}} 2^{K/2} \| II_{KLMN}(t,k) \|_{L^\infty} \lesssim \frac{\e_1^3 }{t^{3/4+2\alpha-p_0}}.
\end{align*}

\smallskip
\noindent
{\it Subcase 2.2: $2^N \leq t^{1/4}$}. Integrating by parts in $\ell$, $m$, and $n$ leads to the bound
\begin{align*}
\begin{split}
& |II_{KLMN}(t,k) |
\\
& \lesssim \frac{1}{|t|^2} \iiint \Big| \partial_m\partial_n\partial_\ell
  \Big( \frac{1}{\ell}  \widetilde{f}(t,\ell) \frac{1}{m} \overline{\widetilde{f}(t,m)} \frac{1}{n} \widetilde{f}(t,n)
  \, k \underline{\varphi}_{KLMN}(k,\ell,m,n) \mu_R (k,\ell,m,n) \Big) \Big| \,d\ell\, dm\, dn
\\
& \lesssim \frac{\e_1^3}{t^2} 2^{K}2^{K_-+ L_- + M_- + N_-} \cdot 2^{-L/2} t^{\frac{1}{4} - \alpha}
  \cdot 2^{-M/2} t^{\frac{1}{4} - \alpha} \cdot 2^{-N/2} t^{\frac{1}{4} - \alpha},
\end{split}
\end{align*}
which gives
\begin{align*}
\sum_{\substack{L\leq M\leq N \\ 2^N \leq t^{1/4}}} \| II_{KLMN}(t,k) \|_{L^2} \lesssim
  \sum_{\substack{L\leq M\leq N \\ 2^N \leq t^{1/4}}} 2^{K/2} \| II_{KLMN}(t,k) \|_{L^\infty} \lesssim \frac{\e_1^3 }{t^{1+3\alpha}}.
\end{align*}
This concludes the proof of \eqref{NRest2}.
\end{proof}

\bigskip
\section{Pointwise estimate}\label{secLinfty}

In this section we prove the key $L^\infty$ bound. Recall Duhamel's formula
\begin{align}
\begin{split}
i \partial_t \widetilde{f}(t,k) &= \frac{1}{4\pi^2} \big[ \mathcal{N}_+ + \mathcal{N}_- + \mathcal{N}_L + \mathcal{N}_R \big],
\\
\mathcal{N}_\ast(t,k) &=
  \iiint e^{it (-k^2 + \ell^2 - m^2 + n^2)} \widetilde{f}(t,\ell) \overline{\widetilde{f}(t,m)} \widetilde{f}(t,n)
  \mu_\ast (k,\ell,m,n) \,d\ell dm dn.
\end{split}
\end{align}
together with Proposition \ref{muprop}.
Our aim is to find asymptotics for such expressions, and show that $\widetilde{f}(t,k)$ satisfies an ODE whose solutions
are bounded in $L^\infty_k$, uniformly in time.


\medskip
\subsection{Three stationary phase lemmas}

\begin{lem}\label{AsLem1}
For $k,t \in \R$, consider the integral expression
\begin{align}
\label{I1}
\begin{split}
I[g_1,g_2,g_3](t,k) = \iiint e^{it \Phi(k,p,m,n)}
  g_1(\gamma(\beta k - p + \delta m - \epsilon n)) \bar{g_2(m)} g_3(n) \frac{\widehat{\phi}(p)}{p} \, dm\, dn\, dp
\\
\Phi(k,p,m,n) = -k^2 + (\beta k - p + \delta m - \epsilon n)^2 - m^2 + n^2.
\end{split}
\end{align}
for an even bump function $\phi \in C_0^\infty$, and with $g:=(g_1,g_2,g_3)$ satisfying
\begin{align}
\label{AsLem1as}
{\| g(t) \|}_{L^\infty} + {\| \langle k \rangle g(t) \|}_{L^2} + \langle t \rangle^{-1/4+\alpha} {\| g'(t) \|}_{L^2}
\leq 1.
\end{align}
for some $\alpha > 0 $. Then, for any $t \in \mathbb{R}$,
\begin{align}
\label{I2}
\begin{split}
I[g_1,g_2,g_3](t,k) = \frac{\pi}{|t|} e^{-itk^2} \int e^{it(-p+\beta k)^2} g_1(\gamma(-p+\beta k))
\overline{g_2(\delta(-p+\beta k))} g_3(\epsilon (-p+\beta k)) \frac{\widehat{\phi}(p)}{p} \,dp
\\ + O(|t|^{-1-\alpha/3}).
\end{split}
\end{align}
\end{lem}

Note that the assumptions \eqref{AsLem1as} above are consistent with taking $g(k) = a_i(k)\wt{f}(k)$, $-2\leq i\leq 2$,
where the coefficients $a_{i}(k)$ are as in Remark \ref{RemmuL+-},
in view of our a priori assumptions \eqref{apriori0}, \eqref{TRk}, and Lemma \ref{Lem0}.

\begin{proof}[Proof of Lemma \ref{AsLem1}]
This is a nonlinear stationary phase argument with amplitudes of limited smoothness, and singularities in the integrand. We assume from now on that $t>0$; the case $t<0$ can be easily deduced by taking the complex conjugate of $I$.

\medskip
\noindent
{\it Step 1: The case $|p| \lesssim t^{-3}$}
Let us define
\begin{align*}
{\Psi_-}(p) = \widehat{\phi}(p) \varphi(pt^3), \qquad  {\Psi_+}(p) = \widehat{\phi}(p) - {\Psi_-}(p),
\end{align*}
and correspondingly let
\begin{align}
\label{AsLem10}
I_\pm[g_1,g_2,g_3](t,k) := \iiint e^{it \Phi(k,p,m,n)}
  g_1(\gamma(\beta k - p + \delta m - \epsilon n)) \bar{g_2(m)} g_3(n)  \frac{{\Psi_\pm}(p)}{p} \, dm\, dn\, dp.
\end{align}
Let us look at $I_-$ and observe that, since the $dp$ integral is understood in the p.v. sense and $\phi_-$ is even, we have
\begin{align}
\begin{split}
I_-[g_1,g_2,g_3](t,k) = \iiint \Big[ e^{it \Phi(k,p,m,n)} g_1(\gamma(\beta k - p + \delta m - \epsilon n)) \bar{g_2(m)} g_3(n)
  \\ - e^{it \Phi(k,0,m,n)} g_1(\gamma(\beta k + \delta m - \epsilon n)) \bar{g_2(m)} g_3(n)  \Big]
  \frac{{\Psi_-}(p)}{p} \, dm \,  dn \, dp.
\end{split}
\end{align}
It follows that we can estimate $|I_-(t,k)| \lesssim A + B$, with
\begin{align}
 \label{AsLem11}
\begin{split}
A & = \iiint \Big| g_1(\gamma(\beta k - p + \delta m - \epsilon n)) -
  g_1(\gamma(\beta k + \delta m - \epsilon n)) \Big| \big|\bar{g_2(m)} g_3(n)\big| \frac{|{\Psi_-}(p)|}{|p|} \, dm\, dn \,dp
\\
B & = \iiint \Big| e^{it\Phi(k,p,m,n)} - e^{it\Phi(k,0,m,n)} \Big| \big|g_1(\gamma(\beta k + \delta m - \epsilon n)) \bar{g_2(m)} g_3(n)\big|
  \frac{|{\Psi_-}(p)|}{|p|} \, dm\, dn\, dp.
\end{split}
\end{align}

Using the assumption on the derivative of $g_1$ in \eqref{AsLem1as} we can estimate
\begin{align*}
\begin{split}
\Big| g_1(\gamma(\beta k - p + \delta m - \epsilon n)) - g_1(\gamma(\beta k + \delta m - \epsilon n)) \Big|
& \lesssim \int_0^p | g_1^\prime(\gamma(\beta k + z + \delta m - \epsilon n)) | dz
\\
& \lesssim |p|^{1/2} {\| g_1' \|}_{L^2}
\end{split}
\end{align*}
and therefore obtain
\begin{align*}
\begin{split}
A & \lesssim \| g_1^\prime \|_{L^2} \| g_2 \|_{L^1} \| g_3 \|_{L^1}   \, \int \frac{|{\Psi_-}(p)|}{|p|^{1/2}} \, dp \lesssim \langle t \rangle^{-5/4}.
\end{split}
\end{align*}

For the second term in \eqref{AsLem11} we have
\begin{align*}
B & \lesssim \iiint t \big| p^2-2p(\beta k + \delta m - \epsilon n) \big| \,
  \big |g_1(\gamma(\beta k + \delta m - \epsilon n)) \bar{g_2(m)} g_3(n) \big| \frac{|{\Psi_-}(p)|}{|p|} \, dm\, dn\, dp
\\
& \lesssim \langle t \rangle {\| \langle k \rangle g \|}_{L^2}^2 {\| g \|}_{L^1 } \int |{\Psi_-}(p)| \, dp
  \lesssim \langle t \rangle^{-2}.
\end{align*}
This shows that $I_-$ is a remainder term, and from now on we concentrate on $I_+[g_1,g_2,g_3](t,k)$, often simply denoting it $I_+$.

In a similar way, one can show that
\begin{align*}
&\frac{\pi}{t} e^{-itk^2} \int e^{it(-p+\beta k)^2} g_1(\gamma(-p+\beta k))
\overline{g_2(\delta(-p+\beta k))} g_3(\epsilon (-p+\beta k)) \frac{\widehat{\phi}(p)}{p} \,dp \\
& \quad = \frac{\pi}{t} e^{-itk^2} \int e^{it(-p+\beta k)^2} g_1(\gamma(-p+\beta k))
\overline{g_2(\delta(-p+\beta k))} g_3(\epsilon (-p+\beta k)) \frac{\Psi_+(p)}{p} \,dp + O(t^{-1-\alpha/3}).
\end{align*}

\medskip
\noindent
{\it Step 2}.
We change variables from $(m,n)$ to $(a,b)$ by letting $m = \delta(a-b-p+\beta k)$ and $n = \epsilon(-p+\beta k -b)$.
This gives
\begin{align}
\label{I_+}
\begin{split}
& I_+(t,k) = e^{-itk^2} \int e^{it(-p+\beta k)^2} J(k,p) \frac{{\Psi_+}(p)}{p} \,dp,
\qquad J(k,p) : = \iint e^{2itab} G(a,b; p,k)\, da\, db,
\\
& G(a,b; p,k) := g_1(\gamma (a - p+\beta k)) \bar{g_2(\delta(a-b-p+\beta k))} g_3(\epsilon(-p+\beta k -b)).
\end{split}
\end{align}
We then decompose, for a parameter $\rho>0$ to be determined,
\begin{align}
\label{I_+2}
\begin{split}
& J = \frac{\pi}{t} G(0,0;p,k) + J_1 + J_2 + J_3,
\\
& J_1 = \iint e^{2itab} G(a,b;p,k) \, \varphi(|a|t^{1/2-\rho}) \varphi(|b|t^{1/2-\rho}) \, da\,db - \frac{\pi}{t} G(0,0;p,k),
\\
& J_2 = \iint e^{2itab} G(a,b;p,k) \, \big[ 1 - \varphi(|a|t^{1/2-\rho}) \big] \, da\,db,
\\
& J_3 = \iint e^{2itab} G(a,b;p,k) \, \varphi(|a|t^{1/2-\rho}) \big[ 1 - \varphi(|b|t^{1/2-\rho}) \big] \, da\,db.
\end{split}
\end{align}
Notice that since the integral in $dp$ is supported on $|p| \gtrsim t^{-3}$
it will suffice to show that $J_i$, $i=1,2,3$, are $O(t^{-1-\alpha/3})$ to obtain that their contributions
to $I_+$, through \eqref{I_+}, are acceptable remainder terms.

Integrating successively in $a$ and $b$, one obtains that
\begin{align}
\label{I_+3}
\begin{split}
\iint e^{2itab} \, \varphi(t^{1/2-\rho}|a|) \varphi(t^{1/2-\rho}|b|) \, da\, db
& = \sqrt{2\pi} \int t^{\rho-\frac{1}{2}} \widehat{\varphi}(2t^{\frac{1}{2}+\rho} b) \varphi (t^{\frac{1}{2}-\rho} b)\,db
\\
& = \frac{\pi}{t} \varphi(0)^2 + O (t^{-2}) = \frac{\pi}{t} + O (t^{-2}).
\end{split}
\end{align}
Therefore, we can write
\begin{align*}
J_1 = \iint e^{2itab} [G(a,b;p,k) -  G(0,0;p,k)] \, \varphi(|a|t^{1/2-\rho}) \varphi(|b|t^{1/2-\rho}) \, da \, db+ O (t^{-2}).
\end{align*}
Arguing as above, using the a priori bounds on the derivative of $g$, we see that
\begin{align*}
| G(a,b;p,k) -  G(0,0;p,k) | \lesssim (|a|+|b|)^{1/2} {\| g^\prime \|}_{L^2} \| g \|_{L^\infty}^2 \lesssim (|a|+|b|)^{1/2} t^{1/4-\alpha}.
\end{align*}
which gives us
\begin{align*}
| J_1 | & \lesssim \iint |G(a,b;p,k) - G(0,0;p,k) |\, \varphi(|a|t^{1/2-\rho}) \varphi(|b|t^{1/2-\rho}) \, da \, db+ t^{-2}
  \\
& \lesssim t^{(-1/2+\rho)(5/2)} t^{1/4-\alpha} + t^{-2} \lesssim t^{-1-\alpha+(5/2)\rho}+ t^{-2}.
\end{align*}

To treat $J_2$ we integrate by parts in $b$ and estimate
\begin{align*}
| J_2 | & \lesssim \frac{1}{t} \Big| \iint \frac{1}{a} e^{2itab} \partial_b G(a,b;p,k) \,\big[ 1 - \varphi(|a|t^{1/2-\rho}) \big] \,da \,db \Big|
\\
& \lesssim \frac{1}{t} {\| g_1 \|}_{L^\infty} {\| a^{-1} (1 - \varphi(|a|t^{1/2-\rho}) \|}_{L^2_a}
  \Big\| \int e^{it[ - (a-b)^2 + b^2]} \partial_b \big[ g_2(a-b) g_3(-b) \big] db \Big\|_{L^2_a}
\\
& \lesssim \frac{1}{t} \cdot t^{1/4-\rho/2} \cdot {\| g^\prime \|}_{L^2} {\|e^{it\partial^2_{x}} \what{g} \|}_{L^\infty}
  \lesssim t^{-1-\rho/2-\alpha}.
\end{align*}
Notice that we have used the linear estimate \eqref{propdisp} and the apriori assumptions \eqref{AsLem1} to deduce
${\|e^{it\partial^2_{x}} \what{g} \|}_{L^\infty} \lesssim t^{-1/2}$.
A similar estimate can be obtained for $J_3$ by integrating by parts in $a$:
\begin{align*}
| J_3 | & \lesssim K_1 + K_2
\\
K_1 & = \frac{1}{t} \Big| \iint \frac{1}{b} e^{2itab} \partial_a G(a,b;p,k)
  \,\varphi(|a|t^{1/2-\rho}) \big[ 1 - \varphi(|b|t^{1/2-\rho}) \big] \,da\, db \Big|
\\
K_2 & = \frac{1}{t} \Big| \iint \frac{1}{b} e^{2itab} G(a,b;p,k)
  \,\varphi'(|a|t^{1/2-\rho})  t^{1/2-\rho} \big[ 1 - \varphi(|b|t^{1/2-\rho}) \big] \,da\, db \Big|
\end{align*}
The term $K_1$ can be estimated analogously to $J_2$ above so we can skip it.
For $K_2$ we have
\begin{align*}
K_2 & \lesssim \frac{1}{t} \cdot t^{1/2-\rho}  {\big\| g_1(a-p+\beta k)\varphi'(|a|t^{1/2-\rho}) \big\|}_{L^2_a}
  \\ & \qquad \times {\Big\| \int \frac{1 - \varphi(|b-p+\beta k|t^{1/2-\rho})}{b-p+\beta k} e^{it[ -(a-b)^2 + (-b)^2]}
  g_2(a-b) g_3(-b) \, db \Big\|}_{L^2_a}
\\
& \lesssim \frac{1}{t} \cdot t^{1/2-\rho} \cdot t^{-1/4+\rho/2}{\| g\|}_{L^\infty}
  \cdot {\|e^{it\partial^2_{x}} \what{g} \|}_{L^\infty} \cdot  {\| b^{-1} (1 - \varphi(|b|t^{1/2-\rho}) \|}_{L^2_b}{\| g \|}_{L^\infty}
  \\ & \lesssim t^{-1-\rho}.
\end{align*}
Choosing $\rho = \alpha/3$ concludes the proof.
\end{proof}

From the proof of the above lemma, we record the following corollary.

\begin{lem}\label{AsLem1'}
For $k,t \in \R$, consider the integral expression
\begin{align}
\label{I1'}
\begin{split}
L[g_1,g_2,g_3](t,k) = \iint e^{it \Phi(k,p,m,n)}
  g_1(\gamma(\beta k + \delta m - \epsilon n)) \bar{g_2(m)} g_3(n) \, dm\, dn,
\end{split}
\end{align}
with $\Phi(k,p,m,n) = -k^2 + (\beta k + \delta m - \epsilon n)^2 - m^2 + n^2$ and $g:=(g_1,g_2,g_3)$ satisfying
\begin{align}
\label{AsLem1as'}
{\| g(t) \|}_{L^\infty} + {\| \langle k \rangle g(t) \|}_{L^2} + \langle t \rangle^{-1/4+\alpha} {\| g'(t) \|}_{L^2}
\leq 1.
\end{align}
for some $\alpha > 0 $. Then, for any $t \in \mathbb{R}$,
\begin{align}
\label{I2'}
\begin{split}
L[g_1,g_2,g_3](t,k) = \frac{\pi}{|t|}g_1(\gamma\beta k) \overline{g_2(\delta\beta k)} g_3(\epsilon \beta k)+ O(|t|^{-1-\alpha/3}).
\end{split}
\end{align}
\end{lem}

\begin{proof}
Simply notice that the trilinear operator $L$ in \eqref{I1'} coincides with $J(k,p=0)$ in \eqref{I_+}.
\end{proof}
\medskip
To deal with expressions such as those in \eqref{I2}, we will use the following:

\begin{lem}\label{AsLem2}
For $K \in \R, t>0$, consider the integral expression
\begin{align}
\label{AsLem2I}
I(t,K) = \mathrm{p.v.}  \int e^{it x^2} g(x) \frac{\psi(x-K)}{x-K} \,dx
\end{align}
for $\psi \in \mathcal{S}$, and $g$ satisfying
\begin{align}
\label{AsLem2as}
{\| g \|}_{L^\infty} + \langle t \rangle^{-1/4+\alpha} {\| g^\prime \|}_{L^2} \leq 1,
\end{align}
for some $\alpha \in (0,\frac{1}{4})$.
Then, for large $t>0$ we have
\begin{align}
\label{AsLem2conc}
I(t,K) = h(t,\sqrt{|t|} K) g(K) + O(|t|^{-\alpha/3})
\end{align}
where
\begin{align}
\label{AsLem2h}
h(t,y) =
\left\{ \begin{array}{ll}
\displaystyle \psi(0) e^{iy^2} \mathrm{p.v.} \int e^{i2xy + ix^2} \varphi(|x||t|^{-2\alpha+2\rho})
  \frac{dx}{x} \quad & \mathrm{for} \quad t>0,
\\
\\
\overline{h(-t,y)} & \mathrm{for} \quad t<0.
\end{array}\right.
\end{align}
\end{lem}

\begin{proof} We only deal with the case $t>0$; the case $t<0$ can be deduced by taking the complex conjugate.
Introduce a parameter $0<\rho<\alpha/2$, which we will optimize at the end of the proof.
In what follows we will often omit the $\mathrm{p.v.}$ notation where it is understood.
A change of variables gives then
\begin{align*}
h(\sqrt t K) & = \psi(0) \int e^{itx^2} \frac{1}{x-K} \varphi(t^{\frac{1}{2}-2\alpha + 2 \rho} |x-K|)\,dx 
\\
& = \int e^{itx^2} \frac{\psi(x-K)}{x-K} \varphi(t^{\frac{1}{2}-2\alpha+2\rho} |x-K|)\,dx + O(t^{-\frac{1}{2}+2\alpha - 2\rho}
).
\end{align*}
Next, we decompose
\begin{align}
\begin{split}
I & = A + B,
\\
A & = \int e^{it x^2} g(x) \frac{\psi(x-K)}{x-K} \varphi(|x-K| t^{1/2-2\alpha + 2\rho}) \,dx,
\\
B & = \int e^{it x^2} g(x) \frac{\psi(x-K)}{x-K} \big[ 1 -\varphi(|x-K| t^{1/2-2\alpha + 2\rho}) \big] \,dx.
\end{split}
\end{align}
For the first term we have
\begin{align*}
& \left| A - g(K) \int e^{it x^2} \frac{ \psi(x-K)}{x-K} \varphi(|x-K| t^{1/2-2\alpha + 2\rho}) \,dx \right| \\
& \qquad\qquad \lesssim \int |g(x)-g(K)| \frac{|\psi(x-K)|}{|x-K|} \varphi(|x-K| t^{1/2-2\alpha + 2\rho}) \,dx \\
& \qquad\qquad \lesssim {\| g^\prime \|}_{L^2} \int \frac{\varphi(|x-K| t^{1/2-2\alpha + 2\rho}) }{\sqrt{|x-K|}}\,dx \\
& \qquad\qquad \lesssim t^{1/4-\alpha} \big( t^{-1/2+2\alpha - 2\rho} \big)^{1/2} \lesssim t^{- \rho}.
\end{align*}

For the second terms we write
\begin{align}
\begin{split}
B & = B_1 + B_2,
\\
B_1 & = \int e^{it x^2} g(x) \frac{\psi(x-K)}{x-K} \big[1 - \varphi(|x-K| t^{1/2-2\alpha + 2\rho}) \big] \varphi(|x| t^{1/2-\alpha}) \,dx,
\\
B_2 & = \int e^{it x^2} g(x) \frac{\psi(x-K)}{x-K} \big[ 1 -\varphi(|x-K| t^{1/2-2\alpha + 2\rho}) \big]
  \big[ 1 -\varphi(|x| t^{1/2-\alpha}) \big] \,dx.
\end{split}
\end{align}
We can see directly that $B_1$ is an acceptable remainder:
\begin{align*}
| B_1 | & \lesssim {\| g \|}_{L^\infty} \int \frac{1}{|x-K|} [1-\varphi(|x-K| t^{1/2-2\alpha + 2\rho})\big]
  \varphi(|x| t^{1/2-\alpha}) \,dx \\
& \lesssim t^{1/2 -2\alpha + 2\rho} t^{-1/2+\alpha} \lesssim t^{-\alpha + 2\rho}.
\end{align*}

For $B_2$ notice that we are away from the singularity of the integrand as well as from the stationary point $x=0$.
We can then integrate by parts in $x$ to show this is also a remainder. In particular we can estimate
\begin{align*}
|B_2| & = \left| \int \frac{1}{t} e^{itx^2} \partial_x \Big( \frac{1}{2x} g(x) \frac{\psi(x-K)}{x-K}
  \big[ 1 -\varphi(|x-K| t^{1/2-2\alpha + 2\rho}) \big]
  \big[ 1 -\varphi(|x| t^{1/2-\alpha}) \big] \Big) \,dx \right| \\
& \lesssim \frac{1}{t} \big( C_1 + C_2 +C_3 \big),
\end{align*}
where
\begin{align}
\begin{split}
|C_1| & \lesssim \int \frac{1}{|x|}|g^\prime(x)| \frac{1}{|x-K|} \big[1 - \varphi(|x-K| t^{1/2-2\alpha + 2\rho}) \big]
  \big[ 1 -\varphi(|x| t^{1/2-\alpha}) \big] \,dx,
\\
|C_2| & \lesssim \int \frac{g(x)}{|x|} \Big| \partial_x \Big[ \frac{\psi(x-K)}{x-K}
  [1 - \varphi(|x-K| t^{1/2-2\alpha + 2\rho})] \Big] \Big| \big[ 1 -\varphi(|x| t^{1/2-\alpha}) \big] \,dx,
\\
|C_3| & \lesssim \int |g(x)| \frac{1}{|x-K|} \big[1 - \varphi(|x-K| t^{1/2-2\alpha + 2\rho}) \big]
  \Big| \partial_x \Big[ \frac{1}{x} \big( 1 -\varphi(|x| t^{1/2-\alpha}) \big) \Big] \Big| \,dx.
\end{split}
\end{align}
We can bound the first term by
\begin{align*}
\begin{split}
|C_1| & \lesssim {\| g^\prime \|}_{L^2} t^{1/2-2\alpha+2\rho}
\left( \int \frac{1}{|x|^2} \big[1 - \varphi(|x| t^{1/2-\alpha}) \big] \,dx \right)^{1/2} \\
& \lesssim t^{1/4-\alpha} \cdot t^{1/2 - 2\alpha + 2\rho} \cdot (t^{1/2-2\rho})^{1/2}
   \lesssim t^{1-7\alpha/2 + 2\rho}.
\end{split}
\end{align*}
We can estimate the second term by
\begin{align*}
|C_2| \lesssim \| g \|_{L^\infty} t^{1/2-\alpha}
\lesssim \int\Big| \partial_x \Big[ \frac{\psi(x-K)}{x-K} [1 - \varphi(|x-K| t^{1/2-2\alpha + 2\rho})] \Big] \Big|\,dx
\lesssim t^{1-3\alpha+2\rho}.
\end{align*}
Finally, $C_3$ can be bounded similarly. Optimizing over $\rho$ leads to the choice $\rho = \alpha/3$, which gives the desired result.
\end{proof}

\medskip
\subsection{Asymptotics for $\mathcal{N}_S + \mathcal{N}_L$}\label{ssas}

Let us recall Proposition \ref{muprop} and that we have decomposed, see \eqref{dtfdec},
\begin{align}
\label{duhamel10}
\begin{split}
& i \partial_t \widetilde{f}(t,k) = \frac{1}{4\pi^2} \big[ \mathcal{N}_S + \mathcal{N}_L + \mathcal{N}_R \big],
\\
& \mathcal{N}_\ast(t,k) =
  \iiint e^{it (-k^2 + \ell^2 - m^2 + n^2)} \widetilde{f}(t,\ell) \overline{\widetilde{f}(t,m)} \widetilde{f}(t,n)
  \mu_\ast (k,\ell,m,n) \,d\ell dm dn.
\end{split}
\end{align}
By Lemma~\ref{NR}, $|\mathcal{N}_R| \lesssim \frac{1}{t^{5/4}}$, which will be an acceptable error. Therefore, we focus on $\mathcal{N}_S + \mathcal{N}_L$.

\subsubsection{Setting up the spectral measure}
We now want to derive asymptotics for $\mathcal{N}_S + \mathcal{N}_L$. 
For this purpose it is convenient to rewrite slightly the expressions for the measures
$\mu_S (k,\ell,m,n) = \mu_+ (k,\ell,m,n) + \mu_- (k,\ell,m,n)$ and $\mu_L (k,\ell,m,n) = \mu_L^+ (k,\ell,m,n) + \mu_L^- (k,\ell,m,n)$
by going back to the decomposition of $\psi(x,k)$.
In particular, we will denote
\begin{align}
\label{psiSL}
\begin{split}
& \psi_S(x,k) + \psi_L(x,k) = \psi_+(x,k) + \psi_-(x,k),
\\
& \psi_+(x,k) = \big[ T(k) e^{ikx} \mathbf{1}_+(k) +  \big( e^{ikx} + R_+(-k) e^{-ikx} \big) \mathbf{1}_-(k) \big] \chi_+(x),
\\
& \psi_-(x,k) = \big[ \big( e^{ikx} + R_-(k) e^{-ikx} \big) \mathbf{1}_+(k) + T(-k) e^{ikx} \mathbf{1}_-(k) \big] \chi_-(x).
\end{split}
\end{align}
in order to distinguish more easily the contribution from positive and negative $x$ and $k$.

By definition, see \eqref{muS}-\eqref{muL-}, we can write
\begin{align}
\label{muSL}
\mu_S (k,\ell,m,n) + \mu_L (k,\ell,m,n) = \nu_+(k,\ell,m,n) + \nu_-(k,\ell,m,n),
\end{align}
where
\begin{align}
\label{nupm}
\begin{split}
\nu_\pm(k,\ell,m,n) & = \int_\R \bar{\psi_\pm(x,k)} \psi_\pm(x,\ell) \bar{\psi_\pm(x,m)} \psi_\pm(x,n)\,dx
\\
& = \sum_{\b,\g,\d,\eps \in \{-1,+1\}}
  \int_\R \chi_\pm^4(x) a^\pm_{\b\g\d\eps}(k,\ell,m,n)
  \overline{e^{\b ikx}} \cdot e^{\g i\ell x} \cdot \overline{e^{\d imx}} \cdot e^{\eps inx} \, dx
\\
& = \sum_{\b,\g,\d,\eps \in \{-1,+1\}} a^\pm_{\b\g\d\eps}(k,\ell,m,n) \what{\varphi_\pm}(\beta k - \g\ell + \d m - \eps n),
\qquad \varphi_\pm = \chi_\pm^4.
\end{split}
\end{align}
By~\eqref{a+-sym} and~\eqref{a+-symbis}, the coefficients can be written
\begin{align}
a^\pm_{\b\g\d\eps}(k,\ell,m,n) = \bar{a^\pm_{\b}(k)} \cdot a^\pm_{\g}(\ell) \cdot \bar{a^\pm_{\d}(m)} \cdot a^\pm_{\eps}(n)
\end{align}
where
\begin{equation}
\label{a+}
a_\epsilon^+(k) = \left\{ \begin{array}{ll}
T(k) & \mbox{if $\epsilon = +1$, $k>0$} \\
1 & \mbox{if $\epsilon=+1$, $k<0$} \\
0 & \mbox{if $\epsilon=-1$, $k>0$} \\
R_+(-k) & \mbox{if $\epsilon=-1$, $k<0$}
\end{array} \right.
\end{equation}
and
\begin{equation}
\label{a-}
a_\epsilon^-(k) = \left\{ \begin{array}{ll}
1 & \mbox{if $\epsilon = +1$, $k>0$} \\
T(-k) & \mbox{if $\epsilon=+1$, $k<0$} \\
R_-(k) & \mbox{if $\epsilon=-1$, $k>0$} \\
0 & \mbox{if $\epsilon=-1$, $k<0$}.
\end{array} \right.
\end{equation}

According to \eqref{muSL} we have
\begin{align}
\label{I+-}
\begin{split}
& \mathcal{N}_S + \mathcal{N}_L = \mathcal{I}^+ + \mathcal{I}^-
\\
& \mathcal{I}^\pm(t,k) = \iiint e^{it (-k^2 + \ell^2 - m^2 + n^2)} \wt{f}(t,\ell) \overline{\wt{f}(t,m)} \wt{f}(t,n)
  \nu_\pm (k,\ell,m,n) \,d\ell \, dm  \,dn.
\end{split}
\end{align}
We now proceed to find asymptotic expressions for these integrals.
The upshot of these calculations is stated at the end of the subsection in Claim \ref{proODE}.

\medskip
\subsubsection{Asymptotics for $\mathcal{I}^+$}
Using formula \eqref{nupm} we can write
\begin{align}
\begin{split}
\mathcal{I}^+(t,k)
= \sum_{\b,\g,\d,\eps \in \{-1,+1\}} \iiint e^{it (-k^2 + \ell^2 - m^2 + n^2)}
  a^+_{\b\g\d\eps}(k,\ell,m,n) \, \wt{f}(t,\ell) \overline{\wt{f}(t,m)} \wt{f}(t,n)
  \\ \times \what{\varphi_+}(\beta k - \g\ell + \d m - \eps n) \,d\ell \,dm\, dn.
\end{split}
\end{align}
Since we are often going to have sums over all possible sign combinations, for brevity we will adopt the short-hand notation
\begin{align}
 \label{sumnot}
\sum_{\ast} := \sum_{\b,\g,\d,\eps \in \{-1,+1\}}.
\end{align}

Recalling the formula \eqref{decphi-},
$$
\what{\varphi_+}(k) = \sqrt{\frac{\pi}{2}} \delta_0 + {\what{\phi}(k) \over ik} - \widehat{\psi},
$$
we can change variables and split into three parts as before:
\begin{align}
\label{I+split}
\begin{split}
 \mathcal{I}^+(t,k) = \sqrt{\frac{\pi}{2}} \mathcal{I}^+_0(t,k)-i\mathcal{I}^+_V(t,k) + \mathcal{I}^+_{V,r}(t,k),
\end{split}
\end{align}
where, denoting
\begin{align}
\label{gaf+}
g^+_\rho (y) := a^+_\rho(y) \wt{f}(t,y),
\end{align}
(omitting the time variable), we have
\begin{align}
\label{I+0}
\begin{split}
\mathcal{I}^+_0(t,k) = \sum_{\ast} \iint e^{it (-k^2 + (\b k + \d m -\eps n)^2 - m^2 + n^2)}
  \overline{a^+_{\b}(k)} \, g^+_\g(\g(\b k + \d m -\eps n)) \overline{g^+_\d(m)} g^+_\eps(n) \, dm\, dn,
\end{split}
\end{align}
\begin{align}
\label{I+V}
\begin{split}
\mathcal{I}^+_V(t,k) = \sum_{\ast} \iiint e^{it (-k^2 + (-p + \b k + \d m -\eps n)^2 - m^2 + n^2)}
  \overline{a^+_{\b}(k)} \, g^+_\g(\g(-p+\b k + \d m -\eps n)) \\ \overline{g^+_\d(m)} g^+_\eps(n)
  \, \frac{\what{\phi}(p)}{p} \, dm\, dn \, dp,
\end{split}
\end{align}
and
\begin{align}
\label{I+Vr}
\begin{split}
\mathcal{I}^+_{V,r}(t,k) = \sum_{\ast} \iiint e^{it (-k^2 + (-p + \b k + \d m -\eps n)^2 - m^2 + n^2)}
  \overline{a^+_{\b}(k)} \, g^+_\g(\g(-p+\b k + \d m -\eps n)) \\ \overline{g^+_\d(m)} g^+_\eps(n) \,\what{\psi}(p) \, dm\, dn \, dp.
\end{split}
\end{align}

\medskip
\noindent{\it Asymptotics for $\mathcal{I}^+_0$}.
This is similar to the case of flat NLS treated in \cite{KP}; it follows from Lemma~\ref{AsLem1'} that
\begin{align}
\label{I+0as1}
\begin{split}
\mathcal{I}^+_0(t,k)
& = \frac{\pi}{|t|} \sum_{\ast} \overline{a^+_{\b}(k)} \, g^+_\g(\g\b k) \overline{g^+_\d(\delta \beta k)} g^+_\eps(\eps\beta k)
  + O(|t|^{-1-\alpha/3})
\\
& = \frac{\pi}{|t|} \sum_{\ast} \overline{a^+_{\b}(k)} \, a^+_\g(\g\b k) \overline{a^+_\d(\d\b k)} a^+_\eps(\eps\b k)
  \wt{f}(\g\b k) \overline{\wt{f}(\d\b k)} \wt{f}(\eps\b k)
  + O(|t|^{-1-\alpha/3})
\end{split}
\end{align}
For $k > 0$, recall from \eqref{a+} that $a^+_+(k) = T(k)$, $a^+_-(k) = 0$, $a^+_-(-k) = R_+(k)$,
so that the sum in \eqref{I+0as1} reduces to
\begin{align}
\label{I+0as2}
\begin{split}
& T(-k) \sum_{\g,\d,\eps \in \{+1,-1\}} a^+_\g(\g k) \overline{a^+_\d(\d k)}
  a^+_\eps(\eps k) \wt{f}(\g k) \overline{\wt{f}(\d k)} \wt{f}(\eps k)
\\
& =T(-k) \big|T(k)\wt{f}(k) + R_+(k)\wt{f}(-k)\big|^2 \big( T(k)\wt{f}(k) + R_+(k)\wt{f}(-k) \big).
\end{split}
\end{align}
Similarly, since for $k<0$ we have $a^+_+(k) = 1$, $a^+_-(k) = R_+(-k)$, $a^+_+(-k) = T(-k)$ and $a^+_-(-k) = 0$,
the sum in \eqref{I+0as1} is given by
\begin{align}
\label{I+0as3}
\begin{split}
& \sum_{\g,\d,\eps \in \{+1,-1\}} a^+_\g(\g k) \overline{a^+_\d(\d k)} a^+_\eps(\eps k) \wt{f}(\g k) \overline{\wt{f}(\d k)} \wt{f}(\eps k)
  \\ & \quad + R_+(k) \sum_{\g,\d,\eps \in \{+1,-1\}} a^+_\g(-\g k) \overline{a^+_\d(-\d k)}
  a^+_\eps(-\eps k) \wt{f}(-\g k) \overline{\wt{f}(-\d k)} \wt{f}(-\eps k)
\\ & \quad \qquad = |\wt{f}(k)|^2 \wt{f}(k) +
R_+(k) \big|T(-k)\wt{f}(-k) + R_+(-k)\wt{f}(k)\big|^2 \big( T(-k)\wt{f}(-k) + R_+(-k)\wt{f}(k) \big)
\end{split}
\end{align}

In conclusion, if we define
\begin{align}
\label{N+}
\mathcal{N}^+[f](k) := \big|T(k)\wt{f}(k) + R_+(k)\wt{f}(-k)\big|^2 \big( T(k)\wt{f}(k) + R_+(k)\wt{f}(-k) \big)
\end{align}
we have
\begin{align}
\label{I+0as10}
\begin{split}
\mathcal{I}^+_0(t,k) & = \frac{\pi}{|t|} \Big[ T(-k) \mathcal{N}^+[f](k) \mathbf{1}_+(k)
 + \Big(|\wt{f}(k)|^2 \wt{f}(k) +  R_+(k) \mathcal{N}^+[f](-k) \Big) \mathbf{1}_-(k) \Big]
 + O(|t|^{-1-\alpha/3}).
\end{split}
\end{align}

\medskip
\noindent{\it Asymptotics for $\mathcal{I}^+_V$}.
We now use Lemmas \ref{AsLem1} and \ref{AsLem2} to derive asymptotics:
we see that $\mathcal{I}^+_V$ is an operator of the form \eqref{I1}
with $g = (g^+_\g,g^+_\d,g^+_\eps)$ satisfying the assumptions \eqref{AsLem1as}.
Applying Lemma \ref{AsLem1} we then obtain
\begin{align*}
\begin{split}
\mathcal{I}^+_V(t,k) & = \sum_{\ast} \overline{a^+_{\b}(k)} I[g^+_\g,g^+_\d,g^+_\eps](t,k)
\\
& = \sum_{\ast} \overline{a^+_{\b}(k)} \, \frac{\pi}{|t|} e^{-itk^2}
  \int_\R e^{itq^2} g^+_\g(-\g q) \overline{g^+_\d(- \delta q)} g^+_\eps(- \eps q)
  \, \frac{\widehat{\phi}(q+\beta k)}{q+\beta k} \, dp + O(|t|^{-1-\alpha/3}).
\end{split}
\end{align*}
Applying Lemma \ref{AsLem2} to this last expression, noticing that the assumptions \eqref{AsLem2as} hold, we obtain
\begin{align}
\label{I+Vas1}
\begin{split}
\mathcal{I}^+_V(t,k) & = \frac{\pi}{|t|} \sum_{\ast} e^{-itk^2} h(t,-\sqrt{|t|} \beta k) \overline{a^+_{\b}(k)}
  g^+_\g(\g\b k) \overline{g^+_\d(\delta\b k)} g^+_\eps(\eps\b k)
  + O(|t|^{-1-\alpha/3})
\\
\end{split}
\end{align}
where $h$ denotes the function from \eqref{AsLem2h} with $\psi(0) = \what{\phi}(0) = 1/\sqrt{2\pi}$.
To write out more explicitly the sum \eqref{I+Vas1} we proceed as above, using the formulas \eqref{a+}
and looking at the cases $k>0$ and $k<0$, eventually obtaining
\begin{align}
\label{I+Vas10}
\begin{split}
\mathcal{I}^+_V(t,k) & = \frac{\pi}{|t|} e^{-itk^2}  \Big[ h(t,-\sqrt{|t|}k) {T(-k)} \mathcal{N}^+[f](k) \mathbf{1}_+(k)
 \\ & + \Big(h(t,-\sqrt{|t|}k) |\wt{f}(k)|^2 \wt{f}(k)
 +  h(t,\sqrt{|t|}k) {R_+(k)} \mathcal{N}^+[f](-k) \Big) \mathbf{1}_-(k) \Big]
 + O(|t|^{-1-\alpha/3}),
\end{split}
\end{align}
where $\mathcal{N}^+[f](k)$ is defined in \eqref{N+}.

\medskip
\noindent{\it The term $\mathcal{I}^+_{V,r}$.}
This is a remainder term that decays faster than $|t|^{-1-\rho}$ and therefore does not contribute to the asymptotic behavior of solutions.
To see this, we can change variables as done before, cfr. \eqref{NVr} and \eqref{NVr2}, and write the term in \eqref{I+Vr} as
\begin{align}
\label{I+Vr2}
I^{+}_{V,r}(t,k) = \sum_{\beta\in\{1,-1\}} \overline{a_\beta^+(k)}
  \int e^{it (-k^2 + q^2)}\, \mathbf{1}_+(k) \, I(t,q) \, \what{\psi}(\beta k - q) \, dq
\end{align}
where, similarly to \eqref{Idef},
\begin{align*}
I(t,q) = \sum_{\g,\d,\eps \in \{-1,+1\}} \g\d\eps
 \iint e^{2it ab} \wt{g}^+_\gamma(t,\gamma(q-a)) \overline{\wt{g}^+_\delta(t,\delta(b-a+q))} \wt{g}^+_\eps(t,\eps(b+q)) \, da\, db.
\end{align*}
In particular, arguing as in \eqref{dkI} and \eqref{Idet}, we have
\begin{align*}
|t| {\| I(t) \|}_{L^2} + |t|^{3/4}{\| \partial_q I(t) \|}_{L^2} \lesssim \e_1^3.
\end{align*}
for $|t|\geq 1$.
Using this it is not hard to see how to estimate \eqref{I+Vr2}, so we just sketch the argument.
When the integral is taken over $|q|\leq |t|^{-1/2}$, we can directly use H\"older's inequality to bound the $L^\infty_k$ norm of \eqref{I+Vr2} by
\begin{align*}
{\| I(t) \|}_{L^2}|t|^{-1/4} \lesssim \e_1^3 |t|^{-5/4}.
\end{align*}
If instead $|q|\geq |t|^{-1/2}$ in the support of the integral in \eqref{I+Vr2}, we can integrate by parts in $q$ obtaining the bound
\begin{align*}
& \frac{1}{|t|} \int_\R \Big| \partial_q \big[q^{-1} I(t,q) \varphi_{\geq0}(q|t|^{1/2}) \what{\psi}(\beta k-q)  \Big| \, dq
\\
& \lesssim \frac{1}{|t|} \Big[ |t|^{3/4}{\| I(t) \|}_{L^2} + |t|^{1/4}{\| \partial_q I(t) \|}_{L^2} \Big] \lesssim \e_1^3 |t|^{-5/4}.
\end{align*}

\medskip
\subsubsection{Asymptotics for $\mathcal{I}^-$}
Using formula \eqref{nupm} we can write
\begin{align}
\begin{split}
\mathcal{I}^-(t,k)
= \sum_{\b,\g,\d,\eps \in \{-1,+1\}} \iiint e^{it (-k^2 + \ell^2 - m^2 + n^2)}
  a^-_{\b\g\d\eps}(k,\ell,m,n) \, \wt{f}(t,\ell) \overline{\wt{f}(t,m)} \wt{f}(t,n)
  \\ \times \what{\varphi_-}(\beta k - \g\ell + \d m - \eps n) \,d\ell dm dn.
\end{split}
\end{align}
As before, we can write $\displaystyle \what{\varphi_-}(k) = \sqrt{\frac{\pi}{2}} \delta_0 - \frac{\what{\phi}(k)}{ik}+\widehat{\psi}(k)$,
change variables and split
\begin{align}
\label{I-split}
& \mathcal{I}^-(t,k) = \sqrt{\frac{\pi}{2}} \mathcal{I}^-_0(t,k) + i\mathcal{I}^-_V(t,k) + \mathcal{I}^-_{V,r}(t,k)
\end{align}
where
\begin{align}
\label{I-0}
\begin{split}
\mathcal{I}^-_0(t,k) = \sum_{\ast} \iint e^{it (-k^2 + (\b k + \d m -\eps n)^2 - m^2 + n^2)}
  \overline{a^-_{\b}(k)} \, g^-_\g(\g(\b k + \d m -\eps n)) \overline{g^-_\d(m)} g^-_\eps(n) \, dm \, dn,
\end{split}
\end{align}
\begin{align}
\label{I-V}
\begin{split}
\mathcal{I}^-_V(t,k) = \sum_{\ast} \iiint e^{it (-k^2 + (-p + \b k + \d m -\eps n)^2 - m^2 + n^2)}
  \overline{a^-_{\b}(k)} \, g^-_\g(\g(-p+\b k + \d m -\eps n)) \\ \overline{g^-_\d(m)} g^-_\eps(n)
  \, \frac{\widehat{\phi}(p)}{p} \, dm \,dn \,dp,
\end{split}
\end{align}
\begin{align}
\label{I-Vr}
\begin{split}
\mathcal{I}^-_{V,r}(t,k) = \sum_{\ast} \iiint e^{it (-k^2 + (-p + \b k + \d m -\eps n)^2 - m^2 + n^2)}
  \overline{a^-_{\b}(k)} \, g^-_\g(\g(-p+\b k + \d m -\eps n)) \\ \overline{g^-_\d(m)} g^-_\eps(n)
  \, \widehat{\psi}(p) \, dm \,dn \,dp,
\end{split}
\end{align}
and we have denoted
\begin{align}
\label{gaf-}
g^-_\rho (y) := a^-_\rho(y) \wt{f}(t,y).
\end{align}
The term \eqref{I-Vr} is a remainder term which satisfies $$|\mathcal{I}^-_{V,r}(t,k)| \lesssim \e_1^3|t|^{-5/4},$$
as it can be seen by applying the same argument used for the term $\mathcal{I}^-_{V,r}$ in \eqref{I+Vr} and \eqref{I+Vr2} above.

\medskip
\noindent{\it Asymptotics for $\mathcal{I}^-_0$}.
By Lemma~\ref{AsLem1'},
\begin{align}
\label{I-0as1}
\begin{split}
\mathcal{I}^-_0(t,k) & = \iint e^{2it ab}
  \overline{a^-_{\b}(k)} \, g^-_\g(\g(\b k + a)) \overline{g^-_\d(\d(\b k + a-b))} g^-_\eps(\eps(\b k-b)) \, da db
\\
& = \frac{\pi}{|t|} \sum_{\ast} \overline{a^-_{\b}(k)} \, a^-_\g(\g\b k) \overline{a^-_\d(\d\b k)} a^-_\eps(\eps\b k)
  \wt{f}(\g\b k) \overline{\wt{f}(\d\b k)} \wt{f}(\eps\b k)
  + O(|t|^{-1-\alpha/3}).
\end{split}
\end{align}
For $k > 0$ we have $a^-_+(k) = 1$, $a^-_-(-k)=0$, $a^-_+(-k)= T(k)$ and $a^-_-(k) = R_-(k)$, and therefore the above sum is
\begin{align}
\label{I-0as2}
\begin{split}
& \sum_{\g,\d,\eps \in \{1,-1\}} a^-_\g(\g k)
  \overline{a^-_\d(\d k)} a^-_\eps(\eps k) \wt{f}(\g k) \overline{\wt{f}(\d k)} \wt{f}(\eps k)
\\
& \quad + {R_-(-k)} \sum_{\g,\d,\eps \in \{1,-1\}} a^-_\g(-\g k)
  \overline{a^-_\d(-\d k)} a^-_\eps(-\eps k) \wt{f}(-\g k) \overline{\wt{f}(-\d k)} \wt{f}(-\eps k)
\\
& \quad \quad = |\wt{f}(k)|^2 \wt{f}(k) +{R_-(-k)} \big|T(k)\wt{f}(-k) + R_-(k) \wt{f}(k)\big|^2 \big(T(k)\wt{f}(-k) + R_-(k) \wt{f}(k)\big).
\end{split}
\end{align}
Similarly, since for $k<0$ we have $a^-_+(k) = T(-k)$, $a^-_-(k) = 0$, and $a^-_-(-k) = R_-(-k)$,
we obtain
\begin{align}
\label{I-0as3}
\begin{split}
&{T(k)} \sum_{\g,\d,\eps \in \{1,-1\}} a^-_\g(\g k)
  \overline{a^-_\d(\d k)} a^-_\eps(\eps k) \wt{f}(\g k) \overline{\wt{f}(\d k)} \wt{f}(\eps k)
\\
& = {T(k)} \big|T(-k)\wt{f}(k) + R_-(-k) \wt{f}(-k)\big|^2 \big(T(-k)\wt{f}(k) + R_-(-k) \wt{f}(-k)\big)
\end{split}
\end{align}
By letting
\begin{align}
\label{N-}
\begin{split}
\mathcal{N}^-[f](k) & := \big|T(k)\wt{f}(-k) + R_-(k)\wt{f}(k)\big|^2 \big( T(k)\wt{f}(-k) + R_-(k)\wt{f}(k) \big)
\end{split}
\end{align}
we have showed that
\begin{align}
\label{I-0as10}
\begin{split}
\mathcal{I}^-_0(t,k) = \frac{\pi}{|t|} \Big[\Big(|\wt{f}(k)|^2 \wt{f}(k) + {R_-(-k)} \mathcal{N}^-[f](k) \Big) \mathbf{1}_+(k)
  + {T(k)} \mathcal{N}^-[f](-k) \mathbf{1}_-(k) \Big]
  \\ + O(|t|^{-1-\alpha/3}).
\end{split}
\end{align}

\medskip
\noindent{\it Asymptotics for $\mathcal{I}^-_V$}.
From the formula \eqref{I-V}, the definition \eqref{gaf-}, and the properties \eqref{TRk},
we see that $\mathcal{I}^-_V$ is an operator of the form \eqref{I1} appearing in Lemma \ref{AsLem1},
with $g = (g^-_\g,g^-_\d,g^-_\eps)$ satisfying the assumptions \eqref{AsLem1as}.
Applying Lemma \ref{AsLem1} we then obtain
\begin{align*}
\begin{split}
\mathcal{I}^-_V(t,k) & = \sum_{\ast} \overline{a^-_{\b}(k)} I[g^-_\g,g^-_\d,g^-_\eps](t,k)
\\ & = \sum_{\ast} \overline{a^-_{\b}(k)} \, \frac{\pi}{|t|} e^{-itk^2}
  \int_\R e^{itq^2} g^-_\g(-\g q) \overline{g^-_\d(-\delta q)} g^-_\eps(-\eps q)
  \, \frac{\widehat{\phi}(q+\beta k)}{q+\beta k} \, dq + O(|t|^{-1-\alpha/3}).
\end{split}
\end{align*}
Applying Lemma \ref{AsLem2} to this last expression, noticing that the assumption \eqref{AsLem2as} holds, we obtain
\begin{align}
\label{I-Vas1}
\begin{split}
\mathcal{I}^-_V(t,k) & = \frac{\pi}{|t|} \sum_{\ast} e^{-itk^2} h(t,-\sqrt{|t|} \beta k) \overline{a^-_{\b}(k)}
  g^-_\g(\g\b k) \overline{g^-_\d(\delta\b k)} g^-_\eps(\eps\b k)
  + O(|t|^{-1-\alpha/3}).
\end{split}
\end{align}
To write out more explicitly \eqref{I-Vas1} we proceed as above, using the formulas \eqref{a-}, to get
\begin{align}
\label{I-Vas10}
\begin{split}
\mathcal{I}^-_V(t,k) = \frac{\pi}{|t|} e^{-itk^2} \Big[ \Big(h(t,-\sqrt{|t|} k) |\wt{f}(k)|^2 \wt{f}(k)
   + h(t,\sqrt{|t|} k) {R_-(-k)} \mathcal{N}^-[f](k) \Big) \mathbf{1}_+(k)
\\
+ h(t,-\sqrt{|t|}k) {T(k)} \mathcal{N}^-[f](-k) \mathbf{1}_-(k) \Big]
+ O(|t|^{-1-\alpha/3}).
\end{split}
\end{align}

Putting together the results above, starting from the decomposition of $i\partial_t \wt{f}$ in \eqref{duhamel10},
the definitions of $\mathcal{I}_+$ and $\mathcal{I}_-$ in \eqref{I+-}, their decompositions \eqref{I+split} and \eqref{I-split}
and using the asymptotic expansions obtained in \eqref{I+0as10}, \eqref{I+Vas10}, \eqref{I-0as10} and \eqref{I-Vas10},
and the estimate \eqref{NRest} for $\mathcal{N}_R$, we have obtained the following

\begin{claim}\label{proODE}
Let $f$ be the profile defined in \eqref{prof0}. Under the apriori assumptions \eqref{apriori0}-\eqref{bootstrap} we have, for $k>0$,
\begin{align}
\label{secas1}
\begin{split}
i\partial_t \wt{f}(k) = \frac{1}{4\pi |t|} \Big[ \sqrt{\frac{\pi}{2}}{T(-k)} \mathcal{N}^+[f](k)
  + \sqrt{\frac{\pi}{2}}|\wt{f}(k)|^2 \wt{f}(k) + \sqrt{\frac{\pi}{2}}{R_-(-k)} \mathcal{N}^-[f](k)\hfill
  \\ -i e^{-ik^2t} h(t,-\sqrt{|t|}k) {T(-k)} \mathcal{N}^+[f](k)
  + i e^{-ik^2t}  h(t,-\sqrt{|t|} k) |\wt{f}(k)|^2 \wt{f}(k) \\ +i e^{-ik^2t} h(t,\sqrt{|t|} k) {R_-(-k)} \mathcal{N}^-[f](k) \Big]
  + O(|t|^{-1-\alpha/3}),
\end{split}
\end{align}
and
\begin{align}
\label{secas2}
\begin{split}
i\partial_t \wt{f}(-k) = \frac{1}{4\pi |t|} \Big[ \sqrt{\frac{\pi}{2}} |\wt{f}(-k)|^2 \wt{f}(-k)
  + \sqrt{\frac{\pi}{2}} {R_+(-k)} \mathcal{N}^+ [f](k)
  + \sqrt{\frac{\pi}{2}} {T(-k)} \mathcal{N}^-[f](k)
  \\ -i e^{-ik^2t} h(t,\sqrt{|t|}k) |\wt{f}(-k)|^2 \wt{f}(-k)
  - i e^{-ik^2t} h(t,-\sqrt{|t|}k) {R_+(-k)} \mathcal{N}^+[f](k)
  \\ + i e^{-ik^2t} h(t,\sqrt{|t|}k) {T(-k)} \mathcal{N}^-[f](k) \Big]
  + O(|t|^{-1-\alpha/3}),
\end{split}
\end{align}
where we are using the notation \eqref{N+} and \eqref{N-} for $\mathcal{N}^\pm[f]$,
and $h$ is as in \eqref{AsLem2h} with $\psi(0) = 1/\sqrt{2\pi}$.
\end{claim}

\medskip
\subsection{The asymptotic ODE and proof of the $L^\infty$ bound}\label{secas}
We now want to analyze the ODE \eqref{secas1}-\eqref{secas2}
and identify the necessary structure that will guarantee the boundedness of its solutions.
To this end let us define
\begin{align}
\label{secas3}
Z(k) := \big(\wt{f}(k), \wt{f}(-k)\big), 
\qquad b(t,y) := \frac{1}{4\pi} \Big[ \sqrt{\frac{\pi}{2}} -i e^{-iy^2}  h(t,y) \Big]
\end{align}
where, see \eqref{AsLem2h} and recall the choice $\psi(0) = \what{\phi}(0) = 1/\sqrt{2\pi}$,
\begin{align}
\label{secas4}
-i e^{-iy^2}h(t,y) = \frac{1}{\sqrt{2\pi}}\int e^{ i2xy + ix^2} \frac{1}{ix} \varphi(|x||t|^{-2\alpha+2\rho})\, dx, \qquad t>0,
\end{align}
and $h(t,y) = \overline{h(-t,y)}$ when $t<0$.

Recall that $h$ is an odd function in $y$.
In what follows we will sometimes omit the dependence of $b$ and $h$ on the variable $t$.
With the above definitions, the equations \eqref{secas1}-\eqref{secas2} become
\begin{align}
\label{secas1.5}
\begin{split}
i\partial_t \wt{f}(k) = \frac{1}{t} \Big[ b(\sqrt{|t|}k) |\wt{f}(k)|^2\wt{f}(k) + b(-\sqrt{|t|}k)\overline{T(k)} \mathcal{N}^+[f](k)
  + b(-\sqrt{|t|} k) \overline{R_-(k)} \mathcal{N}^-[f](k) \Big]
  \\ + O(|t|^{-1-\rho}),
\end{split}
\end{align}
and
\begin{align}
\label{secas2.5}
\begin{split}
i\partial_t \wt{f}(-k) = \frac{1}{t}\Big[ b(\sqrt{|t|}k) |\wt{f}(-k)|^2\wt{f}(-k)
  + b(-\sqrt{|t|}k) \overline{R_+(k)} \mathcal{N}^+[f](k) + b(-\sqrt{|t|}k) \overline{T(k)} \mathcal{N}^-[f](k) \Big]
  \\ + O(|t|^{-1-\rho}).
\end{split}
\end{align}

It is then convenient to write \eqref{secas1.5}-\eqref{secas2.5} in matrix form. Recalling the definition of the (unitary) scattering matrix
\begin{align}
\label{secasS}
S(k) :=
\left( \begin{array}{cc}
T(k)  & R_+(k) \\ R_-(k) & T(k)
\end{array}
 \right), \qquad
S^{-1}(k) :=
\left( \begin{array}{cc}
\overline{T(k)} & \overline{R_-(k)} \\  \overline{R_+(k)} & \overline{T(k)}
\end{array}
 \right),
\end{align}
using the definitions in \eqref{N+} and \eqref{N-}, we see that
\begin{align}
\begin{split}
\mathcal{N}^+[f](k) = \big|(S(k)Z(k))_1\big|^2 (S(k)Z(k))_1, \\ \mathcal{N}^-[f](k) = \big|(S(k)Z(k))_2\big|^2 (S(k)Z(k))_2.
\end{split}
\end{align}
where the index $j=1,2$ denotes the $j$-th component of a vector.
We then have obtained the following:

\begin{claim}
The equation \eqref{secas1}-\eqref{secas2} can be written in vector form as
\begin{align}
\label{secasODE}
i \partial_t Z(t,k) 
= \frac{1}{t} \mathcal{A}(t,k) Z(t,k) + O(|t|^{-1-\rho}),
\end{align}
for $\rho \in (0,\alpha/10)$, where
\begin{align}
\label{secasmat}
\mathcal{A}(t,k) := b(\sqrt{|t|}k) \mathrm{diag} \big( |Z_1|^2, |Z_2|^2 \big)
  + b(-\sqrt{|t|}k) S^{-1} \mathrm{diag} \big( |(SZ)_1|^2, |(SZ)_2|^2 \big) S.
\end{align}
\end{claim}

To understand \eqref{secas3}-\eqref{secas4} for large $t$ we will use the following lemma:

\begin{lem}\label{secaslem1.5}
Let $c(t,y) = -i e^{-iy^2}h(t,y)$ be the expression in \eqref{secas4}.
For all $y\in \R$ and $t>0$
such that $y \geq |t|^{1/4}$ we have
\begin{align}
\label{secaslem1.51}
\Big| c(t,y) - \sqrt{\frac{\pi}{2}} \Big| \lesssim |y|^{-1/2}.
\end{align}

In particular, from the definition of $b$ and $h$ in \eqref{secas3}-\eqref{secas4} above, we have the following:
for $t>0$
\begin{align}
\label{secaslem1.52}
\begin{split}
\Big| b(t,y)- \frac{1}{2\sqrt{2\pi}} \Big| \lesssim |y|^{-1/2}, \qquad & y \geq t^{1/4},
\\
|b(y)| \lesssim |y|^{-1/2}, \qquad & y \leq -t^{1/4},
\end{split}
\end{align}
while for $t<0$
\begin{align}
\label{secaslem1.53}
\begin{split}
\Big| b(t,y) - \frac{1}{4\sqrt{2\pi}}(1+e^{-2iy^2}) \Big| \lesssim |y|^{-1/2}, \qquad & y \geq |t|^{1/4},
\\
\Big| b(t,y) - \frac{1}{4\sqrt{2\pi}}(1-e^{-2iy^2}) \Big| \lesssim |y|^{-1/2}, \qquad & y \leq -|t|^{1/4}.
\end{split}
\end{align}
\end{lem}

\begin{proof}
Using \eqref{Fsign}, we write
\begin{align*}
\begin{split}
& c(y) = c_1(y) + c_2(y) + \sqrt{\frac{\pi}{2}}\sign(2y),
\\
& c_1(y) := \frac{1}{\sqrt{2\pi}} \int e^{i2xy} (e^{ix^2} - 1) \frac{1}{ix} \varphi(|x|t^{-2\alpha+2\rho})\, dx,
\\
& c_2(y) := \frac{1}{\sqrt{2\pi}} \int e^{i2xy} \frac{1}{ix} \big[ \varphi(|x|t^{-2\alpha+2\rho}) - 1\big]\, dx.
\end{split}
\end{align*}
In $c_1$ we see that  the integrand is bounded by $|x|$ which, for $|x| \leq |y|^{-1/4}$, gives the desired bound.
For $|x| \geq |y|^{-1/4}$ instead we can integrate by parts to obtain:
\begin{align}
\label{secas21}
\begin{split}
& \left| \int_{|x|\geq |y|^{-1/4}} e^{i 2xy} (e^{ix^2} - 1) \frac{1}{x}  \varphi(|x|t^{-2\alpha+2\rho})\, dx \right|
\\ & \lesssim \frac{1}{|y|} \left| \int_{|x|\geq |y|^{-1/3}} e^{i 2xy}
  \partial_x\Big( (e^{ix^2} - 1) \frac{1}{x} \varphi(|x|t^{-2\alpha+2\rho}) \Big) \, dx \right|
  \\
& \lesssim \frac{1}{|y|} \left| \int_{|x|\geq |y|^{-1/4}} \Big(\frac{1}{|x|^2} + 1\Big) \varphi(|x|t^{-2\alpha+2\rho}) \, dx
  + \int_{|x|\geq |y|^{-1/4}} \varphi^\prime(|x|t^{-2\alpha+2\rho}) t^{-2\alpha+2\rho} \, dx \right|
\lesssim |y|^{-1/2},
\end{split}
\end{align}
having used that $|y| \geq |t|^{1/4} \gg |t|^{-2\alpha+2\rho} \gtrsim |x|$ on the support of the integral.
A similar integration by parts argument can be used to estimate $c_2$ by showing
\begin{align*}
\begin{split}
\left| \int e^{i2xy} \frac{1}{x} \big[ \varphi(|x|t^{-2\alpha+2\rho}) - 1\big]\, dx \right|
  \lesssim \frac{1}{|y|}t^{-2\alpha+2\rho} +
  \frac{1}{|y|} \left| \int_{|x| \gtrsim t^{2\alpha-2\rho}} \frac{1}{x^2} \, dx \right| \lesssim |y|^{-1/2}.
\end{split}
\end{align*}
This gives us \eqref{secaslem1.51}. \eqref{secaslem1.52} follows since $b(y) = 1/(4\pi)[\sqrt{\pi/2} + c(y)]$ and $c$ is odd.
The bounds \eqref{secaslem1.53} are also a direct consequence of \eqref{secaslem1.51}
since 
$h(t,y) = \overline{h(-t,y)}$ for $t<0$ gives $c(t,y) = e^{2iy^2}\overline{c(-t,y)}$.
\end{proof}

\medskip
We can now prove our main proposition about asymptotics for $Z(k)$.

\begin{prop}\label{secaspro}
Let $S$ be the scattering matrix \eqref{secasS}, 
for $k>0$, define self-adjoint matrices
\begin{align}
\label{secas9.0}
\begin{split}
& \mathcal{S}_0 := \frac{1}{2\sqrt{2\pi}} \mathrm{diag} \big( |Z_1|^2, |Z_2|^2 \big),
\\
& \mathcal{S}_1 := \frac{1}{2\sqrt{2\pi}} S^{-1} \mathrm{diag} \big( |(SZ)_1|^2, |(SZ)_2|^2 \big) S,
\end{split}
\end{align}
and
\begin{align}
\label{secas9}
\mathcal{S}(t,k) :=
\left\{\begin{array}{ll}
\mathcal{S}_0(t,k), & \qquad t>0
\\
\\
\mathbf{1}(k\leq |t|^{-\rho})\mathcal{S}_0(t,k) +
  \mathbf{1}(k\geq |t|^{-\rho}) \dfrac{1}{2}\Big[\mathcal{S}_0(t,k) + \mathcal{S}_1(t,k)\Big], & \qquad t<0.
\end{array}
\right.
\end{align}
Define the modified profile
\begin{align}
\label{secasmod}
W (t,k):= \exp\Big(i \int_0^t \mathcal{S}(t,k) \, \frac{ds}{1+s} \Big) Z(t,k),
\end{align}
where $Z(k) = (\wt{f}(k), \wt{f}(-k))$ is the solution of \eqref{secasmat}-\eqref{secasODE}.

Then, for every $|t_1| < |t_2|$, $t_1t_2>0$, we have
\begin{align}
\label{secas10}
\big| W(t_1,k) - W(t_2,k) | \lesssim \e_1^3|t_1|^{-\rho/2},
\end{align}
for $\rho \in (0,\alpha/10)$.

In particular, $|W(t,k)| = |Z(t,k)|$ is uniformly bounded, and $W(t)$ is a Cauchy sequence in time.
If we denote $W_{\pm\infty}(k)$ its limits as $t \rightarrow \pm \infty$, these are the asymptotic profiles appearing in \eqref{mtas+} and \eqref{mtas4} respectively.
\end{prop}

\medskip
\begin{proof}
Let us look at the case $t>0$ first.
For small frequencies $|k| \ll 1$ we see from the properties of $T$ and $R_\pm$ in \eqref{TRsmallk0} that
\begin{align}
S(k)-S(0) = S(k) - \left( \begin{array}{cc} 0 & -1 \\  -1 & 0 \end{array}\right) = O(|k|)
\end{align}
Under our apriori assumptions on the boundedness of $|Z(k)|$, and since
\begin{align*}
S(0)^{-1} \mathrm{diag}\big( |(S(0) Z)_1|^2, |(S(0) Z)_2|^2 \big) S(0) = S(0)^{-1} \mathrm{diag}\big( |Z_2|^2, |Z_1|^2) S(0)\\= \mathrm{diag} (|Z_1|^2,|Z_2|^2),
\end{align*}
we see that, for all $|k| \leq t^{-\rho}$, we have
\begin{align*}
\mathcal{A}(t,k) & = [b(t,\sqrt{t}k) + b(t,-\sqrt{t}k)] \mathrm{diag} \big( |Z_1|^2, |Z_2|^2 \big) + O(|t|^{-\rho})
\\ & = \frac{1}{2\sqrt{2\pi}} \mathrm{diag} \big( |Z_1|^2, |Z_2|^2 \big) +  O(|t|^{-\rho}) =  \mathcal{S}_0(t,k) + O(|t|^{-\rho}).
\end{align*}

In the case of larger frequencies $|k| \geq t^{-\rho}$ we can write
\begin{align*}
& \Big| \mathcal{A}(t,k) - \frac{1}{2\sqrt{2\pi}} \mathrm{diag} \big( |Z_1|^2, |Z_2|^2 \big) \Big|
\\ & \lesssim \Big|b(\sqrt{t}k) -  \frac{1}{2\sqrt{2\pi}} \Big| \big| \mathrm{diag} \big( |Z_1|^2, |Z_2|^2 \big) \big|
  + |b(-\sqrt{t}k)| \big| S^{-1} \mathrm{diag} \big( |(SZ)_1|^2, |(SZ)_2|^2 \big) S \big|
\\
& \lesssim O(|t|^{-\rho}),
\end{align*}
having used \eqref{secaslem1.52} in Lemma \ref{secaslem1.5} with $y = k\sqrt{t} \geq t^{1/4}$.
It follows, see the definitions \eqref{secasmat} and \eqref{secas9.0}, that
\begin{align}
\label{secas20}
\mathcal{A}(t,k) = \mathcal{S}_0(t,k) + O(|t|^{-\rho}), \qquad t>0.
\end{align}

Let us now look at the case $t<0$. For small frequencies we can deduce as before that
\begin{align}
\label{secas30}
\mathcal{A}(t,k) 
= \mathcal{S}_0(t,k) + O(|t|^{-\rho}), \qquad t<0, \quad |k| \leq |t|^{-\rho}.
\end{align}
When $|k| \geq |t|^{-\rho}$ we use instead \eqref{secaslem1.53} in Lemma \ref{secaslem1.5} to obtain, see the notation \eqref{secas9.0},
\begin{align}
 \label{secas31}
\mathcal{A}(t,k) = \frac{1}{2} \mathcal{S}_0(t,k) + \frac{1}{2} \mathcal{S}_1(t,k)
  +\frac{1}{2}e^{2ik^2t} \mathcal{S}_0(t,k) - \frac{1}{2}e^{2ik^2t} \mathcal{S}_1(t,k) + O(|t|^{-\rho}).
\end{align}

We now look at the ODE \eqref{secasODE}-\eqref{secasmat} and use \eqref{secas20}-\eqref{secas31}, and the definition of
the modified profile $W$ in \eqref{secas9}-\eqref{secasmod}, to see that, for $t>0$ we have
\begin{align*}
i \partial_t W(t,k) = O(|t|^{-1-\rho}),
\end{align*}
from which the conclusion \eqref{secas10} follows immediately when $0<t_1<t_2$.

For $t<0$ we see instead that
\begin{align}
\begin{split}
i \partial_t W(t,k) = \frac{1}{t} B(t,k)
  \,\mathbf{1}(|k|\geq t^{-\rho}) \frac{1}{2} \Big[ e^{-2ik^2t} \mathcal{S}_0(t,k) - e^{-2ik^2t} \mathcal{S}_1(t,k) \Big] + O(|t|^{-1-\rho}),
\\
B(t,k) := \exp\Big(i \int_0^t \mathcal{S}(t,k) \, \frac{ds}{1+s} \Big).
\end{split}
\end{align}
We can then integrate the right-hand side in the above equation between $t_2<t_1<0$,
and exploit the oscillations of the factors $e^{-2ik^2t}$, for $|k|\geq t^{-\rho}$, to integrate by parts.
Using the bounds
\begin{align*}
& \big| \partial_tB(t,k) \big| \lesssim \e_1^3 |t|^{-1},
\\
& \big| \partial_t \widetilde{f}(t,k) \big| = \big| \widetilde{u^3}(t,k) \big| \lesssim {\big\| u^3(t) \big\|}_{L^1}
  \lesssim \e_1^3 (1+|t|)^{-1/2},
\end{align*}
we obtain the desired conclusion \eqref{secas10}.
\end{proof}

\bigskip
\appendix
\section{Useful bounds}

\subsection{Proof of Lemma \ref{lemm+-}}\label{appendixA}

In this section, we give the proof of Lemma \ref{lemm+-}.
We focus on $m_+$, the case of $m_-$ being completely similar.
Recall that $m_+$ solves
\begin{align}
\label{eqm+}
\partial_x^2 m_+(x,k) + 2 ik  \partial_x m_+(x,k) = V(x) m_+(x,k).
\end{align}
It also solves the Volterra equation
\begin{equation}
\label{volterra1}
m_+(x,k) = 1 + \int_x^{+\infty} D_k(y-x)V(y) m_+(y,k)\, dy,
\end{equation}
where
\begin{align}
\label{Dk}
D_k(x) = \int_0^x e^{2ikz}\,dz = \frac{e^{2ikx}-1}{2ik}.
\end{align}
We will denote
$$
\partial_k m_+(x,k) = \dot{m}_+(x,k) \quad \mbox{and} \quad \partial_k^2 m_+(x,k) = \ddot{m}_+(x,k).
$$
By differentiating in $k$ the Volterra  equations  solved by $m_+$, we obtain immediately that
\begin{align}
\label{volterra2}& \dot m_+(x,k) =  \int_x^\infty D_k(y-x)V(y) \dot m_+(y,k)\, dy   +\int_{x}^{+\infty} \dot D_{k}(y-x) V(y) m_{+}(y,k) \, dy \\
\label{volterra3}& \ddot m_+(x,k) = \int_x^\infty D_k(y-x)V(y) \ddot m_+(y,k) dy +\int_{x}^{+\infty} \dot D_{k}(y-x) V(y) \dot m_{+}(y,k) \, dy \\
\nonumber & \hspace{2cm}+\int_{x}^{+\infty} \ddot D_{k}(t-x) V(y) m_{+}(y,k) \, dy
\end{align}

We first prove the existence of $m_{+}$ with the desired behavior at $+ \infty$ by solving the Volterra equation
 \eqref{volterra1} for $x \geq x_{0}$,  $x_{0}$ sufficiently large.
 More precisely, we can set
 $$z_{+}(x,k)= \langle k \rangle {m_{+}(x,k) - 1 \over \mathcal{W}_{+}^1(x)}$$
 and look for $z_{+}$ bounded solution on $[x_{0}, +\infty[$ of
 \begin{equation}
 \label{volterra1bis}
 z_{+}(x,k) - Lz_{+} = {\langle k \rangle \over \mathcal{W}_{+}^1(x) }  \int_{x}^{+ \infty} D_k(y-x)V(y)\, dy
 \end{equation}
 with
 $$
 Lz_{+}(x)=   {1 \over \mathcal{W}_{+}^1(x) } \int_x^\infty D_k(y-x)V(y) \mathcal{W}_{+}^1 (y)  z_{+}(y,k)\, dy.
  $$
  By using that uniformly in $x$ and $k$, we have
    $|D_{k}(z)|\lesssim { \langle x \rangle \over \langle k \rangle}$, we obtain that
 again uniformly in $k$,
 $$ \left \| {\langle k \rangle \over \mathcal{W}_{+}^1(x) }  \int_{x}^{+ \infty} D_k(y-x)V(y)\, dy \right\|_{L^\infty(x_{0}, +\infty)}
  \lesssim 1$$
  and
  $$ \|L z_{+}\|_{L^\infty(x_{0}, +\infty)} \lesssim   \| z_{+}\|_{L^\infty(x_{0}, +\infty)} \mathcal{W}^1_{+}(x_{0})$$
  therefore $Id -L$ is invertible on $L^\infty(x_{0}, +\infty)$ for $x_{0}$ sufficiently large
  and there exists
    a unique solution with $  \| z_{+}\|_{L^\infty(x_{0}, +\infty)} \lesssim 1.$
   This proves the existence of $m_{+}$ with the desired asymptotic behaviour on $[x_{0}, + \infty[$.
   Since $m_{+}$ solves a linear  ODE this completely determines $m_{+}$ on $\mathbb{R}.$
    To get the estimates for $x \leq x_{0}$, we can use the  Gronwall lemma.

    For $-1 \leq x \leq x_{0}$, we have from  \eqref{volterra1bis} that uniformly in $k$,
    $$ |z_{+}(x,k)| \lesssim   1  + \int_{x}^{x_{0}} \langle y \rangle  | V(y) |\,| z_{+}(y, k)|\, dy, \quad \forall x, \, -1 \leq x \leq x_{0}$$
     and hence we find $ |z_{+}(x,k) | \leq 1.$

     For $x \leq 0$,  we have again uniformly in $k$ that
  $$   {|z_{+}(x,k)| \over \langle x \rangle} \lesssim  1  +  { 1\over \langle x\rangle} \int_{x}^{0} \langle  x-y \rangle \langle y \rangle |V(y)| { |z_{+}(y, k)|
  \over \langle y \rangle}\, dy  \lesssim
  1  +   \int_{x}^{0}  \langle y \rangle |V(y)| {| z_{+}(y, k)| \over \langle y \rangle}\, dy
 $$
 and hence we find again by Gronwall that $z_{+} (x,k)/ \langle x \rangle$ is bounded.

  To estimate $\dot m_+(x,k)$ and $\ddot m_{+}(x,k)$, we proceed in the same way
   on the Volterra equations \eqref{volterra2}, \eqref{volterra3} by using that uniformly in $x$, $k$, we have
   $$ |\dot D_{k}(x) | \lesssim {\langle x\rangle^2 \over \langle k \rangle},  \quad
    |\ddot D_{k}(x) | \lesssim {\langle x\rangle^3 \over \langle k \rangle}.$$

Let us turn to the $x$ derivatives.
By taking the $x$ derivative in \eqref{volterra1}, we get that
\begin{equation}
\label{eqdxm+} \partial_{x}m_{+}(x,k) = -\int_{x}^{+ \infty} e^{2ik(x-y)} V(y) m_{+}(y,k)\, dy.
\end{equation}
 By using the estimate for $m_{+}$, we then find uniformly in $k$ that
 $$ |  \partial_{x}m_{+}(x,k) | \lesssim \int_{x}^{+\infty} |V(y)| \, dy \lesssim \mathcal{W}_{+}^0 (x)$$
  for $x \geq 0$ and that
  $$   |  \partial_{x}m_{+}(x,k) | \lesssim  1 +  \int_{x}^{0}  |y|\, |V(y)| \, dy \lesssim 1 $$
  for $x \leq 0$.

 The estimates for $\partial_{k}^s \partial_{x}m_{\pm}$ follow by differentiating in $k$
 the equation \eqref{eqdxm+}.

\medskip
\subsection{Basic multilinear estimates}
 Let us consider
$$
T_{\alpha}(f_{1}, f_{2}, f_{3})= \widehat{\mathcal{F}}^{-1} \iiint \widehat{\alpha}(k,\ell,m,n) \widehat{f}_{1}(\ell) \widehat{f}_{2}(m) \widehat{f}_{3}(n) \,d\ell\,dm\,dn.
$$
We will denote $(w,x,y,z)$ the dual variables of $(k,\ell,m,n)$. In other words,
$$
\widehat{\alpha}(k,\ell,m,n) = \frac{1}{(2\pi)^2} \iiiint e^{-i(wk+x\ell + ym + zn)} \alpha(w,x,y,z)\,dw \, dx \,dy\,dz
$$
We shall prove that
\begin{lem}
\label{lemmult}
 The operator $T_{\alpha}$
\begin{itemize}
\item maps   $L^\infty \times L^\infty \times L^\infty \to L^2$
 with norm bounded by $ \| \alpha(x,y,z,w)\|_{L^2_w L^1_{x,y,z}}$;

\item maps  $L^\infty \times L^\infty \times L^2 \to L^2$ with  norm bounded by
$\| \alpha(x,y,z,w)\|_{L^2_{w,x}  L^1_{y,z}}$.
\end{itemize}
\end{lem}

\begin{proof}
We observe that for every $g \in \mathcal{S}(\mathbb{R})$,
\begin{align*}
 \left( T_{\alpha} (f_{1}, f_{2}, f_{3}), g \right)_{L^2}
 & = \iiiint  \widehat{\alpha}(k,\ell,m,n) \widehat{f}_{1}(\ell) \widehat{f}_{2}(m) \widehat{f}_{3}(n) \overline{\widehat{g}(k)} \,dk \,d\ell\,dm\,dn
\\
& = \iiiint  \alpha (w,x,y,z) f_{1}(x) f_{2}(y) f_{3}(z) \overline{g (-w)} \,dw\, dx\, dy \,dz.
\end{align*}
Therefore, we easily get that
$$  \left| \left( T_{\alpha} (f_{1}, f_{2}, f_{3}), g \right)_{L^2}\right|
 \lesssim \|  \alpha \|_{L^2_w(L^1_{x,y,z})} \|f_{1}\|_{L^\infty} \| f_{2} \|_{L^\infty}
  \| f_{3} \|_{L^\infty} \|g\|_{L^2}$$
  and
  $$  \left| \left( T_{\alpha} (f_{1}, f_{2}, f_{3}), g \right)_{L^2}\right|
 \lesssim \| \alpha \|_{L^2_{w,x} (L^1_{y,z})} \|f_{1}\|_{L^\infty} \| f_{2} \|_{L^\infty}
  \| f_{3} \|_{L^2} \|g\|_{L^2},$$
 which, by duality, proves the desired result.
\end{proof}

Similarly, define
$$
U_\beta(f_1,f_2,f_3) = \widehat{\mathcal{F}}^{-1} \iiint  \widehat{\beta}(k,m,n) \widehat{f}_{1}(k-m-n) \widehat{f}_{2}(m) \widehat{f}_{3}(n) \,d\ell\,dm\,dn.
$$
\begin{lem} \label{lemmult2} If $1\leq p,q,r,s \leq \infty$ satisfy $\frac{1}{p} + \frac{1}{q} + \frac{1}{r} + \frac{1}{s} = 1$, the operator $U_\beta$ maps $L^p \times L^q  \times L^r \to L^{s'}$ with norm bounded by $\| \beta \|_{L^1}$.
\end{lem}
\begin{proof}
Simply notice that
$$
U_\beta(f_1,f_2,f_3) = \frac{1}{\sqrt{2\pi}} \int \beta(w-x,x-y,x-z) f_1(x) f_2(y) f_3(z)\,dx\,dy\,dz,
$$
and argue by duality.
\end{proof}

\end{document}